\documentclass[reqno,11pt]{amsart}

\usepackage{amscd}
\usepackage{amsmath}
\usepackage[new]{old-arrows}
\usepackage{amssymb}
\usepackage[american]{babel}
\usepackage{bbm}
\usepackage{bookmark}
\usepackage{cmap} 
\usepackage{dsfont}
\usepackage{enumerate}
\usepackage{epigraph}
\usepackage[mathscr]{euscript}
\usepackage[myheadings]{fullpage}
\usepackage{graphicx}
\usepackage[geometry]{ifsym}
\usepackage{mathabx}
\usepackage{accents}
\usepackage{mathrsfs}
\usepackage{mathtools}
\usepackage{enumitem}
\usepackage{refcount}
\usepackage{setspace}
\usepackage{stmaryrd}
\usepackage{thinsp}
\usepackage{verbatim}
\usepackage[all,2cell]{xy} \UseTwocells
\usepackage{xr}
\usepackage[multiple]{footmisc}
\usepackage{hyperref}

\usepackage{tikz}
\usetikzlibrary{decorations.pathmorphing}
\usetikzlibrary{cd}

\numberwithin{equation}{subsection}

\newtheorem{thm}{Theorem}[subsubsection]
\newtheorem*{thm*}{Theorem}

\newtheorem{cor}[thm]{Corollary}
\newtheorem*{cor*}{Corollary}
\newtheorem{lem}[thm]{Lemma}

\newtheorem{prop}[thm]{Proposition}
\newtheorem{prop-const}[thm]{Proposition-Construction}

\newtheorem{conjecture}{Conjecture}[subsection]
\newtheorem*{conjecture*}{Conjecture}

\newtheorem*{princ*}{Principle}

\theoremstyle{remark}
\newtheorem{rem}[thm]{Remark}
\newtheorem{example}[thm]{Example}

\newtheorem{defin}[thm]{Definition}

\newcounter{steps}[thm]
\newtheorem{step}[steps]{Step}

\newcommand{\into}{\hookrightarrow}
\newcommand{\onto}{\twoheadrightarrow}

\newcommand{\bA}{{\mathbb A}}
\newcommand{\bB}{{\mathbb B}}
\newcommand{\bC}{{\mathbb C}}
\newcommand{\bD}{{\mathbb D}}

\newcommand{\bG}{{\mathbb G}}

\newcommand{\bM}{{\mathbb M}}

\newcommand{\bP}{{\mathbb P}}

\newcommand{\bZ}{{\mathbb Z}}

\newcommand{\cC}{{\mathcal C}}
\newcommand{\cD}{{\mathcal D}}

\newcommand{\cF}{{\mathcal F}}
\newcommand{\cG}{{\mathcal G}}
\newcommand{\cH}{{\mathcal H}}

\newcommand{\cK}{{\mathcal K}}
\newcommand{\cL}{{\mathcal L}}

\newcommand{\cN}{{\mathcal N}}
\newcommand{\cO}{{\mathcal O}}

\newcommand{\cW}{{\mathcal W}}

\newcommand{\sC}{{\EuScript C}}
\newcommand{\sD}{{\EuScript D}}
\newcommand{\sE}{{\EuScript E}}

\newcommand{\sH}{{\EuScript H}}

\newcommand{\sL}{{\EuScript L}}

\newcommand{\sP}{{\EuScript P}}

\newcommand{\sX}{{\EuScript X}}
\newcommand{\sY}{{\EuScript Y}}
\newcommand{\sZ}{{\EuScript Z}}

\newcommand{\fb}{{\mathfrak b}}

\newcommand{\fg}{{\mathfrak g}}

\newcommand{\fn}{{\mathfrak n}}

\newcommand{\fp}{{\mathfrak p}}

\newcommand{\ft}{{\mathfrak t}}

\newcommand{\on}{\operatorname}

\renewcommand{\lim}{\on{lim}}

\newcommand{\LSU}{\on{LocSys}_G(U)}
\newcommand{\LSD}{\on{LocSys}_G(\overset{\circ}{D_x})}

\newcommand{\LSBA}{\on{LocSys}_{G^{\on{ab}}}(U,\chi_S^{\on{ab}})}
\newcommand{\LS}{\on{LocSys}_G(U,\chi_S, \tau)}
\newcommand{\LSN}{\on{LocSys}_G(U,\chi_S)}
\newcommand{\BunN}{\on{Bun}_{\check{G},\check{N}}(X,S)}

\newcommand{\NR}{\on{QCoh}((\on{LocSys}_G(\overset{\circ}{D})_{\chi_S}^{\on{RS}})_{\on{Ran},S})}

\newcommand{\LSsigma}{\on{LocSys}_G(U, \cO_S,\tau)}
\newcommand{\Autr}{D(\on{Bun}_{\check{G},\check{N}}(X,S))^{\check{T}_S, \chi_S}}

\newcommand{\Autreal}{\sH_{\check{G},\chi_S,\on{Ran},S}^{\on{geom}}}
\newcommand{\LSDisk}{\on{LocSys}_G(\overset{\circ}{D})}

\renewcommand{\subset}{\subseteq}

\newcommand{\dR}{\on{dR}}

\makeatletter
\newcommand{\biggg}{\bBigg@{4}}
\newcommand{\Biggg}{\bBigg@{5}}
\makeatother

\begin{document}

\frenchspacing

\setlength{\epigraphwidth}{0.4\textwidth}
\renewcommand{\epigraphsize}{\footnotesize}

\begin{abstract}

For a reductive group $G$, we prove that complex irreducible rigid $G$-local systems with quasi-unipotent monodromies and finite order abelianization on a smooth curve are motivic, generalizing a theorem of Katz for $GL_n$. We do so by showing that the Hecke eigensheaf corresponding to such a local system is itself motivic. Unlike other works in the subject, we work entirely over the complex numbers.

In the setting of de Rham geometric Langlands, we prove the existence of Hecke eigensheaves associated to any irreducible $G$-local system with regular singularities. We also provide a spectral decomposition of a naturally defined automorphic category over the stack of regular singular local systems with prescribed eigenvalues of the local monodromies at infinity. 

Finally, we establish a relationship between rigidity for complex local systems and automorphic rigidity, answering a conjecture of Yun in the tame complex setting.
\end{abstract}

\title[Motivic realization of rigid $G$-local systems on curves]{Motivic realization of rigid $G$-local systems on curves and\\
tamely ramified geometric Langlands}

\author{Joakim F\ae rgeman}

\address{Yale University, 
Department of Mathematics, 
KT 903, 219 Prospect Street, 
New Haven, CT 06520}

\email{joakim.faergeman@yale.edu}

\maketitle

\tableofcontents

\newpage

\section{Introduction}

\subsection{Motivic local systems}

\subsubsection{}\label{s:intro} 
Simpson conjectured that irreducible rigid local systems with quasi-unipotent local monodromies and finite order abelianization are motivic. In recent years, there has been progress on understanding the properties of rigid local systems (see e.g. \cite{esnault2018cohomologically}, \cite{klevdal2020g}), giving evidence to the conjecture. This progress is based on reduction to characteristic $p$, where Lafforgue's work on the Langlands conjectures is available. In this paper, we show that ideas in the geometric Langlands program in characteristic zero can be used in the study of rigid local systems.

\subsubsection{} In more detail, given a smooth quasi-projective complex variety $Y$, a natural problem is to identify the local systems of $Y$ that appear in the cohomology of a family of varieties over $Y$. We say that an irreducible complex local system $\sL$ is motivic if there exists a Zariski-open dense subset $U\subset Y$ such that $\sL_{\vert U}$ is a direct summand of $R^if_*(\bC_Z)$ for some $i\in \bZ$, where $f:Z\to U$ is a smooth projective family, and $\bC_Z$ is the constant sheaf on $Z$. 

When $\on{dim} Y=1$, Katz proved Simpson's conjecture by describing rigid local systems inductively (\cite[Thm. 5.2.1, 8.4.1]{katz2016rigid}). For higher-dimensional varieties, less is known (see e.g. \cite{corlette2008classification} and \cite{langer2018rank} that concern the rank two and rank three cases, respectively).

\subsubsection{} If $G$ is a connected reductive complex algebraic group, rigidity and motivicity make sense more generally for any $G$-local system on $Y$. Namely, we call a $G$-local system motivic if for any finite-dimensional representation of $G$, the corresponding local system on $Y$ is motivic in the above sense. As such, one may pose the similar question of whether rigidity implies motivicity in this case. Our main theorem asserts that this is true whenever $\on{dim} Y=1$:

\begin{thm}\label{t:main}

Let $Y$ be a smooth complex curve. Then any irreducible rigid $G$-local system on $Y$ with quasi-unipotent local monodromies at infinity and finite order abelianization is motivic.

\end{thm}

\subsection{Geometric Langlands}

Our proof makes use of the geometric Langlands correspondence. Let us therefore review the relevant background. 

Since the focus of this paper is on G-local systems, we describe spectral categories in terms of $G$ and geometric categories in terms of the Langlands dual group $\check{G}$. While this is in contrast with usual conventions in the literature, we hope it will not cause too much confusion.

\subsubsection{}  Henceforth, let $X$ denote a smooth projective complex curve. For a finite set of points $S=\lbrace x_1,...,x_n\rbrace \subset X$, we write $U=X-S$ for its complement. Our goal is to understand complex local systems on $U$. By the Riemann-Hilbert correspondence, we may work with de Rham local systems on $U$ instead; i.e., algebraic vector bundles on $U$ with a regular singular integrable connection.

\subsubsection{} Denote by $\check{G}$ the Langlands dual group of $G$, $\check{B}\subset \check{G}$ the dual Borel subgroup, and $\check{N}\subset \check{B}$ its unipotent radical. We write $\on{Bun}_{\check{G}}(X)$ for the moduli stack of $\check{G}$-bundles on $X$, and we write $\BunN$ for the stack parametrizing $\check{G}$-bundles on $X$ equipped with a reduction to $\check{N}$ on $S\subset X$.

\subsubsection{}

The tamely ramified geometric Langlands correspondence predicts that to any regular singular local system $\sigma$, there is a corresponding Hecke eigensheaf $\cF_{\sigma}$. Explicitly, $\cF_{\sigma}$ is a D-module on $\BunN$ such that that for any $V\in \on{Rep}(G)$, the corresponding Hecke functor
\[
H_V: D(\BunN) \to D(\BunN\times U)
\]

\noindent satisfies
\[
H_V(\cF_{\sigma})\simeq \cF_{\sigma}\boxtimes \sigma_V
\]

\noindent compatibly with the tensor product structure on $\on{Rep}(G)$ and composition of Hecke correspondences. Here, $\sigma_V=\sigma\overset{G}{\times} V$ is the local system on $U$ induced by $V$.

\subsubsection{} The key point\footnote{Implicit in \cite{yun2014rigidity}, for example.} is that whenever $\sigma$ is rigid with quasi-unipotent local monodromies and torsion abelianization, the corresponding Hecke eigensheaf $\cF_{\sigma}$ ought to be motivic in a suitable sense. More precisely, given a D-module $\cF\in D(\BunN)$, let us call $\cF$ motivic if it comes from a combination of sheaf-theoretic operations\footnote{I.e., $*,!-$push/pull, outer tensor product, and Verdier duality.} applied to the constant sheaf on a point. In §5, we show that if $\cF$ is an irreducibe local system, this definition of motivicity coincides with the one in §\ref{s:intro}. Since motivicity in this sense is easily seen to be preserved by Hecke functors, realizing a local system as the eigenvalue of a motivic Hecke eigensheaf forces the local system itself to be motivic (see Section \ref{s:overview} for details). As such, it suffices to prove the following:

\begin{thm}\label{t:Hecke}
Let $\sigma$ be an irreducible rigid local $G$-local system on $U$ with quasi-unipotent monodromies at infinity and finite order abelianization. Then there exists a motivic Hecke eigensheaf $\cF_{\sigma}$ with eigenvalue $\sigma$.

\end{thm}

\subsubsection{Motivic nature of the Langlands program} Arithmetically, over a global field, the Langlands conjectures predict that motives correspond to certain automorphic representations. However, once we pass to the complex numbers, the geometric Langlands correspondence seems to have very little to do with motives. One conceptual explanation is the categorical nature of geometric Langlands over the complex numbers where the tools employed often are of representation-theoretic flavor. 

To show that Hecke eigensheaves corresponding to irreducible rigid local systems are motivic, we show that they are realized as direct summands of a canonically defined Whittaker sheaf, which is motivic by its nature (see Remark \ref{r:can} for a precise statement). In particular, these Hecke eigensheaves admit a concrete description. We deduce this from the spectral decomposition theorem in §\ref{ss:sd} and the existence of Hecke eigensheaves associated to irreducible local systems.

\subsection{Spectral decomposition}\label{ss:sd}

The proof of Theorem \ref{t:Hecke} relies on spectrally decomposing a suitable automorphic category over the stack of regular singular local systems on $U$ with fixed quasi-unipotent local monodromies and fixed abelianization, which we describe next.

The subsequent results hold over an arbitrary algebraically closed field $k$ of characteristic zero.

\subsubsection{}

For each $x\in S$, pick an element $\chi_x\in \ft/X_{\bullet}(T)$. These define character sheaves on $\check{T}$ that we similarly denote by $\chi_x$. Since $\underset{x\in S}{\prod}\check{T}$ canonically acts on the stack $\BunN$, we may consider the category $D(\BunN)^{\check{T}_S,\chi_S}$ of $(\underset{x\in S}{\prod}\check{T},\underset{x\in S}{\prod}\chi_x)$-equivariant D-modules on $\BunN$.

\subsubsection{} Let $\on{LocSys}_G(\overset{\circ}{D}_x)^{\on{RS}}_{\chi_x}$ be the reduced stack parametrizing regular singular local systems on the punctured disk $\overset{\circ}{D}_x$ whose monodromy has eigenvalue $\chi_x$ (we refer to §\ref{s:lsrestr} for details). For example if $\chi_x=0$, then $\on{LocSys}_G(\overset{\circ}{D}_x)^{\on{RS}}_{\chi_x}\simeq \cN/G$, where $\cN$ denotes the nilpotent cone of $\fg$.

\subsubsection{} Denote by $\on{LocSys}_G(U)$ the stack of local systems on $U$ (see §\ref{s:actualls} for a precise definition). We let
\[
\on{LocSys}_G(U,\chi_S):=\on{LocSys}_G(U)\underset{\prod_{x\in S} \on{LocSys}_G(\overset{\circ}{D}_x)}{\times} \prod_{x\in S}\on{LocSys}_G(\overset{\circ}{D}_x)^{\on{RS}}_{\chi_x}
\]

\noindent be the moduli stack parametrizing regular singular local systems on $U$ whose monodromy at $x\in S$ has eigenvalue $\chi_x$. By construction, we have natural maps
\begin{equation}\label{eq:nat1}
\LSN\to \on{LocSys}_G(D_x)=\bB G, \:\: x\notin S,
\end{equation}
\begin{equation}\label{eq:nat2}
\LSN\to \on{LocSys}_G(\overset{\circ}{D}_x)^{\on{RS}}_{\chi_x}, \:\: x\in S.
\end{equation}

\subsubsection{} Recall that Bezrukavnikov's geometric realization of the affine Hecke algebra provides an equivalence of monoidal categories
\[
\on{IndCoh}(\fn/B\underset{\fg/G}{\times} \fn/B)\simeq D(\check{I}\backslash \check{G}(K)/\check{I})^{\on{ren}},
\]

\noindent where the right-hand side denotes the category of renormalized\footnote{So that the compact objects of $D(\check{I}\backslash \check{G}(K)/\check{I})^{\on{ren}}$ are exactly those that become compact after applying the forgetful functor $D(\check{I}\backslash \check{G}(K)/\check{I})^{\on{ren}}\to D(\check{G}(K)/\check{I})$. In particular, we have a fully faithful embedding $D(\check{I}\backslash \check{G}(K)/\check{I})\into D(\check{I}\backslash \check{G}(K)/\check{I})^{\on{ren}}$.} D-modules on $\check{I}\backslash \check{G}(K)/\check{I}$. Here, $\check{G}(K)$ is the loop group of $\check{G}$, and $\check{I}$ is its Iwahori subgroup. This induces a monoidal functor\footnote{Which we remind was already constructed in \cite{arkhipov2009perverse}.}
\[
\on{QCoh}(\on{LocSys}_G(\overset{\circ}{D}_x)^{\on{RS}}_0)\simeq \on{QCoh}(\cN/G)\to D(\check{I}\backslash \check{G}(K)/\check{I})
\]

\noindent by pull-push along the correspondence
\[
\cN/G\leftarrow \fn/B\to \fn/B\underset{\fg/G}{\times} \fn/B.
\]

\subsubsection{} We need a recent extension of the above result to the case of generalized eigenvalues due to Gurbir Dhillon, Yau Wing Li, Zhiwei Yun and Xinwen Zhu. More precisely, consider the stack
\[
\on{LocSys}_B(\overset{\circ}{D}_x)\underset{\on{LocSys}_T(\overset{\circ}{D}_x)}{\times} \lbrace \chi_x\rbrace/T
\]

\noindent of $B$-local systems on $\overset{\circ}{D}_x$ whose induced $T$-local system coincides with $\chi_x$.
\begin{thm}[Dhillon-Li-Yun-Zhu, \cite{dhillonendo}]\label{t:dyz}
There is a monoidal equivalence of categories
\[
\on{IndCoh}(\on{LocSys}_B(\overset{\circ}{D}_x)_{\chi_x}\underset{\on{LocSys}_G(\overset{\circ}{D}_x)}{\times}\on{LocSys}_B(\overset{\circ}{D}_x)_{\chi_x})\simeq D(\check{I},\chi_x\backslash \check{G}(K_x)/\check{I},\chi_x)^{\on{ren}}.
\]
\end{thm}

\noindent As above, one gets a monoidal functor
\begin{equation}\label{eq:monfun2}
\on{QCoh}(\on{LocSys}_G(\overset{\circ}{D}_x)^{\on{RS}}_{\chi_x})\to D(\check{I},\chi_x\backslash \check{G}(K_x)/\check{I},\chi_x)
\end{equation}

\noindent by pull-push along the correspondence
\[
\on{LocSys}_G(\overset{\circ}{D}_x)^{\on{RS}}_{\chi_x}\leftarrow \on{LocSys}_B(\overset{\circ}{D}_x)_{\chi_x}\to \on{LocSys}_B(\overset{\circ}{D}_x)_{\chi_x}\underset{\on{LocSys}_G(\overset{\circ}{D}_x)}{\times} \on{LocSys}_B(\overset{\circ}{D}_x)_{\chi_x}.
\]

\subsubsection{} Recall that in the unramified setting, Drinfeld-Gaitsgory \cite{gaitsgory2010generalized} construct a spectral decomposition of $D(\on{Bun}_{\check{G}}(X))$ over $\on{LocSys}_G(X)$:\footnote{See \cite[§16-17]{Cstterm}, where key inputs of the argument are recorded in detail.}
\[
\on{QCoh}(\on{LocSys}_G(X))\curvearrowright D(\on{Bun}_{\check{G}}(X)).
\]

\noindent We generalize Drinfeld-Gaitsgory's spectral decomposition to the tamely ramified setting. The main input needed is a forthcoming result of the author and E. Bogdanova (see Theorem \ref{t:tamesatake} below).
\begin{thm}[conditional on Theorem \ref{t:tamesatake}]\label{t:action}
We have an action
\[
\on{QCoh}(\on{LocSys}_G(U,\chi_S))\curvearrowright D(\BunN)^{\check{T}_S,\chi_S}.
\]

\noindent The action satisfies the following:
\begin{itemize}
    \item For $x\in X\setminus S$, the induced action
    \[
    \on{Rep}(G)\curvearrowright D(\BunN)^{\check{T}_S,\chi_S}
    \]
\noindent coming from (\ref{eq:nat1}) coincides with the usual Hecke action via geometric Satake.

\item For $x\in S$, the induced action
\[
\on{QCoh}(\on{LocSys}_G(\overset{\circ}{D}_x)^{\on{RS}}_{\chi_x})\curvearrowright D(\BunN)^{\check{T}_S,\chi_S}
\]

\noindent coming from (\ref{eq:nat2}) coincides with the Hecke action via Theorem \ref{t:action}.
\end{itemize}
\end{thm}

\begin{rem}\label{r:remupg}
Let
\[
\LSN':=\on{LocSys}_G(U)\underset{\underset{x\in S}{\prod} \on{LocSys}_G(\overset{\circ}{D}_x)}{\times} \underset{x\in S}{\prod}\on{LocSys}_B(\overset{\circ}{D}_x)_{\chi_x}.
\]

\noindent It is a formal consequence of Theorem \ref{t:action} and Theorem \ref{t:dyz} that we obtain a spectral decomposition of $D(\BunN)^{\check{T}_S,\chi_S}$ over $\LSN'$. While this is a stronger result (and more natural from the point of view of geometric Langlands), we choose to work with $\LSN$, which is more convenient when working with rigid local systems.
\end{rem}

\begin{rem}
Let $P_x\subset G$ be a parabolic for each $x\in S$. Write $I_{P_x}$ for the induced parahoric subgroup of $G(O_x)$ and $I_{\check{P}_x}\subset \check{G}(O_x)$ for the Langlands dual parahoric. Suppose we are given central elements $\chi_x\in Z(M_x)$, where $Z(M_x)$ denotes the center of the Levi subgroup of $P_x$, and let
\[
\on{LocSys}_{P_x}(\overset{\circ}{D}_x)_{\chi_x}=\on{LocSys}_{P_x}(\overset{\circ}{D}_x)\underset{\on{LocSys}_{M_x}(\overset{\circ}{D}_x)}{\times} \lbrace \chi_x\rbrace/M_x.
\]

\noindent Combining the results of \cite{dhillonchen}, \cite{dhillon2023localization}, one should obtain a parabolic version of Theorem \ref{t:dyz}. In this case, the same proof of Theorem \ref{t:action} works to give a spectral decomposition
\[
\on{QCoh}(\on{LocSys}_G(U)\underset{\underset{x\in S}{\prod} \on{LocSys}_G(\overset{\circ}{D}_x)}{\times} \underset{x\in S}{\prod} \on{LocSys}_{P_x}(\overset{\circ}{D}_x)_{\chi_x})\curvearrowright D(\on{Bun}_{\check{G},\check{N}_{P_S}}(X,S))^{\check{M}_S,\chi_S},
\]

\noindent where the right-hand side denotes the category of $(\underset{x\in S}{\prod}\check{M}_x,\underset{x\in S}{\prod}\chi_x)$-equivariant D-modules on the moduli of $\check{G}$-bundles on $X$ equipped with the data of a reduction to $\check{N}_{P_x}$, the unipotent radical of $\check{P}_x$, at each $x\in S$.
\end{rem}

\subsubsection{}
The main technical input needed in the proof of Theorem \ref{t:action} is a factorizable geometric Satake functor with tame ramification at the marked points, which we presently describe. Such a functor will appear in forthcoming work joint with Ekaterina Bogdanova.

Let $\NR$ be the unital factorization category over $\on{Ran}_{X_{\on{dR}},S}$, Ran's space with marked points at $S$, whose fiber at $x'\sqcup S\subset X(k)$ is
\[
\on{Rep}(G)\otimes \bigotimes_{x\in S} \on{QCoh}(\on{LocSys}_G(\overset{\circ}{D}_x)^{\on{RS}}_{\chi_x}).
\]

\noindent Similarly, let $\sH_{\check{G},\chi_S,\on{Ran},S}^{\on{geom}}$ be the unital factorization category over $\on{Ran}_{X_{\on{dR}},S}$ whose fiber at $x'\sqcup S\subset X(k)$ is
\[
D(\check{G}(O_{x'})\backslash \check{G}(K_{x'})/\check{G}(O_{x'}))\otimes \bigotimes_{x\in S} D(\check{I},\chi_x\backslash \check{G}(K_x)/\check{I},\chi_x).
\]

\noindent We need a factorizable geometric Satake functor in this setting. That is:
\begin{thm}[\cite{Satakefunctor}]\label{t:tamesatake}
There exists a unital factorizable monoidal functor
\[
\NR\to \Autreal,
\]

\noindent where the categories are equipped with their respective external convolution monoidal structures. The functor restricts to (\ref{eq:monfun2}) over the ramification points and the usual geometric Satake functor away from the ramification points. 
\end{thm}

\noindent Both categories are glued from their restrictions to (powers of) $U$ and $S$, and the construction of the above functor amounts to checking that the gluing maps are compatible in a precise sense.

\subsubsection{} The proof of Theorem \ref{t:action} mimics the unramified situation of Drinfeld-Gaitsgory \cite{gaitsgory2010generalized}. Namely, the idea of \cite{gaitsgory2010generalized} is that one has good control over D-modules on $\on{Bun}_{\check{G}}(X)$ obtained via:
\begin{itemize}
    \item Kac-Moody localization at critical level.

    \item Geometric Eisenstein series.
\end{itemize}

\noindent More precisely, with the notation of \emph{loc.cit}, the action of $\on{Rep}(G)_{\on{Ran}_X}$ on objects in $D(\on{Bun}_{\check{G}}(X))$ obtained in the above two ways factors through an action of $\on{QCoh}(\on{LocSys}_G(X))$. It turns out that one gets all D-modules on $\on{Bun}_{\check{G}}(X)$ in this way, which in turn gives the spectral decomposition.

Our proof in the tamely ramified setting is similar, except we show that it suffices to consider D-modules obtained by Kac-Moody localization only. In particular, we obtain a simplified proof of the spectral decomposition in the unramified setting that avoids Eisenstein series.

\begin{rem}
In \cite{ben2018betti}, Ben-Zvi and Nadler introduced a Betti version of the geometric Langlands conjecture. Here, automorphic D-modules are replaced by automorphic sheaves of $\bC$-vector spaces with nilpotent singular support. Providing a spectral decomposition of the automorphic category in this setting should be easier. However, the motivic sheaves one naturally encounters (e.g. the Whittaker sheaf and delta sheaves) tend not to have nilpotent singular support, and so we are forced to work in the de Rham setting for our purposes.
\end{rem}

\subsection{Existence of Hecke eigensheaves} 

\subsubsection{} Recall that our strategy for proving the motivicity of rigid local systems is to show that there exists a Hecke eigensheaf corresponding to such a rigid local system that is motivic in the sense of \S 1.2. Generalizing the construction in \cite{faergeman2022non}, we prove the existence of a Hecke eigensheaf associated to any irreducible local system on $U$ with regular singularities.

\begin{thm}\label{t:eigen}

Let $\sigma$ be an irreducible local system on $U$ with regular singularities. There exists a coherent Hecke eigensheaf $\cF_{\sigma}\in D(\BunN)$ with eigenvalue $\sigma$. Moreover, $\cF_{\sigma}$ can be chosen to be suitably Whittaker-normalized.

\end{thm}

\begin{rem}

In the unramified situation, it follows that any Whittaker-normalized Hecke eigensheaf $\cF_{\sigma}$ corresponding to an irreducible local system $\sigma$ on $X$ automatically satisfies:

\begin{itemize}
    \item $\cF_{\sigma}$ is regular holonomic (cf. \cite{arinkin2020stack}).

    \item $\cF_{\sigma}$ is cuspidal (cf. \cite{braverman2006deformations}).

    \item For generic $\sigma$, $\cF_{\sigma}$ is irreducible (on each connected component) and perverse (cf. \cite{faergeman2022non}).

\end{itemize}

\noindent We expect all of the above to hold in the tamely ramified situation as well, although we do not consider any of these problems in this paper.
 
\end{rem}

\subsubsection{} We construct the Hecke eigensheaf through Kac-Moody localization at critical level using opers, as developed by Beilinson-Drinfeld and then later advanced by Frenkel-Gaitsgory and Arinkin. In particular, the techniques are well-known. However, we highlight one step in which the methods differ: in the construction of Beilinson-Drinfeld, they prove their Hecke eigensheaf is non-zero by computing its characteristic cycle. We in turn prove our eigensheaf is non-zero by computing its Whittaker coefficient.

\subsubsection{} The construction uses the fact that passing from quasi-coherent sheaves on local opers to Kac-Moody modules at critical level via Feigin-Frenkel duality (see §\ref{s:4.1}) is compatible with the Hecke action coming from Theorem \ref{t:action}. When the opers have nilpotent residues, this is proven by Frenkel-Gaitsgory in \cite{frenkel2005fusion}. The general case of arbitrary residues is current work in progress by the author and Gurbir Dhillon (see Conjecture \ref{c:linearity}).

\subsubsection{}\label{s:autspecrigidity} Using the existence of Hecke eigensheaves, we outline how rigidity for local systems is related to Yun's notion of rigidity for geometric automorphic data \cite{yun2014rigidity}. More precisely, let us denote by a \emph{tamely ramified geometric automorphic datum} a tuple:
\begin{itemize}
    \item $\sL$ a rank 1 character sheaf on $\on{Bun}_{Z(\check{G}),1}(X,S)$, where the latter denotes the algebraic group stack of $Z(\check{G})$-bundles on $X$ equipped with a trivialization on $S\subset X$.

    \item A collection $\check{P}_S=\lbrace \check{P}_x\rbrace_{x\in S}$, where $\check{P}_x\subset \check{G}$ is a parabolic subgroup of $\check{G}$.

    \item A collection $\chi_S=\lbrace \chi_x\rbrace_{x\in S}$ of rank 1 character sheaves on $\check{P}_x$.
\end{itemize}

\noindent Denote by $I_{\check{P}_x}$ the induced parahoric given by the preimage of $\check{P}_x$ under $\check{G}(O_x)\to \check{G}$. We also denote by $\chi_x$ the induced character sheaf on $I_{\check{P}_x}$. Note that $\chi_x$ is pulled back from a character sheaf on $\check{M}_x^{\on{ab}}=\check{M}_x/[\check{M}_x,\check{M}_x]$, where $\check{M}_x$ denotes the Levi subgroup of $\check{P}_x$.

Let $\on{Bun}_{\check{G}}^{\infty\cdot S}$ denote the pro-stack of $\check{G}$-bundles on $X$ equipped with a trivialization on the formal completion of $S$ in $X$. As in \cite{yun2014rigidity}, we may consider the category
\[
D(\on{Bun}_G^{\infty\cdot S})^{(I_{\check{P}_S},\chi_S),\sL}
\]

\noindent of $(\underset{x\in S}{\prod} I_{\check{P}_x},\underset{x\in S}{\prod}\chi_x)$-equivariant D-modules on $\on{Bun}_G^{\infty\cdot S}$ that are  equivariant against $(\on{Bun}_{Z(\check{G}),1}(X,S),\sL)$.\footnote{We remark that while we do not have an action $\on{Bun}_{Z(\check{G}),1}(X,S)\curvearrowright \on{Bun}_{\check{G}}^{\infty\cdot S}$, we do have an action $D(\on{Bun}_{Z(\check{G}),1}(X,S))\curvearrowright D(\on{Bun}_{\check{G}}^{\infty\cdot S})^{I_{\check{P}_S},\chi_S}$.}

Let us say that a local system $\sigma$ on $U$ is \emph{afforded} by a tamely ramified geometric automorphic datum $(\sL, I_{\check{P}_S},\chi_S)$ if there exists a Hecke eigensheaf in $D(\on{Bun}_G^{\infty\cdot S})^{(I_{\check{P}_S},\chi_S),\sL}$ with eigenvalue $\sigma$. For the sake of simplicity, write $M_{\sigma}$ for the moduli stack of regular singular local systems on $U$ whose local monodromies at infinity and whose abelianization coincide with those of $\sigma$ (see §\ref{s:properdef} for a precise definition). Let $M_{\sigma}^{\on{irr}}\subset M_{\sigma}$ be the open substack consisting of irreducible local systems. 

We prove the following connection between automorphic and spectral rigidity:
\begin{thm}\label{t:autspecrigidity}
Suppose $\sigma$ is an irreducible local system on $U$ afforded by a tamely ramified geometric automorphic datum $(\sL, I_{\check{P}_S},\chi_S)$. If $\sigma$ is the only irreducible local system afforded by $(\sL, I_{\check{P}_S},\chi_S)$, then $M_{\sigma}^{\on{irr}}$ has a single $k$-point. Similarly, if the datum $ (\sL,I_{\check{P}_S},\chi_S)$ only affords finitely many irreducible local systems, then $M_{\sigma}^{\on{irr}}$ is a disjoint union of finitely many connected components, each of which has a single $k$-point.
\end{thm}

\begin{rem}
When $M_{\sigma}^{\on{irr}}$ consists of a single point, $\sigma$ is physically rigid, and when $M_{\sigma}^{\on{irr}}$ consists of finitely many points, $\sigma$ is cohomologically rigid (in the sense of §1.6 below).\footnote{This is slightly false as stated. Namely, one usually defines physical/cohomological rigidity in terms of the moduli stack $M_{\sigma}$, and not $M_{\sigma}^{\on{irr}}$ (as we do in §1.6 below). The fact that we are formulating Theorem \ref{t:autspecrigidity} in terms of $M_{\sigma}^{\on{irr}}$ reflects the fact that we can only prove the existence of Hecke eigensheaves associated to irreducible local systems.} Thus, the above theorem should be read as saying that physical (resp. cohomological)\footnote{In the sense that the geometric automorphic datum affords a single (resp. finitely many) isomorphism classes of irreducible local systems. We emphasize that this is an oversimplified definition of Yun's notion \cite[§2.7]{yun2014rigidity} of a strongly (resp. weakly) rigid geometric automorphic datum.} rigidity for geometric automorphic data implies physical (resp. cohomological) spectral rigidity, at least in the special case of a tamely ramified geometric automorphic datum.
\end{rem}

\subsection{Overview of the proof of the main theorem}\label{s:overview} Given the spectral decomposition and existence of Hecke eigensheaves described above, let us outline the proof that rigid local systems with quasi-unipotent local monodromies and torsion abelianization are motivic. 

Briefly, the idea is that we may spectrally decompose a suitable automorphic category over the moduli stack in which our rigid local system defines an isolated point. This allows us to extract direct summands of objects in the automorphic category which are Hecke eigensheaves with eigenvalue the rigid local system. Thus, it simply remains to find a motivic automorphic sheaf whose corresponding direct summand is non-zero.

\subsubsection{} For simplicity, assume that $G$ is adjoint. Let $\sigma$ be a rigid local system on $U$ with regular singularities, quasi-unipotent local monodromies at infinity and torsion abelianization. Let $\omega_X$ be the canonical sheaf on $X$ and consider the $\check{T}$-bundle $\rho(\omega_X(S))\in \on{Bun}_{\check{T}}(X)$.\footnote{As usual, $\rho$ denotes the cocharacter of $\check{T}$ given by the half-sum of all the positive roots of $\fg$.} For each $x\in S$, let $\chi_x\in \ft//\widetilde{W}^{\on{aff}}$ be the eigenvalue of the monodromy of $\sigma$ at $x$.

By definition of rigidity\footnote{We are using that $G$ is semisimple here. In general, one has to fix the abelianization of $\sigma$ in the moduli description.} (see §\ref{s:rigidls} below), we may find $\cO_x\in \on{LocSys}_G(\overset{\circ}{D}_x)^{\on{RS}}_{\chi_x}$ for each $x\in S$ such that $\sigma$ defines an isolated point in the moduli stack
\[
\on{LocSys}_G(U,\cO_S):=\on{LocSys}_G(U,\chi_S)\underset{\prod_{x\in S}\on{LocSys}_G(\overset{\circ}{D}_x)}{\times} \prod_{x\in S}\lbrace \cO_x\rbrace/\bB \on{Aut}(\cO_x).
\]

\noindent Here, $\on{Aut}(\cO_x)$ denotes the automorphism group of $\cO_x$ in $\on{LocSys}_G(\overset{\circ}{D}_x)$.

\subsubsection{} By the spectral decomposition theorem, we have an action
\[
\on{QCoh}(\on{LocSys}_G(U,\cO_S))\curvearrowright \on{Vect}\underset{\on{QCoh}(\underset{x\in S}{\prod}\on{LocSys}_G(\overset{\circ}{D}_x)^{\on{RS}}_{\chi_x})}{\otimes} D(\BunN)^{\check{T}_S,\chi_S},
\]

\noindent where the action $\on{QCoh}(\underset{x\in S}{\prod}\on{LocSys}_G(\overset{\circ}{D}_x)^{\on{RS}}_{\chi_x})\curvearrowright \on{Vect}$ is induced by the point
\[
\prod_{x\in S} \cO_x\in \prod_{x\in S}\on{LocSys}_G(\overset{\circ}{D}_x)^{\on{RS}}_{\chi_x}.
\]

\noindent Denote by $\cH$ the endofunctor
\[
D(\BunN)^{\check{T}_S,\chi_S}\to \on{Vect}\underset{\on{QCoh}(\underset{x\in S}{\prod}\on{LocSys}_G(\overset{\circ}{D}_x)^{\on{RS}}_{\chi_x})}{\otimes} D(\BunN)^{\check{T}_S,\chi_S}\to
\]
\[
\to D(\BunN)^{\check{T}_S,\chi_S}
\]

\noindent given by pull-push along $\underset{x\in S}{\prod} \cO_x\to \underset{x\in S}{\prod}\on{LocSys}_G(\overset{\circ}{D}_x)^{\on{RS}}_{\chi_x}$. Note that $\cH$ is simply given by convolving with an object of $\underset{x\in S}{\otimes} D(\check{I},\chi_x\backslash \check{G}(K_x)/\check{I},\chi_x)$ under Theorem \ref{t:dyz}.

Let $\cF_{\sigma}\in D(\BunN)^{\check{T}_S,\chi_S}$ be the Hecke eigensheaf with eigenvalue $\sigma$ guaranteed by Theorem \ref{t:eigen}, and let $\sP: \on{Spec}(\bC)\to \BunN$ be a $\bC$-point such that the $!$-fiber of $\cF_{\sigma}$ at $\sP$ does not vanish. Denote by $\delta_{\sP}=\sP_!(\bC)\in D(\BunN)$ the $!$-delta sheaf corresponding to $\sP$. Write $\on{Av}_!^{\check{T}_S,\chi_S}$ for the (partially defined) left adjoint to the forgetful functor
\[
D(\BunN)^{\check{T}_S,\chi_S}\to D(\BunN).
\]

\noindent Concretely, $\on{Av}_!^{\check{T}_S,\chi_S}(-)=a_!(\underset{x\in S}{\boxtimes}\chi_x \boxtimes -)$, where $a:\underset{x\in S}{\prod}\check{T}\times \BunN\to \BunN$ is the action map.

Since $\sigma$ defines an isolated point in $\on{LocSys}_G(U,\cO_S)$, it follows that $\cH(\on{Av}_!^{\check{T}_S,\chi_S}(\delta_{\sP}))$ admits a direct summand $\cG_{\sigma}$ that is a Hecke eigensheaf with eigenvalue $\sigma$. Moreover, $\cG_{\sigma}$ is readily seen to be motivic.\footnote{This is where we use the quasi-unipotence assumption on the local monodromies of $\sigma$.} Thus, if we can show that $\cG_{\sigma}$ is non-zero, this proves Theorem \ref{t:Hecke}. However, it follows by construction that
\[
\on{Hom}_{D(\BunN)^{\check{T}_S,\chi_S}}(\cG_{\sigma},\cF_{\sigma})=\on{Hom}_{D(\BunN)^{\check{T}_S,\chi_S}}(\cH(\on{Av}_!^{\check{T}_S,\chi_S}(\delta_{\sP})),\cF_{\sigma})\neq 0.
\]

\noindent This finishes the proof.

\subsubsection{} We highlight that the Hecke eigensheaf $\cG_{\sigma}$ is very concretely described. To reiterate:
\begin{enumerate}
    \item Take a delta sheaf $\delta_{\sP}\in D(\BunN)$ for which the $!$-fiber of $\cF_{\sigma}$ is non-zero.

    \item $!$-convolve with the sheaf $\underset{x\in S}{\boxtimes} \chi_x$ under the action of $\underset{x\in S}{\prod} \check{T}$ on $\BunN$.

    \item Apply the Hecke functor $\cH$ and take the corresponding direct summand.
\end{enumerate}

\begin{rem}\label{r:can}
One might object that the choice of $\sP\in \BunN$ was not canonical. We may alternatively argue as follows: since the eigensheaf $\cF_{\sigma}$ is Whittaker-normalized, we see that the Hecke eigensheaf with eigenvalue $\sigma$ occuring as a direct summand of the sheaf $\cH(\on{poinc}_{!,\chi_S})$ is non-zero. Here, $\on{poinc}_{!,\chi_S}$ is the Whitttaker sheaf in $D(\BunN)^{\check{T}_S,\chi_S}$ corepresenting the functor $\on{coeff}_{\chi_S}$ considered in §\ref{s:whittakercoeff} and is easily seen to be motivic. This gives an entirely canonical construction of a motivic Hecke eigensheaf corresponding to $\sigma$ (and does not depend on $\sigma$, only the moduli stack $\on{LocSys}_G(U,\cO_S)$).
\end{rem}

\subsection{Relation to other work}\label{s:relation}

\subsubsection{} Katz' proof that rigid $GL_n$-local systems with quasi-unipotent local monodromies are motivic relies on completing the complex local system $\rho: \pi_1(X)\to GL_n(V)$ to an $\ell$-adic local system and employing tools in the theory of $\ell$-adic perverse sheaves on the affine line. Relatedly, rigid complex local system can be defined over a number field which opens up the possibility of studying rigid local systems using methods outside complex geometry (see e.g. \cite{esnault2018cohomologically}). We emphasize the difference between the (pioneering) methods of Katz in \cite{katz2016rigid} and those used here. Namely, we consider the main contribution of this paper to be the demonstration of the usefulness of geometric Langlands in characteristic zero for the study of motives.

Let us expand on this point. The idea of studying rigid local systems by understanding their corresponding Hecke eigensheaves is far from new. In \cite{frenkel2009rigid}, inspired by a family of rigid\footnote{In the sense that they are determined by their local factors at the ramified points.} automorphic representations, Frenkel and Gross construct by explicit means for each simple complex algebraic group $G$ a rigid irregular complex connection on $\bP^1\setminus \lbrace 0,\infty\rbrace$. In \cite{heinloth2013kloosterman}, Heinloth-Ngô-Yun construct the $\ell$-adic counterpart to the connection of Frenkel-Gross as the eigenvalue of a corresponding Hecke eigensheaf (their construction makes sense over the complex numbers as well). The connection of \cite{heinloth2013kloosterman} has since been shown to coincide with that of Frenkel-Gross by Zhu in \cite{zhu2017frenkel}. This was again done by showing that the Hecke eigensheaves coincide.\footnote{Crucially, Frenkel-Gross show that their connection admits an oper structure which allows for the construction of its Hecke eigensheaf using the Beilinson-Drinfeld machinery.}

Building on the above, Yun \cite{yun2014rigidity} introduced the notion of a rigid geometric automorphic datum, which is designed to produce Hecke eigensheaves whose eigenvalues are (conjecturally) rigid. The upshot here is that it is often easier to construct a rigid local system with prescribed properties\footnote{Such as having a prescribed monodromy group.} by constructing the corresponding Hecke eigensheaf than it is to construct the local system directly.

Our paper differs from this story in a couple of ways. First, while the focus of the above papers is often on the construction of rigid local systems, we assume the existence of a rigid local system and prove its motivicity. Second, we employ techniques in the categorical geometric Langlands program over the complex numbers. The main idea of this paper, which is entirely categorical, is that the spectral decomposition realizes Hecke eigensheaves corresponding to rigid local systems as direct summands of motivic D-modules. As such, while the categorical geometric Langlands program (as opposed to its non-categorical versions)\footnote{Say the Langlands correspondence for function fields, which we remind has played a key role in the theory of rigid local systems, the work of Esnault-Groechenig \cite{esnault2018cohomologically} being a great such example.} has long seemed unfit to answer questions about motives, this paper provides an example of its use.

\subsubsection{}\label{s:differentrigids} Finally, we remark that in the literature one may encounter different notions of rigidity for $G$-local systems on curves. For a smooth projective curve $X$, a finite subset $S\subset X$ and an irreducible local system $\sigma$ on $U=X-S$, recall that $M_{\sigma}$ is the moduli stack of local systems on $U$ with prescribed local monodromies at infinity given by those of $\sigma$ and whose abelianization coincides with that of $\sigma$ as in §\ref{s:autspecrigidity}. There are a priori three types of rigidity for $\sigma$:
\begin{itemize}
    \item Physical rigidity: $M_{\sigma}$ has a single $k$-point. That is, there is only a single local system ($\sigma$ itself) on $U$ whose local monodromies and abelianization coincide with those of $\sigma$.

    \item Cohomological rigidity: $\sigma$ is a smooth isolated point in $M_{\sigma}$. That is, the tangent complex at $\sigma$ vanishes.

    \item Rigidity: $\sigma$ defines an isolated (potentially non-reduced) point in $M_{\sigma}$.
\end{itemize}

\noindent In general, one has tautological implications
\[
\on{physical}\: \on{rigidity}\Rightarrow \on{rigidity,}
\]
\[
\on{cohomological}\: \on{rigidity}\Rightarrow \on{rigidity.}
\]

\noindent In §\ref{s:rigidls}, we show that cohomological rigidity and rigidity coincide for irreducible local systems on curves.\footnote{This is not true in higher dimensions. See \cite{de2022rigid} for a counterexample.} That is, if $\sigma$ is an irreducible local system on $U$ defining an isolated point in $M_{\sigma}$, then it is automatically smooth. As such, we have
\[
\on{physical}\: \on{rigidity}\Rightarrow \on{rigidity}=\on{cohomological}\: \on{rigidity.}
\]

\noindent When $G=GL_n$, Katz shows that physical rigidity and cohomological rigidity coincide and that these can only occur in genus $\leq 1$. For a general reductive group $G$, it is not true that cohomological rigidity implies physical rigidity.\footnote{See \cite{kamgarpour2022arithmetic} for a counter example when $G=PGL_2$.}

\subsection{Structure of the paper}
\subsubsection{} In section 2, we review preliminaries on rigid local systems.

\subsubsection{} In section 3, we prove the spectral decomposition theorem.

\subsubsection{} In Section 4, we prove the existence of Hecke eigensheaves associated to local systems with regular singularities.

\subsubsection{} In section 5, we study the category of sheaves of geometric origin on an algebraic stack and prove the main theorem.

\subsection{Conventions}

\subsubsection{Base field} In Section 2,3 and 4, we work over an algebraically closed field $k$ of characteristic zero. In Section 5, we specialize to $k=\bC$.

\subsubsection{Categorical conventions} We use the language of higher category theory and higher algebra in the sense of \cite{lurie2009higher}, \cite{lurie2017higher}, \cite{gaitsgory2019study}. Throughout, by a (DG) category, we mean a $k$-linear presentable stable $(\infty,1)$-category. We denote by $\on{DGCat}_{\on{cont}}$ the category of DG-categories in which the morphisms are colimit-preserving functors.

For a category $\cC$ equipped with a t-structure, we denote by $\cC^{\heartsuit}$ the heart of the t-structure.

\subsubsection{Derived algebraic geometry} Our DAG conventions follow those of \cite{gaitsgory2019study}. In particular, we will freely use the language of prestacks in the sense of \emph{loc.cit}.

Throughout, all terminology is derived. For example, our schemes and stacks are assumed to be derived. Moreover, when we talk about the category of quasi-coherent sheaves, D-modules etc. on a prestack it will be in the derived sense.

\subsection{Acknowledgements} This paper owes its existence to my graduate advisor, Sam Raskin, who initially suggested a version of the strategy outlined in Section \ref{s:overview} to prove the motivicity of rigid local systems on curves. We moreover thank Sam for many helpful discussions and for carefully reading a draft of this paper.

We thank David Ben-Zvi, Ekaterina Bogdanova, Gurbir Dhillon, Hélène Esnault, Dennis Gaitsgory, Tom Gannon, Andreas Hayash, Masoud Kamgarpour, Daniel Litt, Tony Pantev, Andrew Salmon, Carlos Simpson and Zhiwei Yun for useful discussions.

\section{Basic notation} 

\subsection{Preliminaries} 

\subsubsection{Stacks} In this paper, all algebraic stacks will be locally almost of finite type and locally QCA in the sense of \cite{drinfeld2013some}. This ensures good functorial properties of D-modules. Henceforth, we will refer to these simply as algebraic stacks.

\subsubsection{D-modules} Let $\sY$ be a prestack locally almost of finite type. Following \cite{gaitsgory2017study}, we denote by $D(\sY)$ the DG category of $D$-modules on $\sY$.

\subsubsection{Pullback} Let $f:\sX\to \sY$ be a map of prestacks locally almost of finite type. We have the $!$-pullback functor $f^!: D(\sY)\to D(\sX)$. If its left adjoint is defined,\footnote{It is defined on holonomic D-modules, for example.} we denote it by $f_!$.

\subsubsection{Pushforward.} If $f$ is ind-representable, we have the $*$-pushforward functor $f_{*,\on{dR}}:D(\sX)\to D(\sY)$. Whenever its left adjoint is defined, we denote it by $f^{*,\on{dR}}$.

\subsubsection{Lie theory} Throughout, we let $G$ be a connected reductive complex algebraic group. We choose a torus and a Borel $T\subset B$, as well as an opposing Borel $B^{-}$ so that $B^{-}\cap B=T$. We let $N, N^{-}$ denote the unipotent radicals of $B,B^-$. Denote by $\fg=\on{Lie}(G), \fb=\on{Lie}(B), \ft=\on{Lie}(T)$ etc. We let $\Lambda, \check{\Lambda}$ (resp. $\Lambda^+, \check{\Lambda}^+$) be the lattice of integral weights and coweights of $G$ (resp. dominant weights and dominant coweights). Similarly, we denote by $X^{\bullet}(T), X_{\bullet}(T)$ the lattice of characters and cocharacters of $T$. Let $W$ be the finite Weyl group of $G$ and $\widetilde{W}^{\on{aff}}=X_{\bullet}(T)\rtimes W$ the extended affine Weyl group for $G$.

As in the introduction, we let $X$ be an irreducible smooth projective complex curve, and we fix a finite set of points $S=\lbrace x_1,...,x_n\rbrace\subset X$.

\subsubsection{Disks}\label{s:disk} Let $x_I: T\to X_{\on{dR}}^I$ be an $I$-tuple of maps to $X_{\dR}$ from an affine scheme $T$. That is, a map $T_{\on{red}}\to X^I$. Let $\Gamma_{x_I}$ be the closed reduced subscheme of $T_{\on{Red}}\times X$ given by the union of the graphs of $x_I$. We let
\[
\widehat{D_{x_I}}=(\Gamma_{x_I})_{\on{dR}}\underset{(T\times X)_{\on{dR}}}{\times} T\times X
\]

\noindent be the formal completion of $T\times X$ along $\Gamma_{x_I}$. Let $D_{x_I}$ be its affinization.\footnote{Formally, the affinization functor is the partially defined left adjoint functor to the inclusion of affine schemes into the category of prestacks.} We denote by $\overset{\circ}{D_{x_I}}:=D_{x_I}\setminus \Gamma_{x_I}$ the punctured disk at $x_I$.

Suppose $x\in X(k)$ is a $k$-point. Choosing a uniformizer at $x$, we have identifications
\[
\widehat{D_x}\simeq \on{Spf}(k[[t]]),\: D_x\simeq \on{Spec}(k[[t]]), \: \overset{\circ}{D_x}\simeq \on{Spec}(k((t))).
\]

\noindent In general, we write $O_x$ for the ring of functions on $D_x$ and $K_x$ for the ring of functions on $\overset{\circ}{D}_x$.

\subsection{Local systems on the punctured disk}\label{s:ls} 

\subsubsection{} As in \cite{raskin2015notion}, we let $\on{LocSys}_G(\overset{\circ}{D})$ denote the prestack of $G$-local systems on the formal punctured disk:
\[
\on{LocSys}_G(\overset{\circ}{D_x})=\fg((t))dt/G(K_x).
\]

\noindent Here, $G(K_x)$ denotes the loop group of $G$ at $x$. The above quotient is taken with respect to the gauge action of $G(K_x)$ on $\fg((t))dt:=\fg\otimes\Omega^1_{\overset{\circ}{D_x}}$, and then we étale sheafify to obtain $\LSDisk$. Recall that the gauge action is defined as follows:
\[
G(K)\times \fg((t))dt\to \fg((t))dt, \;\; (g,\Gamma dt)\mapsto \on{Ad}_g(\Gamma)dt-(dg)g^{-1}.
\]

\noindent Later, we will need to consider $\on{LocSys}_G(\overset{\circ}{D_x})$ as a mapping prestack, see Appendix \ref{s:fin}.

\begin{rem}
The term $(dg)g^{-1}$ does not make sense unless $G$ is the general linear group. Properly, we should replace $(dg)g^{-1}$ with the pullback of the Cartan form on $G$ under $g:\overset{\circ}{D}\xrightarrow{g} G$, see \cite[§1.12]{raskin2015notion}.
\end{rem}

\subsection{Local systems on the punctured disk with restricted variation}\label{s:lsrestr}

We want a candidate for the stack of regular singular de Rham local systems on the punctured disk whose monodromy has some fixed eigenvalue. When the eigenvalue is $1$ (i.e., when the local systems are unipotent), we may simply take the stack $\cN/G$, where $\cN\subset \fg$ is the nilpotent cone.

When the eigenvalue is not $1$, giving an explicit description is more complicated. There are various ways to approach this problem (see Remark \ref{r:alt} below), but it turns out that a very convenient way is via the stack of local systems on the punctured disk with restricted variation. This stack was essentially introduced in \cite{arinkin2020stack}. More precisely, \cite{arinkin2020stack} studies stacks associated to \emph{gentle Tannakian categories}, the latter being a notion defined in \emph{loc.cit}. In the forthcoming paper \cite{lsrestr} of Bogdanova, it is explained how the stack of local systems with restricted variation on the punctured disk, to be denoted $\on{LocSys}^{\on{restr}}_G(\overset{\circ}{D})$, fits this framework. 

Below, following \cite{lsrestr}, we give the definition of $\on{LocSys}^{\on{restr}}_G(\overset{\circ}{D})$, referring to \emph{loc.cit} for its basic properties.

\subsubsection{} Denote by $\on{Lisse}(\overset{\circ}{D})^{\heartsuit}$ the abelian category of pairs $(V,\nabla)$, where $V$ is a finite-dimensional $K=k((t))$-vector space, and $\nabla$ is a $k$-linear map
\[
\nabla: V\to Vdt
\]

\noindent satisfying
\[
\nabla(fv)=f\nabla(v)+df\otimes v
\]

\noindent for all $f\in K, v\in V$. We let $\on{QLisse}(\overset{\circ}{D})=\on{Ind}(\on{Lisse}(\overset{\circ}{D}))$, where $\on{Lisse}(\overset{\circ}{D})$ denotes the bounded derived category of $\on{Lisse}(\overset{\circ}{D})^{\heartsuit}$.

\subsubsection{} From the Tannakian description of $\on{LocSys}_G(\overset{\circ}{D})$, one obtains a map cf. \cite{lsrestr}:
\[
\on{LocSys}_G(\overset{\circ}{D})^{\on{restr}}\to \on{LocSys}_G(\overset{\circ}{D}).
\]

\noindent The map is a bijection on $k$-points.

\subsubsection{} Recall that a local systems $\sigma\in \on{LocSys}_G(\overset{\circ}{D})(k)$ is semi-simple if whenever it admits a reduction to a parabolic $P$, it further admits a reduction to the Levi subgroup $M$ for some splitting $M\into P$.

Given two local systems $\sigma_1,\sigma_2\in \on{LocSys}_G(\overset{\circ}{D})^{\on{restr}}(k)$ with $\sigma$ semi-simple, we say that $\sigma_1$ is a semi-simplification of $\sigma_2$ if there exists a parabolic $P\subset G$ and reductions $\sigma_{1,P}, \sigma_{2,P}$ of $\sigma_1,\sigma_2$, respectively, such the induced $M$-local systems of the reductions coincide.

\subsubsection{Topological properties}
We have the following description of the connected components of $\on{LocSys}_G(\overset{\circ}{D})^{\on{restr}}$:
\begin{prop}[\cite{arinkin2020stack}, Proposition 3.7.2+Theorem 1.8.3]\label{p:agkrrv}
There is a bijection between the connected components of $\on{LocSys}_G(\overset{\circ}{D})^{\on{restr}}$ and the set of isomorphism classes of semi-simple $G$-local systems on $\overset{\circ}{D}$. The bijection is uniquely characterized by the fact that two local systems $\sigma_1,\sigma_2\in \on{LocSys}_G(\overset{\circ}{D})^{\on{restr}}(k)$ lie in the same connected component if and only if they have the same semi-simplification.

Moreover, the underlying reduced prestack of a connected component is an algbraic stack of finite type.
\end{prop}

\subsubsection{} For each $\lambda\in \ft$, we may consider the element $\lambda \frac{dt}{t}\in \on{LocSys}_T(\overset{\circ}{D})^{\on{restr}}(k)$. By abuse of notation, we similarly denote this element by $\lambda$. Observe that the composition
\begin{equation}\label{eq:TtoG}
\ft\xrightarrow{\frac{dt}{t}} \on{LocSys}_T(\overset{\circ}{D})^{\on{restr}}\to \on{LocSys}_G(\overset{\circ}{D})^{\on{restr}}
\end{equation}

\noindent factors through $\ft//\widetilde{W}^{\on{aff}}$.

\subsubsection{} For $\chi\in\ft//\widetilde{W}^{\on{aff}}$, let $\on{LocSys}_G(\overset{\circ}{D})^{\on{RS}}_{\chi}$ be the underlying \emph{reduced} prestack of the connected component of $\on{LocSys}_G(\overset{\circ}{D})^{\on{restr}}$ containing the image of $\chi$ under (\ref{eq:TtoG}). By Proposition \ref{p:agkrrv}, $\on{LocSys}_G(\overset{\circ}{D})^{\on{RS}}_{\chi}$ is an algebraic stack of finite type. Moreover, the proposition shows that the map that associates to a $k$-point of $\ft//\widetilde{W}^{\on{aff}}$ its corresponding connected component of $\on{LocSys}_G(\overset{\circ}{D})^{\on{restr}}$ is an injection of sets.
\begin{rem}
The superscript "RS" stands for "regular singular", "restricted" and "residue".
\end{rem}

\subsubsection{Alternative description}\label{s:altdesc} For $\lambda\in\ft$, consider the prestack
\[
\on{LocSys}_B(\overset{\circ}{D})_{\lambda}^{\on{restr}}:=\on{LocSys}_B(\overset{\circ}{D})^{\on{restr}}\underset{\on{LocSys}_T(\overset{\circ}{D})^{\on{restr}}}{\times} \lambda/T.
\]

\noindent Note that $\on{LocSys}_B(\overset{\circ}{D})_{\lambda}^{\on{restr}}$ does not depend on $\lambda$ up to integral translates by $X_{\bullet}(T)$. Moreover, $\on{LocSys}_B(\overset{\circ}{D})_{\lambda}^{\on{restr}}$ is irreducible as an algebraic stack of finite type. By (the proof of) \cite[Prop. 4.3.3]{arinkin2020stack}, the natural map
\[
\on{LocSys}_B(\overset{\circ}{D})_{\lambda}^{\on{restr}}\to \on{LocSys}_B(\overset{\circ}{D})_{\lambda}:=\on{LocSys}_B(\overset{\circ}{D})\underset{\on{LocSys}_T(\overset{\circ}{D})}{\times} \lambda/T
\]

\noindent is an isomorphism. In particular, the map
\[
\on{LocSys}_B(\overset{\circ}{D})_{\lambda}\to \on{LocSys}_G(\overset{\circ}{D})
\]

\noindent factors through $\on{LocSys}_G(\overset{\circ}{D})^{\on{restr}}\to \on{LocSys}_G(\overset{\circ}{D})$.

\subsubsection{} By \cite[§3]{arinkin2020stack}, the map
\[
\on{LocSys}_B(\overset{\circ}{D})_{\lambda}\to \on{LocSys}_G(\overset{\circ}{D})^{\on{restr}}
\]

\noindent is proper. Moreover, by the proof of Proposition 3.7.2 in \emph{loc.cit}, it follows that $\on{LocSys}_G(\overset{\circ}{D})^{\on{RS}}_{\chi}$ is given by the image of
\[
\on{LocSys}_B(\overset{\circ}{D})_{\lambda}\to \on{LocSys}_G(\overset{\circ}{D})^{\on{restr}},
\]

\noindent where $\lambda\in \ft/X_{\bullet}(T)$ is any lift of $\chi$.

\begin{example}
If $\chi=0$, we have
\[
\on{LocSys}_B(\overset{\circ}{D})_{0}^{\on{restr}}\simeq\fn/B,\;\; \on{LocSys}_G(\overset{\circ}{D})^{\on{RS}}_{0}\simeq\cN/G.
\]

\end{example}

\begin{rem}\label{r:alt}
If one insisted on avoiding local systems with restricted variation and only working with $\on{LocSys}_G(\overset{\circ}{D})$, we could define $\on{LocSys}_G(\overset{\circ}{D})^{\on{RS}}_{\chi}$ as follows. Let $\lambda\in \ft$ be an \emph{anti-dominant} representative of $\chi$. One can show that the natural map
\[
\lambda\frac{dt}{t}+t^{-1}\fn[[t]]dt/N(O)\to \lambda\frac{dt}{t}+\fn((t))/N(K)=\on{LocSys}_B(\overset{\circ}{D})\underset{\on{LocSys}_T(\overset{\circ}{D})}{\times} \lambda/T
\]

\noindent is an isomorphism in this case.\footnote{If $\lambda$ is moreover dominant, then these in turn coincide with $(\fn+\lambda)/N$.} Denote by $\kappa: \fg\to \ft//W$ the Harish-Chandra map. Then $\on{LocSys}_G(\overset{\circ}{D})^{\on{RS}}_{\chi}$ coincides with the underlying reduced stack of the connected component of $\kappa^{-1}(\lambda)\frac{dt}{t}+\fg[[t]]dt/G(O)$ containing $\lambda\frac{dt}{t}$. 
\end{rem}

\subsubsection{} The following lemma gives a criterion for logarithmic connections on the punctured disk to factor through $\on{LocSys}_G(\overset{\circ}{D})^{\on{RS}}_{\chi}$.
\begin{lem}\label{l:cstresidue}
Let $S=\on{Spec}(A)$ be an affine scheme, and let $(\sP,\nabla)$ be a $G$-bundle on $S\widehat{\times} D=\on{Spf}(A[[t]])$ equipped with a relative connection along $S$ with pole of order at most $1$. Taking the residue, we obtain a map
\begin{equation}\label{eq:resnonordered}
S\to \fg/G\to \ft//W\to \ft//\widetilde{W}^{\on{aff}}.
\end{equation}

\noindent If the map (\ref{eq:resnonordered}) factors through $\chi\in \ft//\widetilde{W}^{\on{aff}}$, then the corresponding map $S\to \on{LocSys}_G(\overset{\circ}{D})$ factors through $\on{LocSys}_G(\overset{\circ}{D})^{\on{RS}}_{\chi}$.
\end{lem}

\begin{proof}

\emph{Step 1.}
We may assume $S$ is connected. Locally trivializing the $G$-bundle, we may further assume the bundle is trivial. As such, let
\[
\nabla=d+\Gamma dt
\]

\noindent be the corresponding connection given by $S\to \on{LocSys}_G(\overset{\circ}{D})$, where
\[
\Gamma=t^{-1}\Gamma_{-1}+\Gamma_0+t\Gamma_1+\cdot\cdot\cdot \in (A\otimes \fg)((t)).
\]

\noindent By §\ref{s:altdesc}, it suffices to provide a lift to a map $S\to \on{LocSys}_B(\overset{\circ}{D})_{\lambda}$ for some lift $\lambda\in\ft$ of $\chi$. That is, we need to gauge $\nabla$ by an element of $g\in G(O)(S)=\on{Hom}(\on{Spf}(A[[t]]),G)$ so that all its terms are elements of $(A\otimes \fb)((t))$ with residue $\lambda$.

Conjugating by an element of $g_0\in G$, we may assume that $\Gamma_{-1}\in A\otimes \fb$. That is, $\Gamma_{-1}$ defines a map $S\to \fb$. By assumption, its further projection to $\ft//\widetilde{W}^{\on{aff}}$ is constant. Since $S$ is connected and $\widetilde{W}^{\on{aff}}$ is discrete, this forces the map $S\to\fb\to \ft$ to be constant. In other words, the element $\Gamma_{-1}$ defines an element of $\lambda+A\otimes \fn\subset A\otimes \fb$.

\emph{Step 2.}
Denote by
\[
\cK_i(S)=\on{Ker}(\on{Hom}(\on{Spf}(A[[t]]),G)\to \on{Hom}(\on{Spec}(A[[t]]/(t^i)),G))\subset G(O)(S)
\]

\noindent the $i$'th congruence subgroup.
We will create the element $g$ by induction. Namely, we will create elements $g_i\in\cK_i(S)$ such that the first $i+1$ terms of the connection
\[
\on{Gauge}{}_{g_ig_{i-1}\cdot\cdot\cdot g_1}(\nabla)
\]

\noindent lie in $A\otimes \fb$. Then $g=\underset{i\geq 1}{\prod} g_i$.

We will let $g_1=\on{exp}(tB_1)\in\cK_1(S)$ for some $B_1\in\on{Lie}(\cK_1(S))=t(A\otimes \fg)[[t]]$.\footnote{Recall that $\cK_i$ is pro-unipotent so that the exponential map is well-defined.} Note that
\[
\on{Gauge}_{\on{exp}(tB_1)}(d+\Gamma dt)=\on{Ad}_{\on{exp}(tB_1)}(d+\Gamma dt)- (d\on{exp}(tB_1))\cdot \on{exp}(tB_1)^{-1}
\]
\[
=\big(\Gamma dt + [tB_1,\Gamma dt]+t(\cdot\cdot\cdot)\big) -\big(B_1+t(\cdot\cdot\cdot)\big).
\]

\noindent In particular, the degree zero term of the above expression is given by 
\begin{equation}\label{eq:mustlieinb}
\Gamma_0+[B_1,\Gamma_{-1}]-B_1.
\end{equation}

\noindent By potentially choosing a different $g_0\in G$, we may assume that $\lambda$ is anti-dominant. By Lemma \ref{l:techlemma} below, we can find $B_1\in A\otimes \fg$ so that (\ref{eq:mustlieinb}) lies in $A\otimes \fb$.

In general, assume that we have found $g_1,...,g_{i-1}$ with the desired properties. We let $g_i=\on{exp}(t^iB_i)$ for a suitable $B_i\in A\otimes \fg$. Namely, the term in front of $t^{i-1}$ in the connection
\[
\on{Gauge}_{g_ig_{i-1}\cdot\cdot\cdot g_1}(d+\Gamma dt)
\]

\noindent is given by 
\[
\Gamma_{i-1}+[B_i,\Gamma_{-1}]-iB_i.
\]

\noindent Applying Lemma \ref{l:techlemma} again, we may choose $B_i$ so that the above expression lies in $A\otimes \fb$.

\end{proof}

\noindent We used the following lemma in the proof above.
\begin{lem}\label{l:techlemma}
Let $\Gamma_{-1}\in \lambda+A\otimes \fn$ with $\lambda\in\ft(k)$ anti-dominant. For $i\geq 1$, let $\Gamma_{i-1}\in A\otimes \fg$. Then there exists $B_i\in A\otimes \fg$ such that
\begin{equation}\label{eq:thingyouwantinb}
\Gamma_{i-1}+[B_i,\Gamma_{-1}]-iB_i\in A\otimes \fb.
\end{equation}
\end{lem}

\begin{proof}
For every positive root $\alpha$, let $x_{-\alpha}\in A\otimes\fn^{-}_{-\alpha}$ be the summand of $\Gamma_{i-1}$ corresponding to $-\alpha$. We will choose suitable $y_{-\alpha}\in A\otimes\fn^{-}_{-\alpha}$, and then let $B_i=\underset{\alpha}{\sum} y_{-\alpha}\in A\otimes \fn^{-}$. To find $y_{-\alpha}$, we induct on the length of $\alpha$.

Suppose first that $\alpha$ is the longest root. In this case, the summand of
\[
\Gamma_{i-1}+[y_{-\alpha},\Gamma_{-1}]-iy_{-\alpha}
\]

\noindent corresponding to the root $-\alpha$ is given by 
\[
x_{-\alpha}+(\alpha(\lambda)-i)y_{-\alpha}.
\]

\noindent Since $\lambda$ is anti-dominant, the element $\alpha(\lambda)-i$ is a non-zero element of $k$, in particular a unit of $A$. Thus, we take
\[
y_{-\alpha}=-(\alpha(\lambda)-i)^{-1}x_{-\alpha}.
\]

\noindent For an arbitrary root $\alpha$, the summand of (\ref{eq:thingyouwantinb}) corresponding to $-\alpha$ will be
\[
x_{-\alpha}+\gamma_{>-\alpha}+(\alpha(\lambda)-i)y_{-\alpha},
\]

\noindent where $\gamma_{>-\alpha}$ depends on $x_{-\beta}$ with $\beta$ being a strictly shorter root than $\alpha$. Then we take
\[
y_{-\alpha}=-(\alpha(\lambda)-i)^{-1}(x_{-\alpha}+\gamma_{>-\alpha}).
\]

\end{proof}

\subsubsection{} We end this subsection with the following lemma:
\begin{lem}\label{l:finman}
The stack $\on{LocSys}_G(\overset{\circ}{D})^{\on{RS}}_{\chi}$ has finitely many $k$-points.
\end{lem}
\begin{proof}
By \cite[Prop. 15.35]{deligne1989groupe}, restriction along
\[
\overset{\circ}{D}\to \bG_m
\]

\noindent induces an equivalence of regular singular connections. The stack $\on{LocSys}_G(\overset{\circ}{D})^{\on{RS}}_{\chi}$ may be defined over a subfield of the complex numbers. As such, we may assume that $k=\bC$. Applying the Riemann-Hilbert correspondence, we need to show that the fibers of
\[
\on{LocSys}_G(S^1)=G/G\to G//G\simeq T//W
\]

\noindent have finitely many $\bC$-points. However, this is clear: there exists some semisimple element $s\in G$ such that any element in the fiber has the form $us$ for some unipotent $u\in G$.
\end{proof}

\subsection{Local systems on the punctured curve}

\subsubsection{} Let $X$ be a smooth projective curve, and let $S\subset X(k)$ be a finite set of $k$-points. Let $U=X-S$ be the complement. We denote by $\bB G$ the classifying stack of $G$-bundles that are étale locally trivial.

\subsubsection{} For prestacks $Y,Z$, we write $\on{Maps}(Y,Z)$ for the mapping prestack between $Y$ and $Z$. For an affine test scheme $T$, we have
\[
\on{Maps}(Y,Z)(T):=\on{Maps}_{\on{PreStk}}(Y\times T, Z).
\]

\subsubsection{}\label{s:lsUnaive} Let
\[
\on{LocSys}_G(U)^{\on{naive}}:=\on{Maps}(U_{\on{dR}}, \bB G).
\]

\noindent This is the prestack parametrizing $G$-bundles on $U$ equipped with an integrable connection. This space is too big in general; instead, we want to require that the underlying $G$-bundles on $U$ extend to $X$. We explain how to do that next.

\subsubsection{} Let $\on{Bun}_G(X)^{\infty\cdot S}$ be the pro-stack whose $T$-points are given by pairs $(\sP,\phi)$, where $\sP$ is a $G$-bundle on $X\times S$, and $\phi$ is a trivialization of $\sP$ on the formal completion of $T\times X$ along $T\times S$. We have a natural action of $G(K_S):=\underset{x\in S}{\prod} G(K_x)$
on $\on{Bun}_G(X)^{\infty\cdot S}$ by changing the trivialization on the punctured disk around each $x\in S$. Consider the quotient
\[
\overset{\circ}{\on{Bun}}_G:=\on{Bun}_G(X)^{\infty\cdot S}/G(K_S)
\]

\noindent obtained by first taking the prestack quotient, and then étale sheafifying. Concretely, the $T$-points of $\overset{\circ}{\on{Bun}}_G$ are given by $G$-bundles $\sP$ on $U\times T$ such that there exists an étale cover $T'\to T$ so that the pullback of $\sP$ to $U\times T'$ extends to $X\times T'$.
\begin{lem}\label{l:drinfsimp}
Let $\sP$ be the $G$-bundle on $T\times U$ corresponding to a map 
$T\to \overset{\circ}{\on{Bun}}_G$. Then there exists an étale cover $T''\to T$ so that the pullback to $T''\times U$ becomes Zariski-locally trivial on $U$.
\end{lem}

\begin{proof}
Let $T'\to T$ be an étale cover such that the pullback of $\sP$ admits an extension to $T'\times X$. By \cite[Thm. 2]{drinfeld1995b}, we may find a further étale cover $T''\to T'$ such that the pullback becomes Zariski-locally trivial on $X$.
\end{proof}

\subsubsection{} We have a natural map
\[
\overset{\circ}{\on{Bun}}_G\to\on{Bun}_G(U),
\]

\noindent which is a monomorphism. That is, for every affine scheme $T$, the induced map of groupoids 
\[
\overset{\circ}{\on{Bun}}_G(T)\to\on{Bun}_G(U)(T)
\]

\noindent is fully faithful.

\subsubsection{}\label{s:actualls} Consider the prestack
\[
\on{LocSys}_G(U):=\on{LocSys}_G(U)^{\on{naive}}\underset{\on{Bun}_G(U)}{\times}\overset{\circ}{\on{Bun}}_G.
\]

\noindent We similarly have a monomorphism
\[
\on{LocSys}_G(U)\to \on{LocSys}_G(U)^{\on{naive}}.
\]

\subsection{Restricting local systems}

\subsubsection{} Let $x\in S$. We explain how to define a restriction map
\begin{equation}\label{eq:restric}
\on{LocSys}_G(U)\to \on{LocSys}_G(\overset{\circ}{D_x}).
\end{equation}

\noindent The construction uses the "fake" de Rham punctured disk appearing in Appendix \ref{app:A}. The reader is invited to skip this subsection on a first read and take the map (\ref{eq:restric}) on faith.

\subsubsection{} Consider the prestack $\on{LocSys}_G(\overset{\circ}{D_x})^{\on{naive}}$ defined by étale sheafifying the functor that associated to an affine test scheme $T$ the groupoid
\[
\on{Maps}_{\on{PreStk}}((\overset{\circ}{D_x})^{\on{fake}}_{\on{dR}},\bB G).
\]

\noindent Here, $(\overset{\circ}{D_x})^{\on{fake}}_{\on{dR}}$ denotes the fake de Rham punctured for the map $T\onto \lbrace x\rbrace\into X$. The definition is due to Gaitsgory and is recalled in Appendix \ref{app:A}.

\subsubsection{} By Lemma \ref{l:asheaf}, we may alternatively define $\on{LocSys}_G(\overset{\circ}{D_x})$ as the étale sheafification of the functor that associates to an affine test scheme $T$ the groupoid
\[
\on{Maps}_{\on{PreStk}}((\overset{\circ}{D_x})^{\on{fake}}_{\on{dR}}, \on{pt}/G).
\]

\noindent Here, $\on{pt}/G$ denotes the \emph{prestack} quotient of $\on{pt}$ by $G$ (that is, we do not étale sheafify the quotient). By construction, we have a map
\begin{equation}\label{eq:ctonaive}
\on{LocSys}_G(\overset{\circ}{D_x})\to \on{LocSys}_G(\overset{\circ}{D_x})^{\on{naive}}.
\end{equation}

\noindent Moreover, the map (\ref{eq:ctonaive}) is a monomorphism. Indeed, this simply follows from the fact that $\on{pt}/G \to \bB G$ is a monomorphism.

\subsubsection{} Since we have a natural map $(\overset{\circ}{D_x})^{\on{fake}}_{\on{dR}}\to U_{\on{dR}}$, we obtain a restriction map
\[
\on{LocSys}_G(U)^{\on{naive}}\to \on{LocSys}_G(\overset{\circ}{D_x})^{\on{naive}}.
\]

\noindent By Lemma \ref{l:drinfsimp}, the composition 
\[
\on{LocSys}_G(U)\to \on{LocSys}_G(U)^{\on{naive}}\to \on{LocSys}_G(\overset{\circ}{D_x})^{\on{naive}}
\]

\noindent naturally factors as 
\[\begin{tikzcd}
	{\mathrm{LocSys}_G(U)} && {\mathrm{LocSys}_G(\overset{\circ}{D_x})} \\
	\\
	{\mathrm{LocSys}_G(U)^{\mathrm{naive}}} && {\mathrm{LocSys}_G(\overset{\circ}{D_x})^{\mathrm{naive}}.}
	\arrow[from=1-1, to=3-1]
	\arrow[from=3-1, to=3-3]
	\arrow[dashed, from=1-1, to=1-3]
	\arrow[from=1-3, to=3-3]
\end{tikzcd}\]

\noindent This produces the desired map (\ref{eq:restric}).

\subsection{Rigid local systems}\label{s:rigidls}

\subsubsection{} Choose eigenvalues $\chi_x\in \ft//\widetilde{W}^{\on{aff}}$ for all $x\in S$. We put
\[
\LSN:=\LSU\underset{\prod_{x\in S} \LSD}{\times} \prod_{x\in S} \on{LocSys}_G(\overset{\circ}{D}_x)^{\on{RS}}_{\chi_x}.
\]

\subsubsection{} Let us study the geometric properties of $\LSN$ in more detail. Define $\widetilde{\on{LocSys}}_G(U)^{\on{RS}}$ to be the stack whose $T$-points are pairs $(\sP,\nabla)$, where $\sP$ is a $G$-bundle on $X\times T$, and $\nabla$ is a connection of $\sP$ on $U\times S$ relative to $S$ whose polar part is an element of
\[
\Gamma(X\times T, \fg_{\sP}\otimes \Omega^1_X( S)/\Omega^1_X)\subset \Gamma(X\times T, \fg_{\sP}\otimes \Omega^1_X(\infty\cdot S)/\Omega^1_X).
\]

\noindent This stack is studied in detail in \cite{herrero2020quasicompactness}.

\subsubsection{} We have a natural residue map
\[
\widetilde{\on{LocSys}}_G(U)^{\on{RS}}\to \fg/G\to \ft//\widetilde{W}^{\on{aff}}
\]

\noindent for each $x\in S$. We let
\[
\widetilde{\on{LocSys}}_G(U,\chi_S):=\widetilde{\on{LocSys}}_G(U)^{\on{RS}}\underset{\underset{x\in S}{\prod}\ft//\widetilde{W}^{\on{aff}}}{\times} \underset{x\in S}{\prod}\chi_x.
\]

\noindent In other words, $\widetilde{\on{LocSys}}_G(U,\chi_S)$ parametrizes $G$-bundles on $X$ equipped with a regular singular connection on $U$ whose eigenvalues of the residues at $x\in S$ are given by $\chi_x$.

\subsubsection{} By Lemma \ref{l:cstresidue}, we have a natural surjective map
\begin{equation}\label{eq:nottildetotilde}
\widetilde{\on{LocSys}}_G(U,\chi_S)\to \LSN.
\end{equation}

\subsubsection{} The following is the main theorem of \cite{herrero2020quasicompactness}:
\begin{thm}\label{l:qsqcher}
$\widetilde{\on{LocSys}}_G(U,\chi_S)$ is an algebraic stack of finite type.
\end{thm}
\begin{proof}
By \cite[Thm. 4.16]{herrero2020quasicompactness}, for a $k$-point
\[
\on{Spec}(k)\to \prod_{x\in S} \LSD
\]

\noindent corresponding to a tuple of regular singular local systems, the algebraic stack
\[
\LSU\underset{\underset{x\in S}{\prod}\LSD}{\times} \on{Spec}(k)
\]

\noindent is of finite type. The Theorem now follows from Lemma \ref{l:finman}.

\end{proof}

\begin{cor}\label{c:qcqs}
$\LSN$ is a quasi-smooth algebraic stack of finite type.
\end{cor}

\begin{proof}
The fact that $\LSN$ is quasi-smooth follows from Lemma \ref{l:TC} below. Moreover, it is easy to see that $\LSN$ is locally of finite type. Since the map (\ref{eq:nottildetotilde}) is surjective and $\widetilde{\on{LocSys}}_G(U,\chi_S)$ is of finite type, so is $\LSN$.
\end{proof}

\subsubsection{Fixing the abelianization} Before defining the notion of a rigid local system, we need to introduce a variant of $\LSN$ where we fix the abelianization of the local systems in the moduli description.

\subsubsection{}\label{s:2.3.2} Let $G^{\on{ab}}=G/[G,G]$ be the abelianization of $G$. Write $\chi_x^{\on{ab}}$ for the image of $\chi_x$ under the map $\ft/X_{\bullet}(T)\to \fg^{\on{ab}}/X_{\bullet}(G^{\on{ab}})$. Note that we have a map
\[
\LSN \to \on{LocSys}_{G^{\on{ab}}}(U,\chi_S^{\on{ab}})= \on{LocSys}_{G^{\on{ab}}}(U)\underset{\underset{x\in S}{\prod}\on{LocSys}_{G^{\on{ab}}}(\overset{\circ}{D}_x)}{\times}\prod_{x\in S}\lbrace\chi_x^{\on{ab}}\rbrace/G^{\on{ab}}.
\]

\noindent For $\tau\in \on{LocSys}_{G^{\on{ab}}}(U, \chi_S^{\on{ab}})(k)$, we may form the fiber product
\[
\LS:=\LSN\underset{\LSBA}{\times}\lbrace \tau\rbrace.
\]

\noindent Analytically (i.e., when $k=\bC$ and under Riemann-Hilbert), the above stack parametrizes $G$-local systems on $U$ whose monodromies have eigenvalue $\chi_x$ at $x$ and whose abelianization is $\tau$.

\subsubsection{}\label{s:rigid} For each $x\in S$, fix $\cO_x\in \on{LocSys}_G(\overset{\circ}{D}_x)^{\on{RS}}_{\chi_x}(k)$. Let $\on{Aut}(\cO_x)$ be its stabilizer in $\on{LocSys}_G(\overset{\circ}{D}_x)^{\on{RS}}_{\chi_x}$. Consider the substack
\[
\LSsigma:=\LS\underset{\prod_{x\in S}\on{LocSys}_G(\overset{\circ}{D}_x)_{\chi_x}^{\on{RS}}}{\times} \prod_{x\in S}\bB\on{Aut}(\cO_x)\subset \LS.
\]

\noindent Analytically, choosing $\cO_x\in \on{LocSys}_G(\overset{\circ}{D}_x)^{\on{RS}}_{\chi_x}$ amounts to choosing a conjugacy class inside the stack $G/G=\on{LocSys}_G(S^1)$.

\begin{defin}
A $G$-local system $\sigma$ on $U$ is \emph{rigid} if it defines an isolated point in $\LSsigma$ for some  $\lbrace \cO_x\rbrace_{x\in S}$ and some abelianization $\tau$. That is, the locally closed embedding $\sigma/\on{Aut}(\sigma)\into (\LSsigma)_{\on{red}}$ is required to be both an open and a closed embedding.
\end{defin}

\begin{rem}\label{r:cohorigid}
If $\sigma$ is irreducible, this is equivalent to requiring that $\sigma$ is isolated in $\LSsigma^{\on{irr}}$, where the latter denotes the open substack of $\LSsigma$ whose local systems are irreducible. Indeed, this follows from the fact that if $\sigma$ is an irreducible local system, the map $\sigma/\on{Aut}(\sigma)\into \LSN$ is a closed embedding (see e.g. \cite[Prop. 4.3.5]{arinkin2020stack}).\footnote{Note that isolated points are not automatically preserved under open embeddings of stacks. For example, the embedding $\on{pt}=\bG_m/\bG_m\into \bA^1/\bG_m$ provides a counter example.}
\end{rem}

\subsubsection{}\label{s:cohorigid}
As noted in §\ref{s:differentrigids}, there is an a priori stronger notion of rigidity called \emph{cohomological rigidity}. By definition, a $G$-local system $\sigma$ on $U$ is cohomologically rigid if it defines a \emph{smooth} isolated point in $\LSsigma$. Let us show that the notions of rigidity and cohomological rigidity coincide for local systems on curves.\footnote{This is not true in higher dimensions, see \cite{de2022rigid}.} That is, if one has an isolated point in $\LSsigma$, then it is smooth.
\begin{lem}\label{l:TC}
For $\sigma \in \LSsigma$, the tangent complex at $\sigma$ is given by 
\[
T_{\sigma}\LSsigma=H^*_{\on{dR}}(U, j_{!*}\fg^{\on{der}}_{\sigma})[1].
\]

\noindent Here $j$ is the open embedding $U\into X$, and $\fg^{\on{der}}_{\sigma}$ is the local system on $U$ induced by the $G$-representation $\fg^{\on{der}}:=\on{Lie}([G,G])$.

\end{lem}

\begin{proof}
First, we claim that 
\[
T_{\sigma}\on{LocSys}_G(U)\simeq H^*_{\on{dR}}(U,\fg_{\sigma})[1].
\]

\noindent This is clear if we replace $\on{LocSys}_G(U)$ with $\on{LocSys}_G(U)^{\on{naive}}=\on{Maps}(U_{\on{dR}},\bB G)$, see e.g. \cite{arinkin2015singular}. Thus, it suffices to show that
\[
\on{LocSys}_G(U)\to \on{LocSys}_G(U)^{\on{naive}}
\]

\noindent is formally étale. By definition, it suffices to show that
\[
\overset{\circ}{\on{Bun}}_G\to \on{Bun}_G(U)
\]

\noindent is formally étale. Concretely, we will show that if $T\subset T'$ is a first-order thickening of an affine scheme $T$, and we have a diagram
\[\begin{tikzcd}
	T && {\overset{\circ}{\on{Bun}}_G} \\
	\\
	{T'} && {\mathrm{Bun}_G(U),}
	\arrow[from=1-1, to=1-3]
	\arrow[from=1-3, to=3-3]
	\arrow[from=1-1, to=3-1]
	\arrow[from=3-1, to=3-3]
\end{tikzcd}\]

\noindent there exists a unique lift $T'\to \overset{\circ}{\on{Bun}}_G$. In other words, we need to show that if $\sP$ is a $G$-bundle on $U\times T'$ such that $\sP$ admits an extension to $X\times T$ when restricted to $U\times T$, then $\sP$ admits an extension to $X\times T'$. However, this is immediate from formal smoothness of $\bB G$.

The lemma now follows from the argument \cite[§3.2.2]{yun2014rigidity}.
\end{proof}

\begin{lem}\label{l:equiv}
For $\sigma\in \LSsigma$ irreducible, the following are equivalent:
\begin{itemize}
    \item $\sigma$ defines a smooth isolated point in $\LSsigma$.

    \item The tangent complex $T_{\sigma}\LSsigma$ vanishes.

    \item $H^1(U_{\on{dR}}, j_{!*}\fg^{\on{der}}_{\sigma})=0$.

\end{itemize}

\end{lem}

\begin{proof}
Suppose first that $\sigma$ defines a smooth isolated point in $\LSsigma$. It follows that the tangent complex $T_{\sigma}\LSsigma$ coincides with that of $\sigma/\on{Aut}(\sigma)$, where $\on{Aut}(\sigma)$ is the automorphism group of $\sigma$ in $\LSsigma$. We claim that $\on{Aut}(\sigma)$ is a finite group. As noted in \cite[§11.2.4]{faergeman2022non}, the automorphism group of $\sigma$ in $\LSN$ is some finite extension of the center of $G$. Since the map $Z(G)\to G^{\on{ab}}$ is an isogeny, it follows from the definition of $\LSsigma$ that $\on{Aut}(\sigma)$ is a finite group. In particular
\[
T_{\sigma}\LSsigma=0.
\]

Next, let us prove that if $\sigma\in \LSsigma$ is irreducible as a local system, then 
\begin{equation}\label{eq:vani}
H^0(X_{\on{dR}}, j_{!*}\fg^{\on{der}}_{\sigma})=H^2(X_{\on{dR}}, j_{!*}\fg^{\on{der}}_{\sigma})=0.
\end{equation}

\noindent By Poincaré duality, it suffices to show that $H^0(X_{\on{dR}}, j_{!*}\fg^{\on{der}}_{\sigma})$ vanishes. Note that for any local system $\sigma$, one has 
\[
H^0(X_{\on{dR}}, j_*\fg_{\sigma})=\on{Lie}(Z(\sigma)),
\]

\noindent where $Z(\sigma)$ denotes the centralizer of $\sigma$ in $G$.\footnote{That is, if $\sigma$ is represented by a representation $\pi_1(U)\to G$, then $Z(\sigma)$ is the subgroup of $G$ consisting of elements that centralize the image.} As noted above, when $\sigma$ is irreducible, one has $\on{Lie}(Z(\sigma))=\on{Lie}(Z(G))$. It follows that
\[
H^0(X_{\on{dR}}, j_*\fg^{\on{der}}_{\sigma})=0.
\]

We have a triangle
\[
i_{*,\on{dR}}i^!j_{!*}\fg^{\on{der}}_{\sigma}\to j_{!*}\fg^{\on{der}}_{\sigma}\to j_*\fg^{\on{der}}_{\sigma}.
\]

\noindent Note that $i_{*,\on{dR}}i^!j_{!*}\fg^{\on{der}}_{\sigma}$ is concentrated in perverse degrees $\geq 1$. In particular, we have an injection 
\[
H^0(X_{\on{dR}}, j_{!*}\fg^{\on{der}}_{\sigma})\into H^0(X_{\on{dR}}, j_*\fg^{\on{der}}_{\sigma})=0.
\]

\noindent This shows that $H^1(X_{\on{dR}}, j_{!*}\fg^{\on{der}}_{\sigma})=0$ is equivalent to the vanishing of the tangent complex of $\LSsigma$ at $\sigma$. The latter clearly implies that $\sigma$ defines a smooth isolated point.
\end{proof}

\begin{cor}
Any isolated point in $\LSsigma$ is automatically smooth.
\end{cor}

\begin{proof}
Let $\sigma$ be such an isolated point. By Lemma \ref{l:equiv}, it suffices to show that
\[
H^1(X_{\on{dR}}, j_{!*}\fg^{\on{der}}_{\sigma})=0.
\]
We know from the above proof that the automorphism group of $\sigma$ is finite. Note that by Lemma \ref{l:TC}, $\LSsigma$ is a quasi-smooth algebraic stack in the sense of \cite[\S 8]{arinkin2015singular}. As such, it follows (see e.g. \cite[p.2]{khan2019virtual}) that we have an inequality
\[
\on{v.dim}\LSsigma\leq \on{dim}\LSsigma,
\]

\noindent where $\on{v.dim}\LSsigma=\chi(T_{\sigma}\LSsigma)$ denotes the virtual dimension of $\LSsigma$ at the connected component containing $\sigma$. Since the dimension of this component is zero, it follows that
\[
\on{dim}H^1(X_{\on{dR}}, j_{!*}\fg^{\on{der}}_{\sigma})=\chi(T_{\sigma}\LSsigma)\leq 0.
\]

\noindent Here we used that $H^0(X_{\on{dR}}, j_{!*}\fg^{\on{der}}_{\sigma})=H^2(X_{\on{dR}}, j_{!*}\fg^{\on{der}}_{\sigma})=0$ as noted in the previous lemma.
\end{proof}

\section{Spectral decomposition}

In this section, we prove the spectral decomposition described in \S\ref{ss:sd}. We mimic Drinfeld-Gaitsgory's proof in the unramified case \cite{gaitsgory2010generalized}. A key input is a recent extension of Bezrukavnikov's geometric realization of the affine Hecke category to the situation of non-zero eigenvalues due to Gurbir Dhillon, Yau Wing Li, Zhiwei Yun and Xinwen Zhu, which we review in the following subsection.

\subsection{Geometric realization of twisted affine Hecke category}\label{ss:sheaf}

\subsubsection{Mellin transform} Recall (e.g. \cite[App. A]{gannon2022classification}) that the Mellin transform provides a symmetric monoidal equivalence
\[
D(\check{T})\simeq \on{QCoh}(\ft/X_{\bullet}(T)).
\]
\noindent Here, the monoidal structure on $D(\check{T})$ is given by convolutation, and the monoidal structure on $\on{QCoh}(\ft/X_{\bullet}(T))$ is given by tensor product.

Let $\chi \in \ft/X_{\bullet}(T)$ corresponding to a map $\on{Spec}(k)\to \ft/X_{\bullet}(T)$. Pulling back along this map yields a symmetric monoidal functor 
\[
D(\check{T})\simeq \on{QCoh}(\ft/X_{\bullet}(T))\to \on{Vect}.
\]

Denote by $m$ the multiplication map $\check{T}\times \check{T}\to \check{T}$. Then
\[
m^!: D(\check{T})\to D(\check{T}\times \check{T})\simeq D(\check{T})\otimes D(\check{T})
\]

\noindent endows $D(\check{T})$ with a comonoidal structure, which is dual to $(D(\check{T}),m_{*,\on{dR}})$. Thus, the above monoidal functor yields a comonoidal functor 
\[
\on{Vect}\to D(\check{T}).
\]

\noindent Such a functor is determined by a multiplicative sheaf in $D(\check{T})$, which we similarly denote by $\chi$.\footnote{By a multiplicative sheaf, we mean a D-module $\chi\in D(\check{T})$ with an identification $m^!(\chi)\simeq \chi\boxtimes \chi.$}

\subsubsection{} Let $\check{I}$ be the Iwahori subgroup of $\check{G}(O)$
given by the preimage of $\check{B}$ under the map $\check{G}(O)\to \check{G}$ evaluating at $t=0$.

We have natural maps 
\[
\check{I}\to \check{B}\to \check{T}.
\]

\noindent Pulling back $\chi$ along this composition, we get a multiplicative sheaf on $\check{I}$ that we also denote by $\chi$.

\subsubsection{} Consider the twisted affine Hecke category
\[
D(\check{I}, \chi\backslash \check{G}(K)/\check{I},\chi)
\]

\noindent of bi-($\check{I},\chi)$-equivariant D-modules on $\check{G}(K)$. This category carries a monoidal structure under convolution. We need the following theorem in our construction of the spectral decomposition.

\begin{thm}[Dhillon-Li-Yun-Zhu, \cite{dhillonendo}]\label{t:metaplectic}
There is a monoidal equivalence:
\[
\on{IndCoh}(\on{LocSys}_B(\overset{\circ}{D})_{\chi}\underset{\on{LocSys}_G(\overset{\circ}{D})}{\times}\on{LocSys}_B(\overset{\circ}{D})_{\chi})\simeq D(\check{I},\chi\backslash \check{G}(K)/\check{I}, \chi)^{\on{ren}},
\]

\noindent where both sides are equipped with their respective convolution monoidal structures.
\end{thm}

\begin{rem}\label{r:renorm}
Here, $D(\check{I},\chi\backslash \check{G}(K)/\check{I}, \chi)^{\on{ren}}$ is the renormalized category of $D(\check{I},\chi\backslash \check{G}(K)/\check{I}, \chi)$ defined as the ind-completion of the small subcategory of $D(\check{I},\chi\backslash \check{G}(K)/\check{I}, \chi)$ consisting of objects that become compact after applying the forgetful functor
\[
D(\check{I},\chi\backslash \check{G}(K)/\check{I}, \chi)\to D(\check{I},\chi\backslash \check{G}(K)).
\]

\noindent In particular, we have a fully faithful embedding
\[
D(\check{I},\chi\backslash \check{G}(K)/\check{I}, \chi)\into D(\check{I},\chi\backslash \check{G}(K)/\check{I}, \chi)^{\on{ren}}.
\]

\end{rem}

\subsubsection{} Note that we have a monoidal functor 
\[
\on{QCoh}(\on{LocSys}_G(\overset{\circ}{D})^{\on{RS}}_{\chi})\to \on{QCoh}(\on{LocSys}_B(\overset{\circ}{D})_{\chi})\to 
\]
\[
\to \on{QCoh}(\on{LocSys}_B(\overset{\circ}{D})_{\chi}\underset{\on{LocSys}_G(\overset{\circ}{D})}{\times}\on{LocSys}_B(\overset{\circ}{D})_{\chi}),
\]

\noindent where the first functor is given by pullback along $\on{LocSys}_B(\overset{\circ}{D})_{\chi}\to \on{LocSys}_G(\overset{\circ}{D})^{\on{RS}}_{\chi}$, and the second functor is pushforward via the diagonal map. By construction, the image lies in 
\[
D(\check{I},\chi\backslash \check{G}(K)/\check{I}, \chi)\subset D(\check{I},\chi\backslash \check{G}(K)/\check{I}, \chi)^{\on{ren}}
\]

\noindent under the equivalence of Theorem \ref{t:metaplectic}.
\begin{rem}
We remind that the stack $\on{LocSys}_G(\overset{\circ}{D})_{\chi}^{\on{RS}}$ only depends on $\chi\in \ft/X_{\bullet}(T)$ up to its projection to $\ft//\widetilde{W}^{\on{aff}}$, whereas this is not true for the stack $\on{LocSys}_B(\overset{\circ}{D})_{\chi}$.
\end{rem}

Thus, we get:

\begin{cor}\label{c:spec1}
There is an induced monoidal functor
\[
\on{QCoh}(\on{LocSys}_G(\overset{\circ}{D})^{\on{RS}}_{\chi})\to D(\check{I},\chi_x\backslash \check{G}(K)/\check{I},\chi).
\]
\end{cor}

\subsection{Ran space with marked points}

We will use the language of factorization categories and factorization spaces in the sense of \cite{raskin2015chiral}. We refer to \cite{Cstterm} for a detailed account of factorization and unitality.

\subsubsection{Ran space} Recall the Ran space of $X$ as considered in \cite{gaitsgory2010notes}. Let $\on{fSet}$ denote the $(1,1)$-category whose objects are non-empty finite sets and whose morphisms are surjective maps. For each surjective map, $\phi:I\twoheadrightarrow J$, we get an induced diagonal map 
\[
\Delta_{\phi}: X^J \to X^I.
\]

\noindent As such, we may form the colimit
\[
\on{Ran}_X:=\underset{\on{fSet}^{\on{op}}}{\on{colim}} \: X^I.
\]

\noindent Concretely, the $T$-points of $\on{Ran}_X$ are identified with non-empty finite subsets of maps $T\to X$.

Replacing $X$ by $X_{\on{dR}}$, we similarly obtain the de-Rham Ran space 
\[
\on{Ran}_{X_{\on{dR}}}=\underset{I\in \on{fSet}^{\on{op}}}{\on{colim}} \: X_{\on{dR}}^I\simeq (\on{Ran}_X)_{\on{dR}}.
\]

\noindent Its $T$-points are given by non-empty finite subsets of maps $T_{\on{red}}\to X$.

Since the maps $\Delta_{\phi}: X^J\to X^I$ are proper, it follows that we may write
\[
D(\on{Ran}_X)\simeq \underset{\on{fSet}^{\on{op}}}{\on{colim}} \: D(X^I),
\]

\noindent where the structure maps are given by $(\Delta_{\phi})_{*,\on{dR}}$.

\subsubsection{Ran space with marked points.} It will be convenient for us to work with a variant of the Ran space in which we account for marked points. We recall this space as considered in \cite[\S 2.5]{gaitsgory2011contractibility}.

\subsubsection{} Let $S$ be a (possibly empty) finite set. We let $\on{fSet}_S$ be the (1,1)-category whose objects are finite sets $I$ equipped with an injective map
\[
S\into I.
\]

\noindent Morphisms in this category are given by surjective maps $I\onto J$ making the diagram
\begin{equation}
\begin{tikzcd}
	S && I \\
	\\
	&& J
	\arrow[hook, from=1-1, to=1-3]
	\arrow[hook, from=1-1, to=3-3]
	\arrow[two heads, from=1-3, to=3-3]
\end{tikzcd}
\end{equation}
\noindent commute. 

\subsubsection{} We take $S$ to be a (possibly empty) fixed subset of $X(k)$ consisting of pairwise different points. That is, an ordered tuple $S=\lbrace x_1,...,x_n\rbrace\subset X(k)$ with $x_i\neq x_j$. We will sometimes use the notation $x^S\in X^S$ instead of $S\in X^n$.

For each $I\in \on{fSet}_S$, consider the fiber product
\[
X_S^I:=x^S\underset{X^S}{\times} X^I.
\]

\noindent Note that for any map $\psi: I\onto J$ in $\on{fSet}_S$, we get an induced map 
\[
X_S^J\xrightarrow{\Delta_{\psi}} X_S^I.
\]

\noindent As such, we may form the colimit
\[
\on{Ran}_{X,S}:=\underset{I\in \on{fSet}_S^{\on{op}}}{\on{colim}}\: X_S^I.
\]

\noindent Concretely, the $T$-points of $\on{Ran}_{X,S}$ are given by finite subsets of $\on{Maps}(T,X)$ that contain the image of $S$ under the map
\[
\on{Maps}(\on{pt},X)\to \on{Maps}(T,X)
\]

\noindent given by precomposing with the projection $T\to \on{pt}$.

Replacing $X$ by $X_{\on{dR}}$, we may similarly form the colimit
\[
\on{Ran}_{X_{\on{dR}},S}:=\underset{I\in \on{fSet}_S^{\on{op}}}{\on{colim}} \: (X_S^I)_{\on{dR}}.
\]

\noindent Note that 
\[
\on{Ran}_{X_{\on{dR}},S}\simeq (\on{Ran}_{X,S})_{\on{dR}}.
\]

\subsubsection{} The natural maps
\[
X^I\times X^J\to X^{I\sqcup J}
\]

\noindent induce a commutative semi-group structure on $\on{Ran}_X$:
\[
\on{add}: \on{Ran}_X\times \on{Ran}_X\to \on{Ran}_X, \:\:\: (\underline{x},\underline{y})\mapsto \underline{x}\cup \underline{y},
\]

\noindent providing an action of $\on{Ran}_X$ on $\on{Ran}_{X,S}$ in the natural way.

As a consequence, we get a (non-unital) symmetric monoidal structure on $D(\on{Ran}_X)$:
\[
D(\on{Ran}_X)\otimes D(\on{Ran}_X)\xrightarrow{\on{add}_{*,\on{dR}}} D(\on{Ran}_X),
\]

\noindent which we refer to as convolution. Note that $D(\on{Ran}_{X,S})$ is a module category for the convolution structure on $D(\on{Ran}_X)$. 

\subsubsection{} $\on{Ran}_{X,S}$ similarly carries a non-unital commutative semi-group structure induced by the maps
\[
X_S^I\times X_S^J\to X_S^{\underset{S}{I\sqcup J}}.
\]

\noindent As above, this gives a non-unital symmetric monoidal convolution structure on $D(\on{Ran}_{X,S})$.

\subsubsection{} $D(\on{Ran}_{X,S})$ also carries a module structure for $D(\on{Ran}_X)$ equipped with its $\overset{!}{\otimes}$-monoidal operation induced by the forgetful map $\on{Ran}_{X,S}\to\on{Ran}_X$. In general, if $\sC$ is a module category for $(D(\on{Ran}_X),\overset{!}{\otimes})$, we say $\sC$ is a category \emph{over} $\on{Ran}_{X_{\on{dR}}}$.

Similarly if $\sD$ is a category equipped with an action of $(D(\on{Ran}_{X,S}),\overset{!}{\otimes})$, we say $\sD$ is a category over $\on{Ran}_{X_{\on{dR}},S}$. Finally, we say that $\sD$ is a factorization category over $\on{Ran}_{X_{\on{dR}},S}$ if it is a factorization module category in the sense of \cite{raskin2015chiral} when considered as a category over $\on{Ran}_{X_{\on{dR}}}$.

\subsubsection{Unitality I}\label{s:uni} Let us recall the notions of unitality for factorization categories. We take the liberty of being somewhat informal and referring to \cite[§6]{raskin2015chiral} for precise definitions.

Suppose $\sD$ is a category over $\on{Ran}_{X_{\on{dR}},S}$. For $I\in \on{fSet}_S$, write $\sD_I$ for its fiber over $(X_S^I)_{\on{dR}}\to \on{Ran}_{X_{\on{dR}},S}$. We say that $\sD$ is \emph{unital} if $\sD$ is equipped with a compatible module structure for the convolution operation on $D(\on{Ran}_{X,S})$. In particular, we require the data of $D(X_S^{I\underset{S}{\sqcup}J})$-linear maps
\[
D(X_S^I)\otimes \cD_J\to \cD_{\underset{S}{I\sqcup J}}
\]

\noindent for all $I,J\in \on{fSet}_{S}$ such that for any $I,J,K$, the following diagram of $D(X_S^{{\underset{S}{I\sqcup J}\underset{S}{\sqcup} K}})\simeq D(X_S^I)\otimes D(X_S^J)\otimes D(X_S^{K})$-linear functors commutes:

\[\begin{tikzcd}
	{D(X_S^I)\otimes D(X_S^J)\otimes \sD_K} && {D(X_S^I)\otimes \sD_{\underset{S}{J\sqcup K}}} \\
	\\
	{D(X_S^{I\underset{S}{\sqcup} J})\otimes \sD_K} && {\sD_{\underset{S}{I\sqcup J}\underset{S}{\sqcup} K}}
	\arrow[from=1-1, to=1-3]
	\arrow[from=1-3, to=3-3]
	\arrow[from=1-1, to=3-1]
	\arrow[from=3-1, to=3-3].
\end{tikzcd}\]

\noindent These functors should satisfy a natural system of higher homotopy compatibilities (cf. \cite[\S 6]{raskin2021chiral}). 

\subsubsection{} If $S=\emptyset$, and $\sC$ is a unital category over $\on{Ran}_{X_{\on{dR}},S}$ in the above sense, we say that $\sC$ is unital over $\on{Ran}_{X_{\on{dR}}}$.

\subsubsection{Unitality II}\label{s:unit2} Suppose now that the assignment
\[
I\mapsto \sD_I\in D(X_S^I)\on{-mod}
\]

\noindent defines a unital sheaf of monoidal categories. In this case, we claim that we can equip
\[
\sD:=\underset{I\in \on{fSet}_S}{\on{colim}} \sD_I
\]

\noindent with a monoidal structure over $\on{Ran}_{X_{\on{dR}},S}$. Namely, we have $D(X_S^I)\otimes D(X_S^J)$-linear maps
\begin{equation}\label{eq:mon}
\sD_I\otimes \sD_J\to \sD_{\underset{S}{I\sqcup J}}
\end{equation}

\noindent for all $I,J\in \on{fSet}_S$ in a compatible manner. The functor (\ref{eq:mon}) is defined as the composition
\[
\sD_I\otimes \sD_J\to \sD_{\underset{S}{I\sqcup J}}\otimes \sD_{\underset{S}{I\sqcup J}}\to \sD_{\underset{S}{I\sqcup J}},
\]

\noindent where:

\begin{itemize}
    \item The first map is defined as the tensor product of the two maps $\sD_I\to \sD_{\underset{S}{I\sqcup J}}$, $\sD_J\to \sD_{\underset{S}{I\sqcup J}}$ given by the unital structure on $I\mapsto \sD_I$.
    
    \item The second map comes from the monoidal structure on $\sD_{\underset{S}{I\sqcup J}}$.
\end{itemize}

\subsubsection{} We denote the above monoidal operation by 
\[
(d_1,d_2)\mapsto d_1\star d_2.
\]

\noindent Note that if $d_1\in \sD$ is supported on $\underline{x}\in \on{Ran}_{X_{\on{dR}},S}(k)$, and $d_2\in \sD$ is supported on $\underline{y}$, then $d_1\star d_2$ is supported on $\underline{x}\cup \underline{y}$. We refer to the monoidal structure $-\star -$ as \emph{external convolution}.

\subsubsection{Independence.}\label{s:indep} If $\sC$ is a unital category over $D(\on{Ran}_X)$, we define
\[
\sC_{\on{indep}}=\sC\underset{D(\on{Ran}_X)}{\otimes} \on{Vect},
\]

\noindent where $D(\on{Ran}_{X})$ acts on $\on{Vect}$ via factorization homology
\[
C_{\on{dR}}(\on{Ran}_{X},-): D(\on{Ran}_{X})\to \on{Vect}, 
\]

\noindent which is symmetric monoidal for the convolution structure on $D(\on{Ran}_X)$. If $\sD$ is a unital category over $\on{Ran}_{X,S}$, we similarly define

\[
\sD_{\on{indep}}:=\sD\underset{D(\on{Ran}_{X,S})}{\otimes} \on{Vect}.
\]

\subsection{Some factorization module spaces} 
\subsubsection{} The goal of the next two subsections is to construct a unital symmetric monoidal category $\NR$ over $\on{Ran}_{X_{\on{dR}},S}$ equipped with a continuous fully faithful embedding
\[
\on{QCoh}(\LSN)\into \NR.
\]

\noindent The fiber of $\NR$ over some $\underline{x}\in \on{Ran}_{X,S}(k)$ will be given by 
\[
\underset{x\in S}{\bigotimes} \on{QCoh}(\on{LocSys}_G(\overset{\circ}{D}_x)^{\on{RS}}_{\chi_x})\otimes \underset{x'\in \underline{x}\setminus S}{\bigotimes}\on{Rep}(G).
\]

The category $\NR$ canonically acts on a suitably defined automorphic category (see Section \ref{s:hecke} below), and the above embedding is the first step towards the spectral decomposition theorem.

\subsubsection{} In fact, we do something more general. Namely, for each $x\in S$, let $\sZ_x\to \on{LocSys}_G(\overset{\circ}{D_x})$ be a prestack living over the space of $G$-local systems on $\overset{\circ}{D_x}$. We construct a space
\[
\sZ_{\on{Ran},S}\to \on{Ran}_{X_{\on{dR}},S}
\]

\noindent whose fiber at $\underline{x}\in\on{Ran}_{X_{\on{dR}},S}(k)$ is given by
\[
\prod_{x\in S} \sZ_x\times \prod_{x'\in \underline{x}\setminus S} \bB G.
\]

\noindent By construction, $\sZ_{\on{Ran},S}$ is a factorization module space for $\on{Jets}_{\on{horiz}}(\bB G)$ in the sense of \cite[\S 6.6]{raskin2015chiral}. The latter is known as the space of horizontal jets of $\bB G$, see Section \ref{s:jets} below.

\subsubsection{} The construction of $\sZ_{\on{Ran},S}$ requires the usage of ''fake'' de Rham disks discussed in Appendix \ref{app:A}. As such, we simply sketch the construction of $\sZ_{\on{Ran},S}$ and leave the details to Appendix \ref{s:fin}. Concretely, the subtlety comes from the fact that the functor
\begin{equation}\label{eq:howeasylifewouldbeifthiswaslocsys}
T\mapsto \on{Maps}((\overset{\circ}{D_x})_{\on{dR}}\times T, \bB G)
\end{equation}

\noindent does not produce $\on{LocSys}_G(\overset{\circ}{D_x})$. If one wants to describe the de-Rham stack of $G$-local systems on the punctured disk as a mapping stack, one needs to work with a modified version of $(\overset{\circ}{D_x})_{\on{dR}}$; namely Gaitsgory's $(\overset{\circ}{D_x})_{\on{dR}}^{\on{fake}}$ recalled in Appendix \ref{app:A}. For all intents and purposes, one may ignore this subtlety and pretend that (\ref{eq:howeasylifewouldbeifthiswaslocsys}) really does produce $\on{LocSys}_G(\overset{\circ}{D}_x)$.

\subsubsection{} We define the prestack $\on{Jets}_{\on{horiz}}^{S-\on{mer}}(\bB G)_{\on{Ran},S}$ as follows. For an affine scheme $T$ and a map $x_I:T_{\on{red}}\to X_S^I$, consider the adic disk $D_{x_I}$ as in §\ref{s:disk}. The datum of a lift to $\on{Jets}_{\on{horiz}}^{S-\on{mer}}(\bB G)_{\on{Ran},S}$ is given by an element of the groupoid
\[
\lbrace (D_{x_I})_{\on{dR}}\underset{X_{\on{dR}}}{\times}U_{\on{dR}}\to \bB G\rbrace.
\]

\subsubsection{} For $I\in \on{fSet}_S$, we write
\[
\on{Jets}_{\on{horiz}}^{S-\on{mer}}(\bB G)_I:=\on{Jets}_{\on{horiz}}^{S-\on{mer}}(\bB G)_{\on{Ran},S}\underset{\on{Ran}_{X_{\on{dR}}}}{\times} (X_S^I)_{\on{dR}}.
\]

\subsubsection{} Let $S\subset \underline{x}\subset X(k)$. Note that the fiber of $\on{Jets}^{S-\on{mer}}_{\on{horiz}}(\bB G)_{\on{Ran},S}$ over $\underline{x}$ is given by 
\[
\prod_{x\in S} \on{Maps}((\overset{\circ}{D_x})_{\on{dR}}, \bB G)\times \prod_{x'\in \underline{x}\setminus S} \on{Maps}((D_{x'})_{\on{dR}}, \bB G)\simeq \prod_{x\in S} \on{LocSys}_G(\overset{\circ}{D_x})\times \prod_{x'\in \underline{x}\setminus S} \bB G.
\]

\subsubsection{} By construction, each $\on{Jets}_{\on{horiz}}^{S-\on{mer}}(\bB G)_I$ comes equipped with a map 
\[
\on{Jets}_{\on{horiz}}^{S-\on{mer}}(\bB G)_I\to \prod_{x\in S} \on{LocSys}_G(\overset{\circ}{D_x}).
\]

\noindent We define
\[
\sZ_I:=\on{Jets}_{\on{horiz}}^{S-\on{mer}}(\bB G)_I\underset{\prod_{x\in S} \on{LocSys}_G(\overset{\circ}{D_x})}{\times} \prod_{x\in S} \sZ_x,
\]
\[
\sZ_{\on{Ran},S}:=\on{Jets}_{\on{horiz}}^{S-\on{mer}}(\bB G)_{\on{Ran},S}\underset{\prod_{x\in S} \on{LocSys}_G(\overset{\circ}{D_x})}{\times} \prod_{x\in S} \sZ_x.
\]

\subsubsection{} Note that the space $\sZ_{\on{Ran},S}$ is naturally co-unital over $\on{Ran}_{X_{\on{dR}},S}$. In explicit terms, this means that we have maps
\begin{equation}\label{eq:co-unital}
\psi_{I,J}: \sZ_{\underset{S}{I\sqcup J}}\to X_S^I\times \sZ_{J}
\end{equation}

\noindent over $X_S^{\underset{S}{I\sqcup J}}$ for any $I,J\in \on{fSet}_S$ satisfying a natural system of compatibilities (see \cite[\S 10]{chen2021extension}).

\subsubsection{} Upon applying the functor $\on{QCoh}(-)$, we get a category $\on{QCoh}(\sZ_{\on{Ran},S})$ over $\on{Ran}_{X_{\on{dR},S}}$. The co-unital structure on $\sZ_{\on{Ran},S}$ provides a unital struture on $\on{QCoh}(\sZ_{\on{Ran},S})$ by pulling back via the maps (\ref{eq:co-unital}). Section \ref{s:unit2} now provides a symmetric monoidal structure on $\on{QCoh}(\sZ_{\on{Ran},S})$.

\begin{rem}
We remark that $\on{QCoh}(\sZ_{\on{Ran},S})$ has a natural structure of a factorization module category for $\on{Rep}(G)_{\on{Ran}}$ (the latter being defined e.g. in \cite[\S 11]{arinkin2020stack}).
\end{rem}

\subsection{Spectral localization}\label{lctII}

\subsubsection{} Let us specialize to the case of interest. Namely, when $\sZ_x=\on{LocSys}_G(\overset{\circ}{D}_x)^{\on{RS}}_{\chi_x}$ with its embedding into $\on{LocSys}_G(\overset{\circ}{D_x})$. Let us write $(\on{LocSys}_G(\overset{\circ}{D})^{\on{RS}}_{\chi_S})_{\on{Ran},S}$ for the resulting co-unital space over $\on{Ran}_{X_{\on{dR}},S}$ constructed in $\S 3.3$. Thus, we get a symmetric monoidal unital category 
\[
\NR
\]

\noindent over $\on{Ran}_{X_{\on{dR}},S}$.

\subsubsection{} Let us construct a symmetric monoidal functor
\begin{equation}\label{eq:sf}
\NR\to \on{QCoh}(\LSN).
\end{equation}

\noindent Namely, for each $I\in \on{fSet}_S$ and a map $x_I:T_{\on{red}}\to X_S^I$, we have a tautological map
\begin{equation}
(D_{x_I})_{\on{dR}}\underset{X_{\on{dR}}}{\times} U_{\on{dR}}\to U_{\on{dR}}\times T,
\end{equation}

\noindent thus providing a map\footnote{A priori, this only provides a map $\on{LocSys}_G(U)^{\on{naive}}\times X_S^I\to \on{Jets}_{\on{horiz}}^{S-\on{mer}}(\bB G)_I$, cf. the notation of §\ref{s:lsUnaive}. The fact that the map descends to $\on{LocSys}_G(U)$ follows from the construction of $\on{Jets}_{\on{horiz}}^{S-\on{mer}}(\bB G)_I$ in Appendix \ref{app:A}.}
\begin{equation}\label{eq:mapres}
\on{LocSys}_G(U)\times X_S^I\to \on{Jets}_{\on{horiz}}^{S-\on{mer}}(\bB G)_I.
\end{equation}

\noindent Moreover, (\ref{eq:mapres}) factors through a map
\[
\on{ev}_I: \LSN\times X_S^I \to (\on{LocSys}_G(\overset{\circ}{D})^{\on{RS}}_{\chi_S})_I.
\]

\noindent Pulling back then yields a continuous functor
\[
\on{ev}_I^*: \on{QCoh}((\on{LocSys}_G(\overset{\circ}{D})^{\on{RS}}_{\chi_S})_I)\to \on{QCoh}(\LSN)\otimes D(X_S^I).
\]

\noindent Taking the colimit over $\on{fSet}_S$, we obtain a functor 
\begin{equation}\label{eq:coloc}
\on{ev}^*: \NR\to \on{QCoh}(\LSN)\otimes D(\on{Ran}_{X,S}).
\end{equation}

\noindent Moreover, it follows from definitions that (\ref{eq:coloc}) is symmetric mononidal, where:
\begin{itemize}
\item The left-hand is equipped with the symmetric monoidal structure described in \S \ref{s:unit2}.

\item The symmetric monoidal structure on the right-hand side is induced by the usual tensor product on $\on{QCoh}(\LSN)$ and the convolution structure on $D(\on{Ran}_{X,S})$.

\end{itemize}

\subsubsection{} Denote by $\pi$ the projection map $\on{Ran}_{X,S}\to \on{pt}$. Then the functor
\[
\pi_{*,\on{dR}}: D(\on{Ran}_{X,S})\to \on{Vect}
\]

\noindent is symmetric monoidal, where the left-hand side is equipped with its convolution monoidal structure. Thus, the composition
\[
\NR\xrightarrow{\on{ev}^*} \on{QCoh}(\LSN)\otimes D(\on{Ran}_{X,S})\to 
\]
\[
\xrightarrow{\on{id}\otimes \pi_{*,\on{dR}}} \on{QCoh}(\LSN)
\]

\noindent is symmetric monoidal and is the sought-for functor (\ref{eq:sf}).

\subsubsection{} Since $\on{Ran}_{X,S}$ is pseudo-ind-proper (cf. \cite{gaitsgory2011contractibility}), the right adjoint to $\pi_{*,\on{dR}}$ is given by $\pi^!$.  Consider the right adjoint to the above composition. Namely:
\[
\on{ev}_*\circ (\on{id}\times\pi)^!: \on{QCoh}(\LSN)\to \NR,
\]

\noindent i.e., by pull-push along the diagram
\[\begin{tikzcd}
	{\LSN\times \textrm{Ran}_{X_{\textrm{dR}},S}} && {(\on{LocSys}_G(\overset{\circ}{D})^{\on{RS}}_{\chi_S})_{\on{Ran},S}} \\
	\\
	{\LSN.}
	\arrow["{\textrm{id}\times \pi}"', from=1-1, to=3-1]
	\arrow["{\textrm{ev}}", from=1-1, to=1-3]
\end{tikzcd}\]

\noindent The following proposition is essentially (a special case of) \cite[Cor. C.1.8]{ambidex}.

\begin{prop}\label{p:ff}
The functor
\[
\on{ev}_*\circ (\on{id}\times\pi)^!: \on{QCoh}(\LSN)\to \NR
\]

\noindent is continuous and fully faithful.
\end{prop}

\begin{proof}
Consider
\[
(\NR)_{\on{{indep}}}=\NR\underset{D(\on{Ran}_{X,S})}{\otimes} \on{Vect}
\]

\noindent as in \S \ref{s:indep}. Note that the functor in question naturally factors as
\[
\on{QCoh}(\LSN)\to (\NR)_{\on{{indep}}}\to \NR.
\]

\noindent Since the second functor above is fully faithful (\cite[\S 2.5]{gaitsgory2011contractibility}), it suffices to show that the map
\begin{equation}\label{eq:indep1}
\on{QCoh}(\LSN)\to (\NR)_{\on{{indep}}}
\end{equation}

\noindent is fully faithful.

In Appendix \ref{s:constr}, we construct a co-unital factorization space $(\on{LocSys}_G(\overset{\circ}{D})^{\on{RS}}_{\chi_S})_{\on{Ran}}$ over $\on{Ran}_{X_{\on{dR}}}$ such that
\[
(\on{LocSys}_G(\overset{\circ}{D})^{\on{RS}}_{\chi_S})_{\on{Ran},S}=(\on{LocSys}_G(\overset{\circ}{D})^{\on{RS}}_{\chi_S})_{\on{Ran}}\underset{\on{Ran}_{X_{\on{dR}}}}{\times} \on{Ran}_{X_{\on{dR}},S}.
\]

\noindent Write $(\on{LocSys}_G(\overset{\circ}{D})^{\on{RS}}_{\chi_S})_{X_{\on{dR}}}$ for the base-change of $(\on{LocSys}_G(\overset{\circ}{D})^{\on{RS}}_{\chi_S})_{\on{Ran}}$ to $X_{\on{dR}}$. By Lemma \ref{l:sectrschi}, we have
\[
\LSN=\on{Sect}_{\nabla}(X,(\on{LocSys}_G(\overset{\circ}{D})^{\on{RS}}_{\chi_S})_{X_{\on{dR}}}),
\]

\noindent where $\on{Sect}_{\nabla}(X,(\on{LocSys}_G(\overset{\circ}{D})^{\on{RS}}_{\chi_S})_{\on{dR}})=\on{Maps}_{/ X_{\on{dR}}}(X_{\on{dR}}, (\on{LocSys}_G(\overset{\circ}{D})^{\on{RS}}_{\chi_S})_{X_{\on{dR}}})$ denotes the stack of horizontal sections of $(\on{LocSys}_G(\overset{\circ}{D})^{\on{RS}}_{\chi_S})_{X_{\on{dR}}}$. Moreover, by construction we have
\[
\on{Jets}_{\on{horiz}}((\on{LocSys}_G(\overset{\circ}{D})^{\on{RS}}_{\chi_S})_{X_{\on{dR}}})_{\on{Ran}}=(\on{LocSys}_G(\overset{\circ}{D})^{\on{RS}}_{\chi_S})_{\on{Ran}}.
\]

\noindent Here, the left-hand side denotes the space of horizontal jets into $(\on{LocSys}_G(\overset{\circ}{D})^{\on{RS}}_{\chi_S})_{X_{\on{dR}}}$, as described in §\ref{s:jets} below.

Consider
\[
(\on{QCoh}((\on{LocSys}_G(\overset{\circ}{D})^{\on{RS}}_{\chi_S})_{\on{Ran}})_{\on{indep}}=\on{QCoh}((\on{LocSys}_G(\overset{\circ}{D})^{\on{RS}}_{\chi_S})_{\on{Ran}})\underset{D(\on{Ran}_X)}{\otimes} \on{Vect}.
\]

\noindent Recall cf. Corollary \ref{c:qcqs} that $\LSN$ is an algebraic stack of finite type. Moreover, both categories
\[
\on{QCoh}(\LSN), \;\; \on{QCoh}((\on{LocSys}_G(\overset{\circ}{D})^{\on{RS}}_{\chi_S})_{\on{Ran}})
\]

\noindent are dualizable. Indeed, this is easy to see for $\on{QCoh}((\on{LocSys}_G(\overset{\circ}{D})^{\on{RS}}_{\chi_S})_{\on{Ran}})$.\footnote{Proof: the category $\on{QCoh}(\on{LocSys}_G(\overset{\circ}{D})^{\on{RS}}_{\chi_x})$ is compactly generated (see e.g. \cite{raskin2015notion}). We have natural functors $\on{Rep}(G)\otimes D(U)\to \on{QCoh}((\on{LocSys}_G(\overset{\circ}{D})^{\on{RS}}_{\chi_S})_{X_{\on{dR}}})$ and $\on{QCoh}(\on{LocSys}_G(\overset{\circ}{D})^{\on{RS}}_{\chi_x})\to \on{QCoh}((\on{LocSys}_G(\overset{\circ}{D})^{\on{RS}}_{\chi_S})_{X_{\on{dR}}})$ for each $x\in S$. It is clear that the images of compact objects under these functors compactly generate the target. Using the factorization structure of $\on{QCoh}((\on{LocSys}_G(\overset{\circ}{D})^{\on{RS}}_{\chi_S})_{\on{Ran}})$, it follows that it too is compactly generated. In particular, $\on{QCoh}((\on{LocSys}_G(\overset{\circ}{D})^{\on{RS}}_{\chi_S})_{\on{Ran}})$ is dualizable.} For $\on{QCoh}(\LSN)$, this follows from \cite[Thm 0.3.4]{drinfeld2013some}.

\cite[Cor. C.1.8]{ambidex} now says that the naturally defined functor
\begin{equation}\label{eq:indep2}
\on{QCoh}(\LSN)\to (\on{QCoh}((\on{LocSys}_G(\overset{\circ}{D})^{\on{RS}}_{\chi_S})_{\on{Ran}}))_{\on{{indep}}}
\end{equation}

\noindent is continuous and fully faithful. Pushforward along the natural map $\on{Ran}_{X,S}\to \on{Ran}_X$ induces an equivalence
\[
(\on{QCoh}((\on{LocSys}_G(\overset{\circ}{D})^{\on{RS}}_{\chi_S})_{\on{Ran},S}))_{\on{{indep}}}\simeq (\on{QCoh}((\on{LocSys}_G(\overset{\circ}{D})^{\on{RS}}_{\chi_S})_{\on{Ran}}))_{\on{{indep}}}.
\]

\noindent This shows that (\ref{eq:indep1}) is fully faithful.

\end{proof}

\subsection{Hecke action}\label{s:hecke}

The goal of this subsection is to construct a canonical Hecke action on the automorphic category occurring in Theorem \ref{t:action}.

\subsubsection{} Recall the notation of \S 3.1. Denote by $\on{Bun}_{\check{G}}^{\infty\cdot S}$ the $(\prod_{x\in S} \check{G}(O_x))$-torsor over $\on{Bun}_{\check{G}}(X)$ parametrizing a $\check{G}$-bundle on $X$ equipped with a trivialization on the formal completion of $X$ along $S$.

As usual, the natural action of $\prod_{x\in S} \check{G}(O_x)$ on $\on{Bun}_{\check{G}}^{\infty\cdot S}$ upgrades to an action of $\prod_{x\in S} \check{G}(K_x)$.

\subsubsection{} We consider the category
\begin{equation}\label{eq:inf}
D(\on{Bun}_{\check{G}}^{\infty\cdot S})^{\check{I}_S, \chi_S}
\end{equation}

\noindent of $(\prod_{x\in S} \check{I}_x, \prod_{x\in S} \chi_x)$-equivariant D-modules on $\on{Bun}_{\check{G}}^{\infty\cdot S}$.

It will be more convenient for us to work with another description of (\ref{eq:inf}). Namely, let $\on{Bun}_{\check{G}}(X)^S$ be the $(\prod_{x\in S} \check{G})$-torsor over $\on{Bun}_{\check{G}}(X)$ parametrizing a $\check{G}$-bundle on $X$ equipped with a trivialization on $S\subset X$. Note that $\on{Bun}_{\check{G}}(X)^S$ is the quotient of $\on{Bun}_{\check{G}}^{\infty\cdot S}$ by the pro-unipotent group $\on{Ker}(\prod_{x\in S} \check{G}(O_x)\to \prod_{x\in S} \check{G})$.

\subsubsection{} Consider the stack $\on{Bun}_{\check{G},\check{N}}(X,S)$ parametrizing a $\check{G}$-bundle on $X$ equipped with a reduction to $\check{N}$ at $S\subset X$. We have a natural identification
\[
\BunN\simeq \on{Bun}_{\check{G}}(X)^S/(\prod_{x\in S} \check{N}).
\]

\noindent As such, $\check{T}$ acts on $\BunN$ at each $x\in S$, and we have an equivalence
\begin{equation}\label{eq:eq}
D(\on{Bun}_{\check{G}}^{\infty\cdot S})^{\check{I}_S, \chi_S}\simeq \Autr,
\end{equation}

\noindent where the latter denotes the category of $(\prod_{x\in S}\check{T}, \prod_{x\in S} \chi_x)$-equivariant D-modules on $\BunN$.

\subsubsection{} We are going to construct a canonical Hecke action
\[
\Autreal\curvearrowright \Autr,
\]

\noindent where $\Autreal$ is a natural Hecke category over $\on{Ran}_{X,S}$ defined below.

\subsubsection{}\label{s:group} Let us construct two co-unital spaces $(\check{I}_S)_{\on{Ran}_{X,S}}$, $\check{G}(K)_{\on{Ran}_{X,S}}$ over $\on{Ran}_{X,S}$ whose fiber at $\underline{x}\subset \on{Ran}_{X,S}(k)$ is given by
\[
\prod_{x\in S} \check{I}_x\times \prod_{x'\in \underline{x}\setminus S} \check{G}(O_{x'}),
\]
\[
\prod_{x\in \underline{x}} \check{G}(K_x),
\]

\noindent respectively. Namely, consider the group scheme $\check{G}(O)_{\on{Ran}_{X,S}}$ over $\on{Ran}_{X,S}$ whose $T$-points are:
\[
\check{G}(O)_{\on{Ran}_{X,S}}(T)=\lbrace x_I: T\to X_S^I, \:\: D_{x_I}\to \check{G}\rbrace,
\]

\noindent where $D_{x_I}$ is defined in \S \ref{s:disk}. By construction, we have a map
\[
\check{G}(O)_{\on{Ran}_{X,S}}\to \prod_{x\in S} \check{G}(O_x),
\]

\noindent and we define
\[
(\check{I}_S)_{\on{Ran}_{X,S}}:=\check{G}(O)_{\on{Ran}_{X,S}}\underset{\underset{x\in S}{\prod} \check{G}(O_x)}{\times}{\prod_{x\in S}} \check{I}_x.
\]

\noindent Similarly, we define $\check{G}(K)_{\on{Ran}_{X,S}}$ by
\[
\check{G}(K)_{\on{Ran}_{X,S}}(T)=\lbrace x_I: T\to X_S^I, \:\: \overset{\circ}{D}_{x_I}\to G\rbrace.
\]

\subsubsection{} From the canonical map
\[
(\check{I}_S)_{\on{Ran}_{X,S}}\to \prod_{x\in S} \check{I}_x,
\]

\noindent we get a multiplicative sheaf on $(\check{I}_S)_{\on{Ran}_{X,S}}$\footnote{I.e., a D-module $\chi$ on $(\check{I}_S)_{\on{Ran}_{X,S}}$ that pulls back to $\chi\underset{\on{Ran}_{X,S}}{\boxtimes}\chi$ under the multiplication map $(\check{I}_S)_{\on{Ran}_{X,S}}\underset{\on{Ran}_{X,S}}{\times} (\check{I}_S)_{\on{Ran}_{X,S}}\to (\check{I}_S)_{\on{Ran}_{X,S}}$.} by pulling back the sheaf $\chi_S=\prod_{x\in S}\chi_x$ on $\prod_{x\in S} \check{I}_x$. We denote this multiplicative sheaf on $(\check{I}_S)_{\on{Ran}_{X,S}}$ by $\chi_S$ as well.

\subsubsection{} Note that we have a canonical action
\[
(\check{I}_S)_{\on{Ran}_{X,S}}\curvearrowright \check{G}(K)_{\on{Ran}_{X,S}}
\]

\noindent over $\on{Ran}_{X,S}$. Thus, we may consider the category
\[
\Autreal:=D((\check{I}_S)_{\on{Ran}_{X,S}}, \chi_S\backslash \check{G}(K)_{\on{Ran}_{X,S}}/ (\check{I}_S)_{\on{Ran}_{X,S}}, \chi_S).
\]

\noindent Note that $\Autreal$ comes equipped two monoidal structures. The first is given by point wise convolution over $\on{Ran}_{X,S}$ coming from the convolution structure on
\[
D((\check{I}_S)_J, \chi_S\backslash \check{G}(K)_J/ (\check{I}_S)_J, \chi_S)
\]

\noindent for all $J\in \on{fSet}_S$. Second, there is the external convolution structure described in §\ref{s:unit2}.

\subsubsection{}\label{s:wildram} We claim that $\Autreal$ equipped with its external convolution structure canonically acts on $\Autr$. Indeed, denote by $\on{Bun}_{\check{G},\on{Ran}_{X,S}}^{\infty}$ the prestack parametrizing triples $(\underline{x},\sP_G, \phi)$, where
\begin{itemize}
    \item $\underline{x}\in \on{Ran}_{X,S}$.
    \item $\sP_G$ is a ,$\check{G}$-bundle on $X$.
    \item $\phi$ is a trivialization of $\sP_G$ on $D_{\underline{x}}$.
\end{itemize}

\noindent We have a canonical action
\[
\check{G}(K)_{\on{Ran}_{X,S}}\curvearrowright \on{Bun}_{\check{G},\on{Ran}_{X,S}}^{\infty}
\]

\noindent over $\on{Ran}_{X,S}$. By taking $((\check{I}_S)_{\on{Ran}_{X,S}}, \chi_S)$-invariance, we get an action
\begin{equation}\label{eq:ranact}
\Autreal\curvearrowright D(\BunN\times \on{Ran}_{X,S})^{\check{T}_S, \chi_S}\simeq \Autr\otimes D(\on{Ran}_{X,S})
\end{equation}

\noindent over $\on{Ran}_{X,S}$, where we consider $\Autreal$ with its \emph{pointwise} monoidal structure.

\noindent In turn, this provides an action of $\Autreal$ equipped with its \emph{external} monoidal structure on $\Autr$ via the composition
\[
\Autreal\otimes \Autr\to \Autreal\otimes \Autr\otimes D(\on{Ran}_{X,S})
\]
\[
\to \Autr\otimes D(\on{Ran}_{X,S})\to  \Autr,
\]

\noindent where:

\begin{itemize}
    \item The first arrow is induced by the map $\pi^!: \on{Vect}\to D(\on{Ran}_{X,S})$.

    \item The second arrow is the action map (\ref{eq:ranact}).

    \item The third arrow is induced by the map $\pi_{*,\on{dR}}: D(\on{Ran}_{X,S})\to \on{Vect}$.
\end{itemize}

\noindent Alternatively, the action of $\Autreal$ on $\Autr$ is given by the map
\[
\Autreal \otimes \Autr\to (\Autreal)_{\on{indep}}\otimes \Autr\to \Autr
\]

\noindent with the evident action of $(\Autreal)_{\on{indep}}$ on $\Autr$.

\subsubsection{} We need an upgraded version of Corollary \ref{c:spec1} that allows for factorization away from the ramification points. The following theorem will appear in forthcoming work of the author and Ekaterina Bogdanova.

\begin{thm}[\cite{Satakefunctor}]\label{p:upgrade}
There is a canonical monoidal functor
\[
\NR\to \Autreal
\]

\noindent of unital categories over $\on{Ran}_{X,S}$ which restricts to the functor (\ref{t:metaplectic}) at $x\in S$ and which restricts to the usual (naïve) geometric Satake functor
\[
\on{Rep}(G)\to D(\check{G}(O_x)\backslash \check{G}(K_x)/\check{G}(O_x))
\]

\noindent away from $S$.
\end{thm}

\begin{cor}\label{c:specact}
There is a canonical Hecke action
\[
\NR\curvearrowright \Autr.
\]
\end{cor}
\qed

\subsection{Kac-Moody localization}\label{s:kmloc1}

In this subsection, we consider twisted D-modules on $\BunN$ obtained from Kac-Moody localization. We therefore freely use Kac-Moody notation in the sense of \cite{frenkel2006local}, \cite{chen2021extension} in what follows. See also the introduction of \cite{raskin2020homological} for a discussion of the role of Kac-Moody localization in geometric Langlands.

\subsubsection{} Fix a level $\kappa$ for $\check{\fg}$.\footnote{In the following subsections, we may work at any level. Later, however, when we construct Hecke eigensheaves using opers, it will be important to work at critical level.} For $x\in X(k)$, we let $\hat{\check{\fg}}_{\kappa,x}\on{-mod}$ denote the usual (renormalized) category of Kac-Moody modules for $\hat{\check{\fg}}_{\kappa}$ at level $\kappa$ at $x$, defined initially in \cite{frenkel2009d}. In \cite[\S B.14]{Cstterm} (see also \cite[\S 4]{chen2021extension}), a unital factorization category is constructed,
\[
\hat{\check{\fg}}_{\kappa,\on{Ran}_X}\on{-mod},
\]

\noindent over $\on{Ran}_X$ whose fiber at $x\in X(k)$ is given by
\[
\hat{\check{\fg}}_{\kappa,x}\on{-mod}.
\]

\noindent We let
\[
\hat{\check{\fg}}_{\kappa,\on{Ran}_{X,S}}\on{-mod}:=\hat{\check{\fg}}_{\kappa,\on{Ran}_X}\on{-mod}\underset{D(\on{Ran}_X)}{\otimes} D(\on{Ran}_{X,S}),
\]

\noindent which is a unital category over $\on{Ran}_{X,S}$. 

\subsubsection{}\label{s:km} We have a localization functor
\begin{equation}\label{eq:KMlocwild}
\hat{\check{\fg}}_{\kappa,\on{Ran}_{X,S}}\on{-mod}\to D_{\kappa}(\on{Bun}_{\check{G},\on{Ran}_{X,S}}^{\infty}).
\end{equation}

\noindent Moreover, as in \cite{Cstterm}, we have a strong action (at level $\kappa$):
\[
\check{G}(K)_{\on{Ran}_{X,S}}\curvearrowright \hat{\check{\fg}}_{\kappa,\on{Ran}_{X,S}}\on{-mod}.
\]

\noindent Thus, we obtain a functor
\begin{equation}\label{eq:kmkmloc}
\on{Loc}_{\kappa,\check{G},\chi_S}: (\hat{\check{\fg}}_{\kappa,\on{Ran}_{X,S}}\on{-mod})^{(\check{I}_S)_{\on{Ran}_{X,S}}, \chi_S}\to D_{\kappa}(\BunN)^{\check{T}_S, \chi_S}\otimes D(\on{Ran}_{X,S})\to
\end{equation}
\[
\xrightarrow{\on{id}\otimes \pi_{*,\on{dR}}} D_{\kappa}(\on{Bun}_{\check{G},\check{N}}(X,S)^{\check{T}_S, \chi_S}.
\]

\noindent When $\kappa=\on{crit}$ is the critical level, we fix an identification 
\[
D_{\on{crit}}(\BunN)^{\check{T}_S,\chi_S}\simeq \Autr.
\]

\noindent In this case, the above localization functor is $\NR$-linear via the action given by Theorem \ref{p:upgrade}.

\subsubsection{} The following proposition is proved in Section \ref{s:4.3}.

\begin{prop}\label{p:km}
When $\kappa=\on{crit}$ is critical, the action of $\NR$ on the image of $\on{Loc}_{\on{crit},\check{G},\chi_S}$ factors through an action of $\on{QCoh}(\LSN)$.
\end{prop}

\subsection{Image of Kac-Moody localization}\label{s:jets}
In this subsection, we show that the image of Kac-Moody localization is large, suitably understood.

\subsubsection{} For a prestack $\sY$ over $X_{\on{dR}}$, we may consider the factorization space over $\on{Ran}_{X_{\on{dR}}}$ given by horizontal jets into $\sY$, which we denote by $\on{Jets}_{\on{horiz}}(\sY)_{\on{Ran}}$.\footnote{We refer to \cite[\S A.2]{gaitsgory2014day} for a detailed discussion of jets constructions.} We remind that for a test scheme $T$ equipped with a map $x_I: T\to X_{\on{dR}}^I$, the data of a lift to $\on{Jets}_{\on{horiz}}(\sY)_{\on{Ran}}$ is given by an element of the groupoid
\[
\lbrace (\Gamma_{x_I})_{\on{dR}}\underset{T_{\on{dR}}}{\times} T\to \sY\;\on{over} \; X_{\on{dR}}\rbrace,
\]

\noindent where $\Gamma_{x_I}$ denotes the union of the graphs defined by $x_I$ inside $T_{\on{red}}\times X$. Define
\[
\on{Jets}_{\on{horiz}}(\sY)_{\on{Ran},S}:=\on{Jets}_{\on{horiz}}(\sY)_{\on{Ran}}\underset{\on{Ran}_{X_{\on{dR}}}}{\times} \on{Ran}_{X_{\on{dR}},S}.
\]

\subsubsection{} For a prestack $\sZ$ not necessarily living over $X_{\on{dR}}$, we simply write 
\[
\on{Jets}_{\on{horiz}}(\sZ)_{\on{Ran}}:=\on{Jets}_{\on{horiz}}(\sZ\times X_{\on{dR}})_{\on{Ran}}
\]

\noindent and
\[
\on{Jets}_{\on{horiz}}(\sZ)_{\on{Ran},S}:=\on{Jets}_{\on{horiz}}(\sZ\times X_{\on{dR}})_{\on{Ran},S}.
\]

\subsubsection{} If $\sY$ is a prestack over $X$, we may consider its Weil restriction, $\on{Res}_{X_{\on{dR}}}^X(\sY)$, which is a prestack over $X_{\on{dR}}$. By definition, for a test scheme $T\to X_{\on{dR}}$, we have
\[
\on{Maps}_{/X_{\on{dR}}}(T, \on{Res}_{X_{\on{dR}}}^X(\sY))=\on{Maps}_{/X}(T\underset{X_{\on{dR}}}{\times} X, \sY).
\]

\noindent We write 
\[
\on{Jets}(\sY)_{\on{Ran}}:=\on{Jets}_{\on{horiz}}(\on{Res}_{X_{\on{dR}}}^X(\sY))_{\on{Ran}}
\]

\noindent and
\[
\on{Jets}(\sY)_{\on{Ran},S}:=\on{Jets}(\sY)_{\on{Ran}}\underset{\on{Ran}_{X_{\on{dR}}}}{\times} \on{Ran}_{X_{\on{dR}},S}.
\]

\subsubsection{} For a prestack $\sZ$ not necessarily living over $X$, we simply write 
\[
\on{Jets}(\sZ)_{\on{Ran}}:=\on{Jets}(\sZ\times X)_{\on{Ran}}
\]
\noindent and
\[
\on{Jets}(\sZ)_{\on{Ran},S}:=\on{Jets}(\sZ\times X)_{\on{Ran},S}.
\]

\subsubsection{} Let us specialize to $\sZ=\bB \check{G}$. Note that in this case, the fiber of $\on{Jets}(\bB \check{G})_{\on{Ran},S}$ over $\underline{x}\in \on{Ran}_{X_{\on{dR}},S}(k)$ is given by 
\[
\prod_{x\in \underline{x}} \bB \check{G}(O_x),
\]

\noindent where we identify $\bB \check{G}(O_x)$ with the stack of $\check{G}$-bundles on the formal disk at $x$.
Consider the prestack
\[
\on{Jets}(\bB \check{G}, \bB \check{N})_{\on{Ran},S}:=\on{Jets}(\bB \check{G})_{\on{Ran},S}\underset{\prod_{x\in S}\bB \check{G}(O_x)}{\times} \prod_{x\in S}\bB \overset{\circ}{\check{I}_x},
\]

\noindent where $\overset{\circ}{\check{I}_x}$ denotes the unipotent radical of $\check{I}_x$. We may identify $\bB \overset{\circ}{\check{I}_x}$ with the moduli stack of $\check{G}$-bundles on the formal disk at $x$ equipped with a reduction to $\check{N}$ at the point $x$.

\subsubsection{} Note that for any $x\in S$, we have a restriction map
\[
\BunN\to \bB \overset{\circ}{\check{I}_x}.
\]

\noindent Similarly, for $x'\in X-S$, we have a map
\[
\BunN\to \bB \check{G}(O_{x'}).
\]

\noindent These naturally upgrade to a map
\[
\BunN\times \on{Ran}_{X_{\on{dR}},S}\to \on{Jets}(\mathbb{B}\check{G},\mathbb{B}\check{N})_{\on{Ran},S}
\]

\noindent over $\on{Ran}_{X_{\on{dR}},S}$, giving rise to the correspondence 

\[\begin{tikzcd}
	{\BunN\times \textrm{Ran}_{X_{\textrm{dR}},S}} && {\textrm{Jets}(\mathbb{B}\check{G},\mathbb{B}\check{N})_{\on{Ran},S}} \\
	\\
	{\BunN.}
	\arrow["{\textrm{ev}}", from=1-1, to=1-3]
	\arrow["{\textrm{id}\times \pi}"', from=1-1, to=3-1]
\end{tikzcd}\]

\subsubsection{} For a quasi-compact open substack $j: V\into \BunN$, we may consider the composition
\[
\Gamma_{\check{G},S}^{\on{QCoh, naive}}:=\on{ev}_*\circ (\on{id}\times \pi)^*\circ j_*: \on{QCoh}(V)\to \on{QCoh}(\on{Jets}(\mathbb{B}\check{G},\mathbb{B}\check{N})_{\on{Ran},S}).
\]

\noindent The target category is somewhat ill-behaved.\footnote{E.g., the category $\on{QCoh}(\bB \check{G}(O))$ is not compactly generated.} We renormalize to make it more suitable for our purposes.

\subsubsection{} We have a natural identification
\[
\on{Jets}(\mathbb{B}\check{G},\mathbb{B}\check{N})_{\on{Ran},S}=\bB (\check{I}_S)_{\on{Ran}_{X,S}}.
\]

\noindent Here, the right-hand side is the quotient of $\on{Ran}_{X_{\on{dR}},S}$ by the trivial action of $(\check{I}_S)_{\on{Ran}_{X,S}}$. The group $(\check{I}_S)_{\on{Ran}_{X,S}}$ admits a presentation
\[
(\check{I}_S)_{\on{Ran}_{X,S}}\simeq \underset{n}{\on{lim}}\; (\check{I}^{(n)}_S)_{\on{Ran}_{X,S}},
\]

\noindent where each group scheme $(\check{I}^{(n)}_S)_{\on{Ran}_{X,S}}$ is of finite type over $\on{Ran}_{X_{\on{dR},S}}$. Moreover, each transition map in the above limit is smooth and surjective. We define
\[
\on{Rep}((\check{I}_S)_{\on{Ran}_{X,S}}):=\underset{n}{\on{colim}}\on{QCoh}(\bB (\check{I}^{(n)}_S)_{\on{Ran}_{X,S}}),
\]

\noindent where the transition functors are given by $*$-pullback along the transition maps
\[
\bB(\check{I}^{(n)}_S)_{\on{Ran}_{X,S}}\to \bB (\check{I}^{(m)}_S)_{\on{Ran}_{X,S}}.
\]

\noindent As in \cite[§B.14.2]{Cstterm}, the category
\[
\on{QCoh}( \bB (\check{I}^{(n)}_S)_{\on{Ran}_{X,S}})
\]

\noindent is compactly generated. As a consequence, so is the category
\[
\on{Rep}((\check{I}_S)_{\on{Ran}_{X,S}}).
\]

\noindent Moreover, $\on{Rep}((\check{I}_S)_{\on{Ran}_{X,S}})$ is naturally a factorization category.
\begin{rem}
In the language of \cite{chen2021extension}, we have
\[
\on{Rep}((\check{I}_S)_{\on{Ran}_{X,S}})\simeq \on{IndCoh}^{!,\on{ren}}(\on{Jets}(\bB\check{G},\bB \check{N})).
\]
\end{rem}

\subsubsection{} For each $n$, we have a correspondence
\[\begin{tikzcd}
	{\BunN\times \textrm{Ran}_{X_{\textrm{dR}},S}} && {\bB(\check{I}^{(n)}_S)_{\on{Ran}_{X,S}}} \\
	\\
	{\BunN.}
	\arrow["{\textrm{ev}^{(n)}}", from=1-1, to=1-3]
	\arrow["{\textrm{id}\times \pi}"', from=1-1, to=3-1]
\end{tikzcd}\]

\noindent The functors
\[
j^*\circ (\on{id}\otimes \pi_{*,\on{dR}})\circ \on{ev}^{(n),^*}: \on{QCoh}(\bB(\check{I}^{(n)}_S)_{\on{Ran}_{X,S}})\to \on{QCoh}(\BunN)\to \on{QCoh}(V)
\]

\noindent glue to define a functor
\begin{equation}\label{eq:qcohloc}
\on{Loc}_{\check{G},S}^{\on{QCoh}}: \on{Rep}((\check{I}_S)_{\on{Ran}_{X,S}})\to \on{QCoh}(V).
\end{equation}

\noindent It admits a continuous right adjoint
\[
\Gamma_{\check{G},S}^{\on{QCoh}}: \on{QCoh}(V)\to \on{Rep}((\check{I}_S)_{\on{Ran}_{X,S}})
\]

\noindent satisfying
\[
\Gamma_{\check{G},S}^{\on{QCoh,naive}}\simeq \on{oblv}\circ \Gamma_{\check{G},S}^{\on{QCoh}}: \on{QCoh}(V)\to \on{QCoh}(\on{Jets}(\bB\check{G},\bB\check{N})),
\]

\noindent where
\[
\on{oblv}: \on{Rep}((\check{I}_S)_{\on{Ran}_{X,S}})\to \on{QCoh}(\on{Jets}(\bB\check{G},\bB\check{N}))
\]

\noindent is the tautological functor.

\subsubsection{} Similar to Proposition \ref{p:ff}, we have the following:
\begin{prop}\label{p:ppanff}
The functor $\Gamma_{\check{G},S}^{\on{QCoh}}$ is fully faithful.
\end{prop}
\begin{rem}
This is not true for the functor $\Gamma_{\check{G},S}^{\on{QCoh,naive}}$. The problem is the failure of $\on{Jets}(\bB \check{G},\bB\check{N})$ to be \emph{passable} in the sense of \cite[§3]{gaitsgory2019study}.
\end{rem}

\begin{proof}
\step 
First, we have a factorization category
\[
\on{Rep}((\check{I}_S)_{\on{Ran}_{X}})
\]

\noindent over $\on{Ran}_{X_{\on{dR}}}$ such that
\[
\on{Rep}((\check{I}_S)_{\on{Ran}_{X,S}})\simeq \on{Rep}((\check{I}_S)_{\on{Ran}_{X}})\underset{D(\on{Ran}_{X})}{\otimes} D(\on{Ran}_{X,S}).
\]

\noindent Indeed, let
\[
(\check{I}_S)_{\on{Ran}_{X}}
\]

\noindent be the group scheme over $\on{Ran}_{X_{\on{dR}}}$ such that
\[
(\check{I}_S)_{\on{Ran}_{X,S}}=(\check{I}_S)_{\on{Ran}_{X}}\underset{\on{Ran}_{X_{\on{dR}}}}{\times}\on{Ran}_{X_{\on{dR}},S}.
\]

\noindent Then 
\[
(\check{I}_S)_{\on{Ran}_{X}}=\underset{n}{\on{lim}} \;(\check{I}^{(n)}_S)_{\on{Ran}_{X}}
\]

\noindent is similarly the inverse limit of finite type group schemes over $\on{Ran}_{X_{\on{dR}}}$. Let
\[
\on{Rep}((\check{I}_S)_{\on{Ran}_{X}}):=\underset{n}{\on{colim}}\on{QCoh}(\bB (\check{I}^{(n)}_S)_{\on{Ran}_{X}}).
\]

\step The functor
\[
\on{Rep}((\check{I}_S)_{\on{Ran}_{X}})\to \on{Rep}((\check{I}_S)_{\on{Ran}_{X,S}})
\]

\noindent induces an equivalence
\[
(\on{Rep}((\check{I}_S)_{\on{Ran}_{X}}))_{\on{indep}}\xrightarrow{\simeq} (\on{Rep}((\check{I}_S)_{\on{Ran}_{X,S}}))_{\on{indep}}.
\]

\noindent Moreover, the functor
\[
\Gamma_{\check{G},S}^{\on{QCoh}}: \on{QCoh}(V)\to \on{Rep}((\check{I}_S)_{\on{Ran}_{X,S}})
\]

\noindent naturally factors as 
\[
\on{QCoh}(V)\to (\on{Rep}((\check{I}_S)_{\on{Ran}_{X}}))_{\on{indep}}\xrightarrow{\simeq} (\on{Rep}((\check{I}_S)_{\on{Ran}_{X,S}}))_{\on{indep}}\to \on{Rep}((\check{I}_S)_{\on{Ran}_{X,S}}).
\]

\noindent Thus, it suffices to show that the functor
\[
\on{QCoh}(V)\to (\on{Rep}((\check{I}_S)_{\on{Ran}_{X}}))_{\on{indep}}
\]

\noindent is fully faithful.

However, this follows from (the proof of) \cite[Cor. C.1.8]{ambidex}. Indeed, we need to verify the conditions of §C.1.1 in \emph{loc.cit}:
\begin{itemize}
    \item The diagonal map
    \[
    \bB(\check{I}_S)_{\on{Ran}_{X}}\to \bB (\check{I}_S)_{\on{Ran}_{X}}\underset{\on{Ran}_{X_{\on{dR}}}}{\times} \bB (\check{I}_S)_{\on{Ran}_{X}}
    \]

    \noindent is clearly affine.

    \item The category $\on{Rep}((\check{I}_S)_{\on{Ran}_{X}})$ is rigged to be compactly generated, and such that the unit is compact.

    \item Since $V$ is a quasi-compact algebraic stack, it is passable.
\end{itemize}
\end{proof}

\subsubsection{} Recall the notation from \S \ref{s:group}. Define the group scheme 
\[
(\overset{\circ}{\check{I}_S})_{\on{Ran}_{X,S}}:=\check{G}(O)_{\on{Ran}_{X,S}}\underset{\underset{x\in S}{\prod}\check{G}(O_x)}{\times}\underset{x\in S}{\prod}\overset{\circ}{\check{I}_x}.
\]

\noindent We have a strong action 
\[
(\overset{\circ}{\check{I}_S})_{\on{Ran}_{X,S}}\curvearrowright \hat{\check{\fg}}_{\kappa,\on{Ran}_{X,S}}\on{-mod},
\]

\noindent and we may form the category
\[
\hat{\check{\fg}}_{\kappa,\on{Ran}_{X,S}}\on{-mod}^{(\overset{\circ}{\check{I}_S})_{\on{Ran}_{X,S}}}.
\]

\noindent For a $k$-point $\underline{x}\in \on{Ran}_{X_{\on{dR}},S}(k)$, the fiber of $\hat{\check{\fg}}_{\kappa,\on{Ran}_{X,S}}\on{-mod}^{(\overset{\circ}{\check{I}_S})_{\on{Ran}_{X,S}}}$ at $\underline{x}$ is given by 
\[
\bigotimes_{x\in \underline{x}\setminus S} \hat{\check{\fg}}_{\kappa,x}\on{-mod}^{\check{G}(O_x)}\otimes \bigotimes_{x'\in S} \hat{\check{\fg}}_{\kappa,x'}\on{-mod}^{\overset{\circ}{\check{I}_{x'}}}.
\]

\subsubsection{} By construction of \cite{Cstterm}, we have a forgetful functor 
\[
\on{\textbf{oblv}}: \hat{\check{\fg}}_{\kappa,\on{Ran}_{X,S}}\on{-mod}^{(\overset{\circ}{\check{I}_S})_{\on{Ran}_{X,S}}}\to \on{Rep}((\check{I}_S)_{\on{Ran}_{X,S}}),
\]

\noindent which over $\underline{x}\in \on{Ran}_{X_{\on{dR},S}}(k)$ restricts to the usual forgetful functor
\[
\bigotimes_{x\in \underline{x}\setminus S} \hat{\check{\fg}}_{\kappa,x}\on{-mod}^{\check{G}(O_x)}\otimes\bigotimes_{x'\in S} \hat{\check{\fg}}_{\kappa,x'}\on{-mod}^{\overset{\circ}{\check{I}_{x'}}}\to 
\bigotimes_{x\in \underline{x}\setminus S} \on{Rep}(\check{G}(O_x))\otimes\bigotimes_{x\in S} \on{Rep}(\overset{\circ}{\check{I}_x}).
\]

\noindent Similarly, we have a natural induction functor 
\[
\on{\textbf{ind}}: \on{Rep}((\check{I}_S)_{\on{Ran}_{X,S}})\to \hat{\check{\fg}}_{\kappa,\on{Ran}_{X,S}}\on{-mod}^{(\overset{\circ}{\check{I}_S})_{\on{Ran}_{X,S}}},
\]

\noindent which is left adjoint to $\on{\textbf{oblv}}$. Since $\on{\textbf{oblv}}$ is conservative, $\on{\textbf{ind}}$ generates the target under colimits.

\subsubsection{} Recall the localization functor from §\ref{s:kmloc1}:
\[
\on{Loc}_{\kappa,\check{G},S}^{\check{N}}: \hat{\check{\mathfrak{g}}}_{\kappa,\textrm{Ran}_{X,S}}\textrm{-mod}^{(\overset{\circ}{\check{I}_S})_{\textrm{Ran}_{X,S}}}\to D_{\kappa}(\BunN).
\]

\noindent By slight abuse of notation, we also denote by $\on{Loc}_{\kappa,\check{G},S}^{\check{N}}$ the composition
\[ \hat{\check{\mathfrak{g}}}_{\kappa,\textrm{Ran}_{X,S}}\textrm{-mod}^{(\overset{\circ}{\check{I}_S})_{\textrm{Ran}_{X,S}}}\to D_{\kappa}(\BunN)\xrightarrow{j^!}D_{\kappa}(V).
\]

\subsubsection{} The main compatibility between the above functors is captured by the commutative diagram:
\[\begin{tikzcd}\label{d:loc}
	{\hat{\check{\mathfrak{g}}}_{\kappa,\textrm{Ran}_{X,S}}\textrm{-mod}^{(\overset{\circ}{\check{I}_S})_{\textrm{Ran}_{X,S}}}} && {D_{\kappa}(V)} \\
	\\
	{\on{Rep}((\check{I}_S)_{\on{Ran}_{X,S}})} && {\textrm{QCoh}(V).}
	\arrow["{\textrm{Loc}^{\check{N}}_{\kappa,\check{G},S}}", from=1-1, to=1-3]
	\arrow["{\textrm{\textbf{ind}}}", from=3-1, to=1-1]
	\arrow["{\on{Loc}^{\on{QCoh}}_{\check{G},S}}", from=3-1, to=3-3]
	\arrow["{\textrm{ind}_{\textrm{dR}}}"', from=3-3, to=1-3]
\end{tikzcd}\]

\begin{cor}\label{c:pc}
Let $j:V\subset \BunN$ be a quasi-compact open substack. Then the functor
\[
\on{Loc}^{\check{N}}_{\kappa,\check{G},S}: \hat{\check{\mathfrak{g}}}_{\kappa,\on{Ran}_{X,S}}\on{-mod}^{(\overset{\circ}{\check{I}}_S)_{\on{Ran}_{X,S}}}\to D(V).
\]

\noindent preserves compact objects and generates the target under colimits.
\end{cor}

\begin{proof}
Since $\on{\textbf{ind}}$ generates the target under colimits and preserves compact objects, the assertion follows from the diagram in §3.7.20 using the fact that
\[
\on{Rep}((\check{I}_S)_{\on{Ran}_{X,S}})\xrightarrow{\on{Loc}_{\check{G},S}^{\on{QCoh}}} \on{QCoh}(V)
\]

\noindent preserves compact objects and generates the target under colimits, cf. Proposition \ref{p:ppanff}.

\end{proof}

\subsubsection{} We denote by $\Gamma_{\kappa,\check{G},S}^{\check{N}}$ the right adjoint to $\on{Loc}_{\kappa,\check{G},S}^{\check{N}}$:
\[
\Gamma_{\kappa,\check{G},S}^{\check{N}}: D_{\kappa}(\BunN)\to \hat{\check{\mathfrak{g}}}_{\kappa,\textrm{Ran}_{X,S}}\textrm{-mod}^{(\overset{\circ}{\check{I}}_S)_{\textrm{Ran}_{X,S}}}. 
\]

\noindent From Corollary \ref{c:pc}, we see that the composition 
\begin{equation}\label{eq:compff?}
D_{\kappa}(V)\xrightarrow{j_{*,\on{dR}}} D_{\kappa}(\BunN)\xrightarrow{\Gamma_{\kappa,\check{G},S}^{\check{N}}} \hat{\check{\mathfrak{g}}}_{\kappa,\textrm{Ran}_{X,S}}\textrm{-mod}^{(\overset{\circ}{\check{I}}_S)_{\textrm{Ran}_{X,S}}}
\end{equation}

\noindent is continuous and conservative.

\begin{rem}
Like the unramified situation (see \cite[Prop. 2.14]{gaitsgory2010generalized}), the composition (\ref{eq:compff?}) should be fully faithful. We will not need this, however.
\end{rem}

\subsection{The spectral decomposition}\label{s:spectraldecomp3}

We are now ready to prove Theorem \ref{t:action}. Let $\kappa=\on{crit}$ be the critical level of $\check{\fg}$.

\subsubsection{} Recall from Proposition \ref{p:ff} that we have a fully faithful continuous embedding
\begin{equation}\label{eq:ffagain}
\on{QCoh}(\LSN)\into \NR.
\end{equation}

\subsubsection{} Consider the full subcategory category 
\[
D_{*-\on{gen}}(\BunN)^{\check{T}_S,\chi_S}\subset \Autr
\]

\noindent defined as the cocompletion of the essential images $j_{*,\on{dR}}(D(V)^{\check{T}_S,\chi_S})$, where $j$ runs over the poset of $\check{T}_S$-stable quasi-compact open substacks $V$ of $\BunN$. 

\begin{prop}\label{p:embe}
    The category $D_{*-\on{gen}}(\BunN)^{\check{T}_S,\chi_S}$ lies in the image of the fully faithful embedding
    \begin{equation}\label{eq:emb}
     \on{QCoh}(\LSN)\underset{\NR}{\otimes} \Autr\into \Autr.   
    \end{equation}
\end{prop}

\begin{proof}
Let $V\subset \BunN$ be a quasi-compact open substack, and let $\cF\in D(V)$. Let $\cK\in\NR$ be an object in the kernel of the projection
\[
\NR\to \on{QCoh}(\LSN).
\]

\noindent We need to show that $\cK\star j_{*,\on{dR}}(\cF)=0$. By Corollary \ref{c:pc}, it suffices to show that the adjunction map
\[
\cK\star \on{Loc}_{\on{crit},\check{G},S}\circ \Gamma_{\on{crit},\check{G},S}(\cF)\simeq \on{Loc}_{\on{crit},\check{G},S}\circ \Gamma_{\on{crit},\check{G},S}(\cK\star j_{*,\on{dR}}(\cF))\to \cK\star j_{*,\on{dR}}(\cF)
\]

\noindent is zero. However, we know from Proposition \ref{p:km} that the image of 
\[
\on{Loc}_{\on{crit},\check{G},\chi_S}: (\hat{\check{\fg}}_{\on{crit},\on{Ran}_{X,S}}\on{-mod})^{(\check{I}_S)_{\on{Ran}_{X,S}}, \chi_S}\to D(\BunN)^{\check{T}_S,\chi_S}
\]

\noindent lies in the image of (\ref{eq:emb}).
\end{proof}

\subsubsection{} Let us now finish the proof of the spectral decomposition.

\begin{thm}\label{t:action'}
The action of $\NR$ on $\Autr$ factors through a unique action of $\on{QCoh}(\LSN)$.
\end{thm}

\begin{proof}
We have to prove that the fully faithful continuous embedding
\[
\on{QCoh}(\LSN)\underset{\NR}{\otimes} \Autr\into \Autr
\]

\noindent is an equivalence. Since the above embedding admits a left adjoint, it follows that it commutes with cofiltered limits. Next, write $\BunN$ as a union of quasi-compact open substacks $V$ each of which are $\check{T}_S$-stable.\footnote{For example, write $\on{Bun}_G$ as a union of quasi-compact open substacks and pull back these under the forgetful map $\BunN\to \on{Bun}_{\check{G}}(X)$.} For any $\cF\in \Autr$, there exists some quasi-compact open $j:V\into \BunN$ such that $j^{*,\on{dR}}(\cF)\neq 0$. This formally implies that the category $\Autr$ is generated under cofiltered limits of sheaves of the form $j_{*,\on{dR}}(\cG)$ for $\cG\in D(V)^{\check{T}_S,\chi_S}$. The theorem now follows from Proposition \ref{p:embe}.

\end{proof}

\subsection{Spectral decomposition with fixed abelianization}\label{s:twist}

\subsubsection{} Recall the notation from §\ref{s:2.3.2}. Theorem \ref{t:action'} gives a spectral decomposition of the category $D(\BunN)^{\check{T}_S,\chi_S}$ over $\LSN$. We now see how this provides a spectral decomposition of a suitable automorphic category over $\LS$.

\subsubsection{} From the map 
\[
\on{ab}: \LSN\to \LSBA
\]

\noindent and Theorem \ref{t:action'}, we obtain an action
\begin{equation}\label{eq:abact1}
\on{QCoh}(\LSBA)\curvearrowright D(\BunN)^{\check{T}_S,\chi_S}.
\end{equation}

\subsubsection{}\label{s:centralstack} Denote by $Z(\check{G})^{\circ}$ the connected component of the unit in the center of $\check{G}$. We denote by
\[
\on{Bun}_{Z(\check{G})^{\circ},1}(X,S)
\]

\noindent the moduli stack of $Z(\check{G})^{\circ}$-bundles on $X$ equipped with a trivialization on $S\subset X$. Note that $\on{Bun}_{Z(\check{G})^{\circ},1}(X,S)$ is naturally a commutative group stack. In particular, $D(\on{Bun}_{Z(\check{G})^{\circ},1}(X,S))$ carries a symmetric monoidal structure and is moreover equipped with an action of $D(Z(\check{G})^{\circ})$ at each $x\in S$. We let 
\[
D(\on{Bun}_{Z(\check{G})^{\circ},1}(X,S))^{Z(\check{G})^{\circ}_S,\chi^{\on{ab}}_S}
\]

\noindent be the category of $(\prod_{x\in S}Z(\check{G})^{\circ},\underset{x\in S}{\boxtimes}\chi^{\on{ab}}_x)$-equivariant objects of $D(\on{Bun}_{Z(\check{G})^{\circ},1}(X,S))$. Note that it is naturally symmetric monoidal.

We have a natural action
\[
\on{Bun}_{Z(\check{G})^{\circ},1}(X,S)\curvearrowright \BunN
\]

\noindent that descends to an action
\[
D(\on{Bun}_{Z(\check{G})^{\circ},1}(X,S))^{Z(\check{G})^{\circ}_S,\chi^{\on{ab}}_S}\curvearrowright D(\BunN)^{\check{T}_S,\chi_S}.
\]

\noindent

\subsubsection{} We have a symmetric monoidal Fourier-Mukai equivalence
\begin{equation}\label{eq:fm}
D(\on{Bun}_{Z(\check{G})^{\circ},1}(X,S))^{Z(\check{G})^{\circ}_S,\chi^{\on{ab}}_S}\simeq \on{QCoh}(\LSBA).
\end{equation}

\noindent Via this equivalence, we also obtain an action
\begin{equation}\label{eq:abact2}
\on{QCoh}(\LSBA)\curvearrowright D(\BunN)^{\check{T}_S,\chi_S}.
\end{equation}

\noindent By construction (see e.g. \cite[\S 11.3.6]{faergeman2022non}), the two actions (\ref{eq:abact1}), (\ref{eq:abact2}) coincide.

\subsubsection{}\label{s:charsheaf} Now, let $\tau\in \LSBA(k)$. Under (\ref{eq:fm}), the delta sheaf corresponding to $\tau$ goes to a character sheaf $\sL_{\tau}\in D(\on{Bun}_{Z(\check{G})^{\circ},1}(X,S))^{Z(\check{G})^{\circ}_S,\chi^{\on{ab}}_S}$. We then have an equivalence
\[
D(\BunN)^{(\check{T}_S,\chi_S),\sL_{\tau}}\simeq D(\BunN)^{\check{T}_S,\chi_S}\underset{\on{QCoh}(\LSBA)}{\otimes} \on{Vect},
\]

\noindent where the action $\on{QCoh}(\LSBA)\curvearrowright \on{Vect}$ is induced by $\tau$, and 
\[
D(\BunN)^{(\check{T}_S,\chi_S),\sL_{\tau}}
\]

\noindent denotes the category of $\sL_{\tau}$-equivariant objects of $D(\BunN)^{\check{T}_S,\chi_S}$. Recall that
\[
\LS=\LSN\underset{\LSBA}{\times} \lbrace\tau\rbrace.
\]

\noindent Hence from Theorem \ref{t:action'}, we obtain an action
\begin{equation}\label{eq:acttwist}
    \on{QCoh}(\LS)\curvearrowright D(\BunN)^{(\check{T}_S,\chi_S),\sL_{\tau}}.
\end{equation}

\section{Existence of Hecke eigensheaves}

In this section, we prove the existence of Hecke eigensheaves associated to irreducible $G$-local system with regular singularities. The proof follows closely that of \cite[Appendix A]{faergeman2022non} where the unramified situation is considered. In particular we will use Arinkin's result \cite{arinkin2016irreducible} that irreducible $G$-local systems on $U$ generically admit oper structures and obtain the corresponding Hecke eigensheaf via Kac-Moody localization as pioneered by Beilinson-Drinfeld.

We will need a generalization of \cite{frenkel2005fusion} to the situation of non-unipotent local systems. More precisely, \emph{loc.cit} shows that passing from the category of quasi-coherent sheaves on the affine scheme of regular singular opers on the punctured disk with nilpotent residues to the category of Iwahori-integrable Kac-Moody modules at critical level for the dual group coming from Feigin-Frenkel duality is compatible with Bezrukavnikov's geometric realization of the affine Hecke algebra. We need a similar compatibility between opers with residues of some fixed eigenvalues and Theorem \ref{t:metaplectic}. This compatibility is current work in progress joint with Gurbir Dhillon.

\subsection{Opers with regular singularities}\label{s:4.1} 

\subsubsection{} We fix a square root $\omega^{\frac{1}{2}}$ of the canonical sheaf on $X$. We set $\rho(\omega_X(S)):=2\rho(\omega_X^{\frac{1}{2}})(S)$.

\subsubsection{} Let $\on{Op}_{G}(\overset{\circ}{D})$ be the ind-affine ind-scheme of $G$-opers on the punctured disk, cf. \cite[§A]{bezrukavnikov2016quantization} (see also \cite{barlevoper}). Let $\on{Op}_{G}(\overset{\circ}{D})^{\on{RS}}$ be the closed subscheme of opers with regular singularities defined in \cite[§2.4]{frenkel2006local}. Explicitly, we have:
\[
\on{Op}_{G}(\overset{\circ}{D})=(f+\fb((t)))dt/N(K),
\]

\noindent where the quotient is taken with respect to the gauge action. Here, $f$ is the sum of simple root generators for $\fn^{-}$. Similarly:
\begin{equation}\label{eq:oprspr}
\on{Op}_{G}(\overset{\circ}{D})^{\on{RS}}=t^{-1}(f+\fb[[t]])dt/N(O).
\end{equation}
\begin{rem}
Note that if $G\to H$ is an isogeny, the resulting map
\[
\on{Op}_{G}(\overset{\circ}{D})\to \on{Op}_{H}(\overset{\circ}{D})
\]

\noindent is an isomorphism.
\end{rem}

\subsubsection{} We have an evident forgetful map
\begin{equation}\label{eq:optoloc}
\on{Op}_{G}(\overset{\circ}{D})\to \on{LocSys}_{G}(\overset{\circ}{D}).
\end{equation}

\subsubsection{}

There is a natural residue map (see \cite[\S 2.4]{frenkel2006local} and \cite[\S 3.8.11]{beilinson1991quantization}):
\[
\on{Res}^{\on{RS}}: \on{Op}_{G}(\overset{\circ}{D})^{\on{RS}}\to\fg/G\to \ft//W.
\]

\noindent For an oper of the form $\nabla=d+t^{-1}(f+\phi(t))dt$ with $\phi(t)\in\fb[[t]]$, the above map attaches the image of the element $f+\phi(0)$ under $\fg/G\to \ft//W$.

\subsubsection{} Let $\varpi$ be the projection $\ft\to \ft//W$. For $\lambda\in\ft$, we denote by $\on{Op}_{G}(\overset{\circ}{D}_x)^{\on{RS},\varpi(\lambda)}$ the fiber of $\varpi(\lambda)$ under $\on{Res}^{\on{RS}}$. The following is an immediate consequence of Lemma \ref{l:cstresidue}:
\begin{lem}
Let $\chi\in\ft$. Then the map
\[
\on{Op}_{G}(\overset{\circ}{D})^{\on{RS},\varpi(-\chi-\check{\rho})}\to \on{LocSys}_G(\overset{\circ}{D})
\]

\noindent factors through
\[
\on{LocSys}_G(\overset{\circ}{D})_{\chi}^{\on{RS}}\to \on{LocSys}_G(\overset{\circ}{D}).
\]
\end{lem}

\subsubsection{} Consider the critical level $\kappa=\on{crit}$ for $\check{\fg}$. For $\mu\in \ft\simeq (\ft^*)^*$, we have the affine Verma module
$\bM_{\mu}=\on{ind}_{\hat{\check{\fg}}[[t]]\oplus k\mathbf{1}}^{\hat{\check{\fg}}_{\on{crit}}}(M_{\mu})$, where $M_{\mu}$ is the usual Verma module for $\check{\fg}$ with highest weight $\mu$. Denote by $\chi_{\mu}$ the image of $\mu$ under the quotient map $\ft\to \ft/X_{\bullet}(T)$ and note that $\bM_{\mu}$ is $(\check{I},\chi_{\mu})$-equivariant. 

\subsubsection{} By \cite[Corollary 13.3.2]{frenkel2006local}, we have an isomorphism of algebras
\[
\on{Fun}(\on{Op}_{G}(\overset{\circ}{D})^{\on{RS},\varpi(-\mu-\check{\rho})})\simeq \on{End}_{{\hat{\check{\fg}}}_{\on{crit}}}(\bM_{\mu}).
\]

\noindent Write $A_{-\mu-\check{\rho}}$ for the above algebra. We get a functor
\begin{equation}\label{eq:critfle}
\on{QCoh}(\on{Op}_{G}(\overset{\circ}{D})^{\on{RS},\varpi(-\mu-\check{\rho})})\to \hat{\check{\fg}}_{\on{crit}}\on{-mod}^{\check{I},\chi_{\mu}}, \:\: N\mapsto \bM_{\mu}\underset{A_{-\mu-\check{\rho}}}{\otimes} N,
\end{equation}

\noindent By \cite[Cor. 13.3.2]{frenkel2006local}, the functor (\ref{eq:critfle}) is t-exact.

\subsection{Geometric twists and Drinfeld-Sokolov reduction}

\subsubsection{} Let $\check{N}^{-,\omega(S)}$ denote the group scheme over $X$ of endomorphisms of $\rho(\omega_X(S))\overset{\check{T}}{\times} \check{B}^{-}$ preserving the identification $(\rho(\omega_X(S))\overset{\check{T}}{\times} \check{B}^{-})/\check{N}^{-}\simeq \rho(\omega_X(S))$. Let
\[
\check{N}^{-,\omega(S)}(O_x)=\Gamma(D_x,\check{N}^{-,\omega(S)}_{\vert D_x}),  \; \check{N}^{-,\omega(S)}(K_x)=\Gamma(\overset{\circ}{D}_x,\check{N}^{-,\omega(S)}_{\vert \overset{\circ}{D}_x}).
\]

\noindent Write 
\[
\check{\fn}^{-,\omega(S)}(O_x)=\on{Lie}(\check{N}^{-,\omega(S)}(O_x)),\;\; \check{\fn}^{-,\omega(S)}(K_x)=\on{Lie}(\check{N}^{-,\omega(S)}(K_x)).
\]

\noindent For each $x\in X$, we choose a uniformizer $t=t_x$ of $O_x$. Note that we have canonical identifications
\[
\check{N}^{-,\omega(S)}(K_x)=\check{N}^{-,\omega}(K_x), \;\; \check{\fn}^{-,\omega(S)}(K_x)=\check{\fn}^{-,\omega}(K_x).
\]

\subsubsection{} We let $\check{N}_{(1)}^{-,\omega(S)}(O_x)$
be the first congruence subgroup of $\check{N}^{-,\omega(S)}(O_x)$. That is, the subgroup of $\check{N}^{-,\omega(S)}(O_x)$ consisting of endomorphisms that restrict to the identity on $x\in D_x$. Let
\[
\check{\fn}_{(1)}^{-,\omega(S)}(O_x):=\on{Lie}(\check{N}_{(1)}^{-,\omega(S)}(O_x)).
\]

\subsubsection{} For $x\in X$, let $\check{\fn}^{-,\omega(S)}(K_x)\on{-mod}$ be the DG category of smooth modules for $\check{\fn}^{-,\omega(S)}(K_x)$ as in \cite{raskin2020homological}.

\subsubsection{}\label{s:grpscheme} Let $\check{G}^{\omega(S)}$ be the group scheme over $X$ of endomorphisms of $\rho(\omega_X(S))\overset{\check{T}}{\times} \check{G}$. Let $\on{Bun}_{\check{G}^{\omega(S)}}(X)$ be the stack of $\check{G}^{\omega(S)}$-bundles on $X$, and let $\on{Bun}_{\check{G}^{\omega(S)},\check{N}^{\omega(S)}}(X,S)$ be the moduli stack of $\check{G}^{\omega(S)}$-bundles on $X$ equipped with a reduction to $\check{N}^{\omega(S)}$ on $S\subset X$. We have canonical identifications
\[
\on{Bun}_{\check{G}^{\omega(S)}}(X)\simeq \on{Bun}_{\check{G}}(X), \;\; \on{Bun}_{\check{G}^{\omega(S)},\check{N}^{\omega(S)}}(X,S)\simeq \on{Bun}_{\check{G},\check{N}}(X,S).
\]

\subsubsection{} Let $\check{\fg}^{\omega(S)}=\on{Lie}(\check{G}^{\omega(S)})$. We may similarly define the category
\begin{equation}
\hat{\check{\fg}}^{\rho(\omega_X(S))}_{\on{crit},x}\on{-mod}
\end{equation}

\noindent of modules for $\hat{\check{\fg}}_{\on{crit}}^{\rho(\omega_X(S))}$. We have an equivalence
\begin{equation}\label{eq:twistcrit}
\hat{\check{\fg}}^{\rho(\omega_X(S))}_{\on{crit},x}\on{-mod}\simeq \hat{\check{\fg}}_{\on{crit},x}\on{-mod}.
\end{equation}

\noindent We refer to \cite{Cstterm} for a detailed discussion of equivalences (\ref{eq:twistcrit}) arising from geometric twists.

\subsubsection{}\label{s:minusDS} We consider the minus Drinfeld-Sokolov reduction functor studied in \cite{arakawa2004vanishing}. For $x\in S$, it is given by
\begin{equation}
\Psi_x^{-}(-)=H^{\frac{\infty}{2}+\bullet}(\check{\fn}^{-,\omega(S)}(K_x),\check{\fn}_{(1)}^{-,\omega(S)}(O_x),-\otimes \psi_x^{-}): \hat{\check{\fg}}_{\on{crit},x}\on{-mod}\to \on{Vect}.
\end{equation}

\noindent Here, $\psi_x^{-}$ is the standard character of $\check{\fn}^{-,\omega(S)}(K_x)$ of conductor 1.\footnote{More precisely, there is a standard character of $\check{\fn}^{-,\omega(S)}(K_x)=\check{\fn}^{-,\omega}(K_x)$ of conductor 0. Since for $x\in S$, sections of $\omega$ over $D_x$ naturally identify with section of $\omega(S)$ over $D_x$ that vanish at the origin, we obtain a character of $\check{\fn}^{-,\omega(S)}(K_x)$ that vanishes on $\check{\fn}_{(1)}^{-,\omega(S)}(O_x)$.} 

Implicit in the definition of $\Psi_x^{-}$ is the restriction functor
\[
\hat{\check{\fg}}_{\on{crit},x}\on{-mod}\to \check{\fn}^{-,\omega(S)}(K_x)\on{-mod}.
\]

\noindent This uses the identification (\ref{eq:twistcrit}), which we henceforth suppress.

\subsubsection{} By \cite[Cor. 7.2.4]{raskin2021w}, the minus Drinfeld-Sokolov reduction functor is t-exact on the category
\[
\hat{\check{\fg}}_{\on{crit},x}\on{-mod}^{\overset{\circ}{\check{I}_x}}.
\]

\subsubsection{} For $\lambda\in \ft$, we consider the exact functor
\[
\on{QCoh}(\on{Op}_G(\overset{\circ}{D}_x)^{\on{RS},\varpi(-\lambda-\check{\rho})})\to \hat{\check{\fg}}_{\on{crit},x}\on{-mod}^{\check{I}_x,\chi_x}, \:\: N\mapsto \bM_{\lambda}\underset{A_{-\lambda-\check{\rho}}}{\otimes} N.
\]

\noindent Denote by $\chi_{\lambda}$ the image of $\lambda$ under $\ft\to \ft//\widetilde{W}^{\on{aff}}$. The following is folklore.
\begin{lem}\label{l:globsectDS}
The diagram below commutes:
\[\begin{tikzcd}
	{\mathrm{QCoh}(\mathrm{Op}_{G}(\overset{\circ}{D}_x)^{\mathrm{RS},\varpi(-\lambda-\check{\rho})})} && {\hat{\check{\mathfrak{g}}}_{\mathrm{crit},x}\mathrm{-mod}^{\check{I}_x,\chi_{\lambda}}} \\
	\\
	&& {\mathrm{Vect}.}
	\arrow[from=1-1, to=1-3]
	\arrow["{\Psi_x^{-}}", from=1-3, to=3-3]
	\arrow["{\Gamma(\mathrm{Op}_G(\overset{\circ}{D}_x)^{\mathrm{RS},\varpi(-\lambda-\check{\rho})},-)}"', from=1-1, to=3-3]
\end{tikzcd}\]
\end{lem}

\begin{proof}
We need to show that
\[
\Psi_x^{-}(\bM_{\lambda})=\on{Fun}(\mathrm{Op}_{G}(\overset{\circ}{D}_x)^{\mathrm{RS},\varpi(-\lambda-\check{\rho})}).
\]

\noindent Note that $\Psi_x^{-}$ naturally enhances to a functor
\[
\Psi_x^{-}: \hat{\check{\mathfrak{g}}}_{\mathrm{crit},x}\mathrm{-mod}^{\check{I}_x,\chi_{\lambda}}\to \cW_{\check{\fg},\on{crit}}\on{-mod},
\]

\noindent where $\cW_{\check{\fg},\on{crit}}$ denotes the affine $W$-algebra at critical level associated to $\check{\fg}$. The key observation, cf. \cite{arakawa2004vanishing} (see also \cite[Thm. 5.2.1]{dhillon2023localization})\footnote{Strictly speaking, the paper \cite{arakawa2004vanishing} does not work at critical level. \cite{dhillon2023localization} does.} is that $\Psi_x^{-}$ sends Verma modules to Verma modules. That is,
\[
\Psi_x^{-}(\bM_{\lambda})=\mathbf{M}_{\lambda}.
\]

\noindent Here, $\mathbf{M}_{\lambda}$ is the Verma module for $\cW_{\check{\fg},\on{crit}}$ with highest weight $\lambda$, see \cite[§5]{arakawa2004vanishing}. Using Feigin-Frenkel duality, we identify
\begin{equation}\label{eq:FF}
\cW_{\check{\fg},\on{crit}}\simeq \on{Fun}(\on{Op}_{G}(\overset{\circ}{D}_x)).
\end{equation}

\noindent The algebras in (\ref{eq:FF}) carry natural $\bZ$-gradings induced by the operator $L_0=-t\partial_t$ \cite[§9.4]{frenkel2007langlands}, and the isomorphism is compatible with the grading. Consider the module
\begin{equation}\label{eq:Zhu}
\on{Fun}(\on{Op}_{G}(\overset{\circ}{D}_x))\underset{\on{Fun}(\on{Op}_{G}(\overset{\circ}{D}_x))_{\leq 0}}{\otimes} \on{Fun}(\on{Op}_{G}(\overset{\circ}{D}_x))_0,
\end{equation}

\noindent where we identify
\[
\on{Fun}(\on{Op}_{G}(\overset{\circ}{D}_x))_0\simeq \on{Fun}(\on{Op}_{G}(\overset{\circ}{D}_x))_{\leq 0}/\on{Fun}(\on{Op}_{G}(\overset{\circ}{D}_x))_{<0}.
\]

\noindent By \cite[§9.4]{frenkel2007langlands}, we have an isomorphism
\[
\on{Fun}(\on{Op}_{G}(\overset{\circ}{D}_x)^{\on{RS}})\simeq \on{Fun}(\on{Op}_{G}(\overset{\circ}{D}_x))\underset{\on{Fun}(\on{Op}_{G}(\overset{\circ}{D}_x))_{\leq 0}}{\otimes} \on{Fun}(\on{Op}_{G}(\overset{\circ}{D}_x))_0.
\]

\noindent In particular, for any $\lambda\in\ft$, we get
\[
\textbf{M}_{\lambda}=\on{Fun}(\on{Op}_{G}(\overset{\circ}{D}_x))\underset{\on{Fun}(\on{Op}_{G}(\overset{\circ}{D}_x))_{\leq 0}}{\otimes} \big(\on{Fun}(\on{Op}_{G}(\overset{\circ}{D}_x))_0\underset{\on{Fun}(\on{Op}_{G}(\overset{\circ}{D}_x))_0}{\otimes} k_{-\lambda-\check{\rho}}\big)
\]
\[
\simeq \on{Fun}(\on{Op}_{G}(\overset{\circ}{D}_x)^{\on{RS}})\underset{\on{Fun}(\on{Op}_{G}(\overset{\circ}{D}_x))_0}{\otimes} k_{-\lambda-\check{\rho}}\simeq \on{Fun}(\on{Op}_{G}(\overset{\circ}{D}_x)^{\on{RS},\varpi(-\lambda-\check{\rho})}).
\]

\noindent Here we use the fact (cf. \cite[Prop. 9.4.1]{frenkel2007langlands}) that the residue map
\[
\on{Op}_{G}(\overset{\circ}{D}_x)^{\on{RS}}\to \ft//W
\]

\noindent induces an isomorphism
\begin{equation}\label{eq:resiso}
\on{Fun}(\ft)^W\simeq \on{Fun}(\on{Op}_{G}(\overset{\circ}{D}_x)^{\on{RS}})_0
\end{equation}

\noindent such that the composition
\[
\on{Fun}(\ft)^W=Z(U(\fg))\simeq \on{Fun}(\on{Op}_{G}(\overset{\circ}{D}_x)^{\on{RS}})_0\xrightarrow{(\ref{eq:resiso})} \on{Fun}(\ft)^W
\]

\noindent sends $f$ to $f^{-}=(\lambda\mapsto f(-\lambda))$. Here, the first isomorphism in the composition comes from Feigin-Frenkel duality.
\end{proof}

\subsubsection{}\label{ss:workinprogress} As mentioned in the introduction of this section, we need the following result:
\begin{conjecture}\label{c:linearity}
The functor
\[
\on{QCoh}(\on{Op}_{G}(\overset{\circ}{D})^{\on{RS},\varpi(-\lambda-\check{\rho})})\to \hat{\check{\fg}}_{\on{crit}}\on{-mod}^{\check{I},\chi_{\lambda}}
\]

\noindent is $\on{QCoh}(\on{LocSys}_G(\overset{\circ}{D})^{\on{RS}}_{\chi_{\lambda}})$-linear, where $\on{QCoh}(\on{LocSys}_G(\overset{\circ}{D})^{\on{RS}}_{\chi_{\lambda}})$ acts on $\hat{\check{\fg}}_{\on{crit}}\on{-mod}^{\check{I},\chi_{\lambda}}$ via Corollary \ref{c:spec1}.
\end{conjecture}

\noindent When $\lambda=0$, this is \cite[Theorem 5.4]{frenkel2005fusion}. In general, the conjecture is expected to follow from similar methods as those employed in \cite{frenkel2005fusion} and is currently work in progress joint with Gurbir Dhillon.

\subsection{The Whittaker functional and unipotent localization}\label{s:comp}

\subsubsection{} Consider the algebraic stack $\on{Bun}_{\check{N}^{-}}^{\omega_X(S)}(X,S)$ classifying triples $(\sP_{\check{B}^{-}},\tau,\phi)$, where 
\begin{itemize}

\item $\sP_{\check{B}^{-}}$ is a $\check{B}^{-}$-bundle on $X$.

\item $\tau$ is an isomorphism $\sP_{\check{B}^{-}}\overset{\check{B}^{-}}{\times} \check{T}\simeq \rho(\omega_X(S))$.

\item $\phi$ is an identification of $\sP_{\check{B}^{-}}$ with $\rho(\omega_X(S))\overset{\check{T}}{\times} \check{B}^{-}$ on $S\subset X$ that preservers $\tau_{\vert S}$.

\end{itemize}

\subsubsection{Whittaker functional} As in \cite[\S 2.5]{nadler2019geometric}, we have a canonical Whittaker functional
\begin{equation}\label{eq:whitfunction}
\psi^{-}: \on{Bun}_{\check{N}^{-}}^{\omega_X(S)}(X,S)\to \bG_a.
\end{equation}

\noindent Concretely, for every positive simple root, we have a map
\[
\on{Bun}_{\check{N}^{-}}^{\omega_X(S)}(X,S)\to \on{Bun}_{\bG_a}^{\omega_X(S)}(X,S),
\]

\noindent where the latter classifies extensions $\omega_X(S)\to \sE\to \cO_X$ with a splitting at each $x\in S$. The natural map
\[
H^1(X,\omega_X)\to \on{Bun}_{\bG_a}^{\omega_X(S)}(X,S)
\]

\noindent induced by $\omega_X\to \omega_X(S)$ is an isomorphism. Thus, the functional (\ref{eq:whitfunction}) is given by summing up over all positive simple roots:
\[
\on{Bun}_{\check{N}^{-}}^{\omega_X(S)}(X,S)\to \underset{i}{\prod} \on{Bun}_{\bG_a}^{\omega_X(S)}(X,S)\simeq \underset{i}{\prod} H^1(X,\omega_X)\xrightarrow{\on{sum}} \bA^1.
\]

\subsubsection{} For any $x\in S$, there is a natural map
\[
\check{N}^{-,\omega(S)}(K_x)\to \check{N}^{-,\omega(S)}(K_x)/\check{N}_{(1)}^{-,\omega(S)}(O_x)\to \on{Bun}_{\check{N}^{-}}^{\omega_X(S)}(X,S).
\]

\noindent By construction, the Whittaker functional is such that the composition 
\[
\check{N}^{-,\omega(S)}(K_x)\to \on{Bun}_{\check{N}^{-}}^{\omega_X(S)}(X,S)\to \bG_a
\]

\noindent coincides with the canonical lift of $\psi_x^{-}$ (see §\ref{s:minusDS}) to $\check{N}^{-,\omega(S)}(K_x)$.

\subsubsection{Localization in the unipotent setting} Consider the category
\[
\check{\fn}^{-,\omega(S)}(K_x)\on{-mod}^{\check{N}^{-,\omega(S)}(O_x)}
\]

\noindent of $\check{N}^{-,\omega(S)}(O_x)$-integrable $\check{\fn}^{-,\omega(S)}(K_x)\on{-modules}$. The category
\[
\check{\fn}^{-,\omega(S)}(K_x)\on{-mod}^{\check{N}_{(1)}^{-,\omega(S)}(O_x)}
\]

\noindent is defined similarly. We may identify the stack $\on{Bun}_{\check{N}^{-}}^{\omega_X(S)}(X,S)$ with the moduli stack of $\check{N}^{-,\omega(S)}$-bundles on $X$ equipped with a trivialization on $S\subset X$. As such, for $x\in S$, we have a localization functor
\[
\check{\fn}^{-,\omega(S)}(K_x)\on{-mod}^{\check{N}_{(1)}^{-,\omega(S)}(O_x)}\to D(\on{Bun}_{\check{N}^{-}}^{\omega_X(S)}(X,S)).
\]

\noindent In particular, we obtain a functor
\[
\on{Loc}_{\check{N}^{-},S}:\bigotimes_{x\in S} \check{\fn}^{-,\omega(S)}(K_x)\on{-mod}^{\check{N}_{(1)}^{-,\omega(S)}(O_x)}\to D(\on{Bun}_{\check{N}^{-}}^{\omega_X(S)}(X,S)).
\]

\subsubsection{} Let 
\[
\check{\fn}^{-,\omega(S)}(K)\on{-mod}_{\on{Ran},S}^{\check{N}_{(1)}^{-,\omega(S)}(O)}
\]

\noindent be the unital category over $\on{Ran}_{X_{\on{dR}},S}$ whose fiber at $x\in S$ is
\[
\check{\fn}^{-,\omega(S)}(K_x)\on{-mod}^{\check{N}_{(1)}^{-,\omega(S)}(O_x)}
\]

\noindent and whose fiber at $x\in X-S$ is
\[
\check{\fn}^{-,\omega(S)}(K_x)\on{-mod}^{\check{N}^{-,\omega(S)}(O_x)}.
\]

\subsubsection{} We also have a factorizable minus Drinfeld-Sokolov reduction functor
\begin{equation}\label{eq:DSenh}
\Psi^{-}_{\on{fact}}: \check{\fn}^{-,\omega(S)}(K)\on{-mod}_{\on{Ran},S}^{\check{N}_{(1)}^{-,\omega(S)}(O)}\to D(\on{Ran}_{X,S})\xrightarrow{C_{\on{dR}}(\on{Ran}_{X,S},-)} \on{Vect}
\end{equation}

\noindent induced by the usual factorizable Drinfeld-Sokolov reduction functor for $\check{\fn}^{-,\omega}$. Explicitly, for $x\in S$, the fiber of the first functor at $x$ is given by the functor defined in §\ref{s:minusDS}, and for $x\in X-S$ the fiber of the first functor is given by the usual Drinfeld-Sokolov reduction functor at $x$ for $\check{\fn}^{-,\omega}(K_x)\on{-mod}$ using that $\omega(S)_{\vert D_x}\simeq \omega_{\vert D_x}$.

\subsubsection{} We denote by $\on{ins}_S$ the functor
\[
\on{ins}_S: \bigotimes_{x\in S} \hat{\check{\mathfrak{g}}}_{\on{crit},x}\on{-mod}^{\check{I}_x,\chi_x}\to \hat{\check{\mathfrak{g}}}_{\on{crit},\on{Ran}_{X,S}}\on{-mod}^{({\check{I}}_S)_{\on{Ran}_{X,S}},\chi_S}
\]

\noindent that inserts the vacuum module away from $S$ as in \cite[Construction 4.4.5]{chen2021extension} (see also \cite[Appendix A 1.6]{faergeman2022non}). By abuse of notation, we also denote by $\on{ins}_S$ the similarly defined functor
\[
 \bigotimes_{x\in S}\check{\fn}^{-,\omega(S)}(K_x)\on{-mod}^{\check{N}_{(1)}^{-,\omega(S)}(O_x)}\to \check{\fn}^{-,\omega(S)}(K)\on{-mod}_{\on{Ran},S}^{\check{N}_{(1)}^{-,\omega(S)}(O)}.
\]

\subsubsection{} Let $\on{exp}$ be the exponential D-module on $\bA^1$ normalized to be concentrated in perverse degree $-1$. Recall the Whittaker functional $\psi^{-}$ (\ref{eq:whitfunction}).

Consider the functor
\[
D(\on{Bun}_{\check{N}^{-}}^{\omega_X(S)}(X,S))\to \on{Vect}, \:\: \cF\mapsto C_{\on{dR}}(\on{Bun}_{\check{N}^{-}}^{\omega_X(S)}(X,S),  \cF\overset{!}{\otimes} (\psi^{-})^!(\on{exp})). 
\]

\subsubsection{} The following proposition follows from \cite[Thm. 4.0.5 (4)]{chen2021extension} (see also \cite[§12.5]{Cstterm}):

\begin{prop}\label{p:semiinf}
The following diagram commutes:
\[\begin{tikzcd}
	{\bigotimes_{x\in S} \check{\mathfrak{n}}^{-,\omega(S)}(K_x)\mathrm{-mod}^{\check{N}_{(1)}^{-,\omega(S)}(O_x)}} &&& {D(\mathrm{Bun}_{\check{N}^{-}}^{\omega_X(S)}(X,S))} \\
	\\
	{\check{\mathfrak{n}}^{-,\omega(S)}(K)\mathrm{-mod}_{\mathrm{Ran},S}^{\check{N}_{(1)}^{-,\omega(S)}(O)}} &&& {\mathrm{Vect}.}
	\arrow["{\mathrm{Loc}_{\check{N}^{-},S}}", from=1-1, to=1-4]
	\arrow["{\mathrm{ins}_S}"', from=1-1, to=3-1]
	\arrow["{\Psi_{\mathrm{fact}}^{-}[\on{dim}\on{Bun}_{\check{N}^{-}}^{\omega_X(S)}(X,S)]}", from=3-1, to=3-4]
	\arrow["{C_{\mathrm{dR}}(\mathrm{Bun}_{\check{N}^{-}}^{\omega_X(S)}(X,S),-\overset{!}{\otimes}(\psi^{-})^!(\mathrm{exp}))}", from=1-4, to=3-4]
\end{tikzcd}\]

\end{prop}

\subsection{Compatibility between localization and Whittaker coefficients}

\subsubsection{}\label{s:whittakercoeff} Recall the group scheme $\check{G}^{\omega(S)}$ and the stack $\on{Bun}_{\check{G}^{\omega(S)},\check{N}^{\omega(S)}}(X,S)$ from §\ref{s:grpscheme}. We have a canonical map
\[
\mathrm{Bun}_{\check{N}^{-}}^{\omega_X(S)}(X,S)\to \on{Bun}_{\check{G}^{\omega(S)},\check{N}^{\omega(S)}}(X,S).
\]

\noindent We denote by $\fp$ the map
\[
\fp: \mathrm{Bun}_{\check{N}^{-}}^{\omega_X(S)}(X,S)\to \on{Bun}_{\check{G}^{\omega(S)},\check{N}^{\omega(S)}}(X,S)\simeq \on{Bun}_{\check{G},\check{N}}(X,S).
\]

\subsubsection{}\label{s:fouriercoeff} Let $\on{coeff}_{\chi_S}$ be the composition
\[
\on{coeff}_{\chi_S}: D(\BunN)^{\check{T}_S,\chi_S}\xrightarrow{\on{oblv}^{\chi_S}} D(\BunN)\xrightarrow{\fp^!} D(\on{Bun}_{\check{N}^{-}}^{\omega_X(S)}(X,S))\to 
\]
\[
\xrightarrow{C_{\on{dR}}(\on{Bun}_{\check{N}^{-}}^{\omega_X(S)}(X,S), -\overset{!}{\otimes} (\psi^{-})^!(\on{exp}))} \on{Vect}.
\]
\begin{rem}\label{r:glc}
The conjectural tamely ramified geometric Langlands equivalence 
\[
D(\BunN)^{\check{T}_S,\chi_S}\simeq \on{IndCoh}_{\on{Nilp}}(\on{LocSys}_G(U)\underset{\underset{x\in S}{\prod}\on{LocSys}_G(\overset{\circ}{D}_x)}{\times} \underset{x\in S}{\prod}\on{LocSys}_B(\overset{\circ}{D}_x)_{\chi_x})
\]

\noindent is supposed to exchange the functor of taking global sections on the right-hand side with $\on{coeff}_{\chi_S}$ on the left-hand side. We emphasize that the choice of $\on{coeff}_{\chi_S}$ is \emph{not} canonical. Instead, we could have chosen our Whittaker character to be defined on $N^{\omega}(K)$ and be of conductor 0. Then the stack $\mathrm{Bun}_{\check{N}^{-}}^{\omega_X(S)}(X,S))$ would be replaced with the stack $\on{Bun}_N^{\omega}(X)$ of $N^{\omega}$-bundles on $X$, which also appears in the unramified story. In this case, we would obtain a different geometric Langlands equivalence from the one above.

The multiple choices for Whittaker-normalization directly reflects the fact that there are multiple realizations of $\on{QCoh}(\fn/B)$ as Whitakker-equivariant sheaves on the (dual) affine flag variety.
\end{rem}

\subsubsection{} Note that we have a natural factorizable restriction functor
\[
\hat{\check{\mathfrak{g}}}_{\mathrm{crit},\mathrm{Ran}_{X,S}}\mathrm{-mod}^{({\check{I}}_S)_{\mathrm{Ran}_{X,S}},\chi_S}\xrightarrow{\on{Res}^{\check{G}}_{\check{N}^{-,\omega(S)}}}\check{\mathfrak{n}}^{-,\omega(S)}(K)\mathrm{-mod}_{\mathrm{Ran},S}^{\check{N}_{(1)}^{-,\omega(S)}(O)}.
\]

\noindent Combining Proposition \ref{p:semiinf} with \cite[Thm. 4.0.5(3)]{chen2021extension}, we obtain:
\begin{prop}\label{p:coeffwhit}
    The following diagram commutes: 
\[\begin{tikzcd}
	{\bigotimes_{x\in S} \hat{\check{\mathfrak{g}}}_{\mathrm{crit},x}\mathrm{-mod}^{\check{I}_x,\chi_x}} &&& {\hat{\check{\mathfrak{g}}}_{\mathrm{crit},\mathrm{Ran}_{X,S}}\mathrm{-mod}^{({\check{I}}_S)_{\mathrm{Ran}_{X,S}},\chi_S}} \\
	\\
	{D(\mathrm{Bun}_{\check{G},\check{N}}(X,S)^{\check{T}_S,\chi_S}} &&& {\check{\mathfrak{n}}^{-,\omega(S)}(K)\mathrm{-mod}_{\mathrm{Ran},S}^{\check{N}_{(1)}^{-,\omega(S)}(O)}} \\
	\\
	&&& {\mathrm{Vect}.}
	\arrow["{\mathrm{ins}_S}", from=1-1, to=1-4]
	\arrow["{\mathrm{Loc}_{\mathrm{crit},\check{G},S}}"', from=1-1, to=3-1]
	\arrow["{\mathrm{Res}^{\check{G}}_{\check{N}^{-,\omega(S)}}}", from=1-4, to=3-4]
	\arrow["{\Psi_{\mathrm{fact}}^{-}[\mathrm{dim}\;\mathrm{Bun}_{\check{N}^{-}}^{\omega_X(S)}(X,S)]}", from=3-4, to=5-4]
	\arrow["{\mathrm{coeff}_{\chi_S}}"', from=3-1, to=5-4]
\end{tikzcd}\]
\end{prop}

\noindent Henceforth, we also write $\Psi^{-}_{\on{fact}}$ for the composition $\Psi^{-}_{\on{fact}}\circ  \on{Res}^{\check{G}}_{\check{N}^{-,\omega(S)}}$.

\subsection{Accounting for singularities}\label{s:4.3}

\subsubsection{} Let
\[
\on{Op}_G(\overset{\circ}{D})^{\on{mon-free}}:=\on{Op}_G(\overset{\circ}{D})\underset{\on{LocSys}_G(\overset{\circ}{D})}{\times} \on{LocSys}_G(D)
\]

\noindent be the ind-scheme of monodromy-free opers on the punctured disk. By \cite{frenkel2007langlands}, we have an equivalence
\begin{equation}\label{eq:pointwisecritfle}
\on{QCoh}(\on{Op}_G(\overset{\circ}{D})^{\on{mon-free}})\simeq \hat{\check{\fg}}_{\on{crit}}\on{-mod}^{\check{G}(O)}=:\on{KL}_{\on{crit},\check{G}}.
\end{equation}

\noindent This functor is $\on{Rep}(G)$-linear by \cite[Thm. 1.10]{frenkel2005fusion} (or \cite[§6]{Cstterm}), where $\on{Rep}(G)$ acts on the left-hand side via pullback along the map $\on{Op}_{G}(\overset{\circ}{D})^{\on{mon-free}}\to\on{LocSys}_G(D)\simeq \bB G$ and on the right-hand side via geometric Satake.

\subsubsection{Global opers with singularities}\label{s:xminuss} Fix a finite set of points $T=\lbrace x_1,...,x_n\rbrace\subset X(k)$ that are disjoint from $S$. Let $U'=X\setminus (T\sqcup S)=U\setminus T$. Let $\on{Op}_{G}(U')$ be the ind-scheme of $G$-opers on $U'$, cf. \cite[§A]{bezrukavnikov2016quantization}.

\subsubsection{} For any $x\in T\sqcup S$, we have a restriction map
\[
\on{Op}_{G}(U')\to \on{Op}_{G}(\overset{\circ}{D}_x).
\]

\noindent For each $x\in S$, choose a coweight $\lambda_x\in \ft$.

\subsubsection{} Let
\[
\on{Op}_{G,\lambda_S}^{\on{RS}}(U')^{\on{mon-free}}:=\on{Op}_{G}(U')\underset{\underset{x'\in T}{\prod}\on{Op}_{G}(\overset{\circ}{D}_{x'})\times\underset{x\in S}{\prod}\on{Op}_{G}(\overset{\circ}{D}_x)}{\times} \underset{x'\in T}{\prod}\on{Op}_{G}(\overset{\circ}{D}_{x'})^{\on{mon-free}}\times\underset{x\in S}{\prod}\on{Op}_{G}(\overset{\circ}{D}_x)^{\on{RS},\varpi(-\lambda_x-\check{\rho})}.
\]

\noindent We have a forgetful map
\begin{equation}\label{eq:forgetful}
\on{Op}_{G,\lambda_S}^{\on{RS}}(U')^{\on{mon-free}}\to \on{LocSys}_G(U,\chi_S),
\end{equation}

\noindent where $\chi_x$ is the projection of $\lambda_x$ to $\ft/X_{\bullet}(T)$.

\subsubsection{}
We let $r$ denote the map
\[
r:\on{Op}_{G,\lambda_S}^{\on{RS}}(U')^{\on{mon-free}}\to \underset{x'\in T}{\prod}\on{Op}_{G}(\overset{\circ}{D}_{x'})^{\on{mon-free}}\times\underset{x\in S}{\prod}\on{Op}_{G}(\overset{\circ}{D}_x)^{\on{RS},\varpi(-\lambda_x-\check{\rho})}.
\]

\subsubsection{Spectral localization}\label{s:diagram} Consider the composition 
\[
\on{Loc}_{G,\on{spec}}:\bigotimes_{x'\in T}\on{QCoh}(\on{Op}_{G}(\overset{\circ}{D}_{x'})^{\on{mon-free}})\otimes \bigotimes_{x\in S} \on{QCoh}(\on{Op}_{G}(\overset{\circ}{D}_x)^{\on{RS},\varpi(-\lambda_x-\check{\rho})})\to
\]
\[
\to \bigotimes_{x'\in T} \on{KL}_{\on{crit},\check{G},x'}\otimes \bigotimes_{x\in S} \hat{\check{{\mathfrak{g}}}}_{\on{crit},x}\on{-mod}^{\check{I}_x,\chi_{_x}} \to D(\BunN)^{\check{T}_S,\chi_S},
\]

\noindent where the first functor is $(\ref{eq:pointwisecritfle})\otimes (\ref{eq:critfle})$, and the second functor is given by localization. We have:
\begin{prop}\label{p:factoropreg}
The above functor factors as
\[\begin{tikzcd}
	{\underset{x'\in T}{\otimes}\mathrm{QCoh}(\mathrm{Op}_{G}(\overset{\circ}{D}_{x'})^{\mathrm{mon-free}}) \otimes\underset{x\in S}{\otimes} \mathrm{QCoh}(\mathrm{Op}_{G}(\overset{\circ}{D}_x)^{\mathrm{RS},\varpi(-\lambda_x-\check{\rho})})} \\
	\\
	{\on{QCoh}(\on{Op}_{G,\lambda_S}^{\on{RS}}(U')^{\on{mon-free}})} && {D(\BunN)^{\check{T}_S,\chi_{S}}.}
	\arrow["{r^*}"', from=1-1, to=3-1]
	\arrow["{\mathrm{Loc}_{G,\mathrm{spec}}^{\mathrm{glob}}}"', dashed, from=3-1, to=3-3]
	\arrow["{\mathrm{Loc}_{G,\mathrm{spec}}}", shift right=1, from=1-1, to=3-3]
\end{tikzcd}\]

\noindent Moreover, the resulting functor $\on{Loc}_{G,\on{spec}}^{\on{glob}}$ is $\NR$-linear, where the latter acts on $\on{QCoh}(\on{Op}_{G,\lambda_S}^{\on{RS}}(U')^{\on{mon-free}})$ via its symmetric monoidal functor to $\on{QCoh}(\LSN)$ and the map (\ref{eq:forgetful}).

\end{prop}

\begin{proof}
This follows from Theorem \ref{t:factor} below.
\end{proof}

\subsubsection{} By construction of $\on{Loc}_{G,\on{spec}}^{\on{glob}}$, it follows from Lemma \ref{l:globsectDS} and Proposition \ref{p:coeffwhit} that we have:
\begin{prop}\label{p:globalsect=coeff}
The following diagram commutes:
\[\begin{tikzcd}
	{\mathrm{QCoh}(\on{Op}_{G,\lambda_S}^{\on{RS}}(U')^{\on{mon-free}})} && {D(\BunN)^{\check{T}_S,\chi_{S}}} \\
	\\
	&& {\mathrm{Vect}.}
	\arrow["{\mathrm{Loc}_{G,\mathrm{spec}}^{\mathrm{glob}}}", from=1-1, to=1-3]
	\arrow["{\mathrm{coeff}_{\chi_S}}", from=1-3, to=3-3]
	\arrow["{\Gamma(\on{Op}_{G,\lambda_S}^{\on{RS}}(U')^{\on{mon-free}},-)[\on{dim}\;\mathrm{Bun}_{\check{N}^{-}}^{\omega_X(S)}(X,S)]}"', from=1-1, to=3-3]
\end{tikzcd}\]
\end{prop}

\subsubsection{}\label{s:simplifynot} Recall that the Feigin-Frenkel isomorphism provides actions (see \cite[§4.8]{Cstterm}):
\[
\on{QCoh}(\on{Op}_G(\overset{\circ}{D}_x))\curvearrowright \hat{\check{{\mathfrak{g}}}}_{\on{crit},x}\on{-mod}^{\check{I}_x,\chi_x}, \on{KL}_{\on{crit},\check{G},x}.
\]

\noindent Moreover, the action
\[
\on{QCoh}(\on{Op}_G(\overset{\circ}{D}_x))\curvearrowright \on{KL}_{\on{crit},\check{G},x}
\]

\noindent factors as 
\[
\on{QCoh}(\on{Op}_G(\overset{\circ}{D}_x))\to \on{QCoh}(\on{Op}_G(\overset{\circ}{D}_x)^{\on{mon-free}}) \curvearrowright \on{KL}_{\on{crit},\check{G},x}.
\]

\subsubsection{}\label{s:gengoodsh} Let
\[
\hat{\check{{\mathfrak{g}}}}_{\on{crit},x}\on{-mod}^{\check{I}_x,\chi_x}_{\lambda_x}:=\on{Hom}_{\on{QCoh}(\on{Op}_G(\overset{\circ}{D}_x))\textbf{-mod}}\big(\on{QCoh}(\on{Op}_G(\overset{\circ}{D}_x)^{\on{RS},\varpi(-\lambda_x-\check{\rho})}),\hat{\check{{\mathfrak{g}}}}_{\on{crit},x}\on{-mod}^{\check{I}_x,\chi_x}\big).
\]

\noindent Pullback along $\iota_{\lambda}: \on{Op}_G(\overset{\circ}{D}_x)^{\on{RS},\varpi(-\lambda_x-\check{\rho})}\to \on{Op}_G(\overset{\circ}{D}_x)$ induces a functor
\begin{equation}\label{eq:pullbackoprsop}
\hat{\check{{\mathfrak{g}}}}_{\on{crit},x}\on{-mod}^{\check{I}_x,\chi_x}_{\lambda_x}\to \hat{\check{{\mathfrak{g}}}}_{\on{crit},x}\on{-mod}^{\check{I}_x,\chi_x}.
\end{equation}

\noindent We claim that the functors (\ref{eq:pullbackoprsop}) generate the target under colimits, where $\lambda_x$ ranges over all coweights lying over $\chi_x\in\ft/X_{\bullet}(T)$. Indeed, it follows by \cite[Cor. 13.3.2]{frenkel2006local} that the Verma module $\bM_{\lambda_x}$ is in the image of (\ref{eq:pullbackoprsop})

\subsubsection{} To simplify notation, we write
\[
\on{Op}_{G,\lambda_S,T}^{\on{RS}}(\overset{\circ}{D})^{\on{mon-free}}:=\underset{x'\in T}{\prod}\mathrm{Op}_{G}(\overset{\circ}{D}_{x'})^{\mathrm{mon-free}}\times\underset{x\in S}{\prod}\mathrm{Op}_{G}(\overset{\circ}{D}_x)^{\mathrm{RS},\varpi(-\lambda_x-\check{\rho})}.
\]

\noindent Thus, we have an action
\[
\on{QCoh}(\on{Op}_{G,\lambda_S,T}^{\on{RS}}(\overset{\circ}{D})^{\on{mon-free}})\curvearrowright \underset{x'\in T}{\bigotimes} \on{KL}_{\on{crit},\check{G},x'}\otimes \underset{x\in S}{\bigotimes} \hat{\check{{\mathfrak{g}}}}_{\on{crit},x}\on{-mod}^{\check{I}_x,\chi_x}_{\lambda_x}.
\]

\subsubsection{} Denote by $\on{Loc}_{\check{G}}$ the localization map
\[
\on{Loc}_{\check{G}}:\bigotimes_{x'\in T} \on{KL}_{\on{crit},\check{G},x'}\otimes \bigotimes_{x\in S} \hat{\check{{\mathfrak{g}}}}_{\on{crit},x}\on{-mod}^{\check{I}_x,\chi_{x}}\to D(\BunN)^{\check{T}_S,\chi_S}.
\]

\noindent We have the following theorem, which is a particular case of Theorem \ref{t:factormoving} below.

\begin{thm}\label{t:factor}
The localization functor
\[
\underset{x'\in T}{\bigotimes} \on{KL}_{\on{crit},\check{G},x'}\otimes \underset{x\in S}{\bigotimes} \hat{\check{{\mathfrak{g}}}}_{\on{crit},x}\on{-mod}^{\check{I}_x,\chi_x}_{\lambda_x}\to \underset{x'\in T}{\bigotimes} \on{KL}_{\on{crit},\check{G},x'}\otimes \underset{x\in S}{\bigotimes} \hat{\check{{\mathfrak{g}}}}_{\on{crit},x}\on{-mod}^{\check{I}_x,\chi_x}\to
\]
\[
\to D(\BunN)^{\check{T}_S,\chi_S}
\]

\noindent factors through a functor
\[
\underset{x'\in T}{\bigotimes} \on{KL}_{\on{crit},\check{G},x'}\otimes \underset{x\in S}{\bigotimes} \hat{\check{{\mathfrak{g}}}}_{\on{crit},x}\on{-mod}^{\check{I}_x,\chi_x}_{\lambda_x}\to 
\]
\[
\to \on{QCoh}(\mathrm{Op}_{G,\lambda_S}^{\mathrm{RS}}(U')^{\mathrm{mon-free}})\underset{\on{QCoh}(\on{Op}_{G,\lambda_S,T}^{\on{RS}}(\overset{\circ}{D})^{\on{mon-free}})}{\otimes}\underset{x'\in T}{\bigotimes} \on{KL}_{\on{crit},\check{G},x'}\otimes \underset{x\in S}{\bigotimes} \hat{\check{{\mathfrak{g}}}}_{\on{crit},x}\on{-mod}^{\check{I}_x,\chi_x}_{\lambda_x}
\]
\[
\to D(\BunN)^{\check{T}_S,\chi_S}.
\]

\noindent Moreover, the resulting functor
\[
\on{QCoh}(\mathrm{Op}_{G,\lambda_S}^{\mathrm{RS}}(U')^{\mathrm{mon-free}})\underset{\on{QCoh}(\on{Op}_{G,\lambda_S,T}^{\on{RS}}(\overset{\circ}{D})^{\on{mon-free}})}{\otimes}\underset{x'\in T}{\bigotimes} \on{KL}_{\on{crit},\check{G},x'}\otimes \underset{x\in S}{\bigotimes} \hat{\check{{\mathfrak{g}}}}_{\on{crit},x}\on{-mod}^{\check{I}_x,\chi_x}_{\lambda_x}
\]
\[
\to D(\BunN)^{\check{T}_S,\chi_S}
\]

\noindent is $\NR$-linear, where the latter acts on the left-hand side by acting on $\on{QCoh}(\mathrm{Op}_{G,\lambda_S}^{\mathrm{RS}}(U')^{\mathrm{mon-free}})$.

\end{thm}

\subsubsection{} We now formulate a moving-point version of the above theorem.

Consider the $(X_S^I)_{\on{dR}}$-scheme $\on{Op}_{G,\lambda_S}^{\on{RS}}(\overset{\circ}{D})_I^{\on{mon-free}}$ parametrizing $(\underline{x},\sigma)$, where:
\begin{itemize}
    \item $\underline{x}\in (X_S^I)_{\on{dR}}$.

    \item $\sigma$ is a $G$-oper on $\overset{\circ}{D}_{\underline{x}}$ that lies in $\mathrm{Op}_{G}(\overset{\circ}{D}_{x'})^{\on{mon-free}}$ when restricted to $\overset{\circ}{D}_{x'}$ for each $x'\in \underline{x}-S$, and that lies in $\mathrm{Op}_{G}(\overset{\circ}{D}_{x})^{\on{RS},\varpi(-\lambda_x-\check{\rho})}$ when restricted to $\overset{\circ}{D}_x$ for each $x\in S$.
\end{itemize}

\subsubsection{} We let $\on{Op}_{G,\lambda_S}^{\on{RS}}(X-\underline{x})_I^{\on{mon-free}}$ be the corresponding global space. Namely, it parametrizes pairs $(\underline{x},\sigma)$, where now $\sigma$ is an oper on $X-\underline{x}$ that satisfies the same requirements as above when restricted to each punctured disk around $\underline{x}$.

\subsubsection{} By construction, we have a natural map
\[
\on{Op}_{G,\lambda_S}^{\on{RS}}(X-\underline{x})_I^{\on{mon-free}}\to \on{Op}_{G,\lambda_S}^{\on{RS}}(\overset{\circ}{D})_I^{\on{mon-free}}
\]

\noindent over $(X_S^I)_{\on{dR}}$ by restricting the oper along $\overset{\circ}{D}_{\underline{x}}\to X-\underline{x}$. Moreover, from the map
\[
\on{Op}_{G,\lambda_S}^{\on{RS}}(X-\underline{x})_I^{\on{mon-free}}\to \on{LocSys}_G(U,\chi_S),
\]

\noindent we obtain an action
\begin{equation}\label{eq:actopI}
\NR\to \on{QCoh}(\LSN)\curvearrowright \on{QCoh}(\on{Op}_{G,\lambda_S}^{\on{RS}}(X-\underline{x})_I^{\on{mon-free}}).
\end{equation}

\subsubsection{} Denote by $\on{Op}_G(\overset{\circ}{D})_I$ the relative ind-affine ind-scheme over $(X_S^I)_{\on{dR}}$ parametrizing a point $\underline{x}\in (X_S^I)_{\on{dR}}$ and an oper on the punctured disk around $\underline{x}$.

\subsubsection{} To reduce notation, we write
\[
\hat{\check{\fg}}_{\on{crit},I}\on{-mod}^{\check{I}_S, \chi_S}:=D(X_S^I)\underset{D(\on{Ran}_{X,S})}{\otimes} \hat{\check{\fg}}_{\on{crit},\on{Ran}_{X,S}}\on{-mod}^{(\check{I}_S)_{\on{Ran}_{X,S}}, \chi_S}.
\]

\noindent Similar to above, \cite[§4]{Cstterm} provides an action
\begin{equation}\label{eq:actopoverI}
\on{QCoh}(\on{Op}_G(\overset{\circ}{D})_I)\curvearrowright \hat{\check{\fg}}_{\on{crit},I}\on{-mod}^{\check{I}_S, \chi_S}.
\end{equation}

\noindent We let
\[
\hat{\check{\fg}}_{\on{crit},I,\lambda_S}\on{-mod}^{\check{I}_S, \chi_S}:=\on{Hom}_{\on{QCoh}(\on{Op}_G(\overset{\circ}{D})_I)\textbf{-mod}}\big(\on{QCoh}(\on{Op}_{G,\lambda_S}^{\on{RS}}(\overset{\circ}{D})_I^{\on{mon-free}}),\hat{\check{\fg}}_{\on{crit},I}\on{-mod}^{\check{I}_S, \chi_S}\big).
\]

\subsubsection{} Pullback along
\[
\on{Op}_{G,\lambda_S}^{\on{RS}}(\overset{\circ}{D})_I^{\on{mon-free}}\to \on{Op}_G(\overset{\circ}{D})_I
\]

\noindent induces a forgetful functor
\[
\hat{\check{\fg}}_{\on{crit},I,\lambda_S}\on{-mod}^{\check{I}_S, \chi_S}\to \hat{\check{\fg}}_{\on{crit},I}\on{-mod}^{\check{I}_S, \chi_S}.
\]

\subsubsection{} The following is a moving-point version of Theorem \ref{t:factor}. It is a tamely ramified analogue of \cite[Thm. 16.1.7+Cor. 16.5.9]{Cstterm}. Given Conjecture \ref{c:linearity}, the same proof goes through mutatis mutandis, as we explain in §\ref{s:someremarks} below.
\begin{thm}\label{t:factormoving}
The localization functor
\[
\hat{\check{\fg}}_{\on{crit},I,\lambda_S}\on{-mod}^{\check{I}_S, \chi_S}\to \hat{\check{\fg}}_{\on{crit},I}\on{-mod}^{\check{I}_S, \chi_S}\to D(\BunN)^{\check{T}_S,\chi_S}
\]

\noindent factors through a functor
\[
\hat{\check{\fg}}_{\on{crit},I,\lambda_S}\on{-mod}^{\check{I}_S, \chi_S}\to \on{QCoh}(\on{Op}_{G,\lambda_S}^{\on{RS}}(X-\underline{x})_I^{\on{mon-free}})\underset{\on{QCoh}(\on{Op}_{G,\lambda_S}^{\on{RS}}(\overset{\circ}{D})_I^{\on{mon-free}})}{\otimes} \hat{\check{\fg}}_{\on{crit},I,\lambda_S}\on{-mod}^{\check{I}_S, \chi_S}\to 
\]
\[
\to D(\BunN)^{\check{T}_S,\chi_S}.
\]

\noindent Moreover, the resulting functor
\[
\on{QCoh}(\on{Op}_{G,\lambda_S}^{\on{RS}}(X-\underline{x})_I^{\on{mon-free}})\underset{\on{QCoh}(\on{Op}_{G,\lambda_S}^{\on{RS}}(\overset{\circ}{D})_I^{\on{mon-free}})}{\otimes} \hat{\check{\fg}}_{\on{crit},I,\lambda_S}\on{-mod}^{\check{I}_S, \chi_S}\to D(\BunN)^{\check{T}_S,\chi_S}
\]

\noindent is $\NR$-linear, where the latter acts on the left-hand side by acting on $\on{QCoh}(\on{Op}_{G,\lambda_S}^{\on{RS}}(X-\underline{x})_I^{\on{mon-free}})$ via (\ref{eq:actopI}).
\end{thm}

\subsubsection{}\label{s:someremarks}

Let us make a few remarks about the construction in \cite{Cstterm} and how it adapts to Theorem \ref{t:factormoving}.

\begin{itemize}
    \item The first part of Theorem \ref{t:factormoving}, the fact that the localization functor factors as stated, is a very general consequence of \cite{beilinson2004chiral}. This is evident from the arguments in \cite[§16.6-16.7]{Cstterm}. In fact, the analogous statement holds true in the presence of wild ramification with the same proof.
\end{itemize}

\begin{itemize}
    \item The result \cite[Cor. 16.5.9]{Cstterm} gives $\NR$-linearity away from the ramified points $S$. Conjecture \ref{c:linearity} gives linearity at the ramified points. It then follows from the gluing construction of Theorem \ref{p:upgrade} and the action (\ref{eq:actopoverI}) that the functor in Theorem \ref{t:factormoving} is $\NR$-linear.
\end{itemize}

\vspace{2mm}

\noindent \textbf{Proof of Proposition \ref{p:km}}

\noindent 

We need to show that the localization functor
\[
\on{Loc}_{\on{crit},\check{G},\chi_S}: \hat{\check{\fg}}_{\on{crit},\on{Ran}_{X,S}}\on{-mod}^{(\check{I}_S)_{\on{Ran}_{X,S}}, \chi_S}\to D(\BunN)^{\check{T}_S,\chi_S}
\]

\noindent lands in the full subcategory
\[
\on{QCoh}(\LSN)\underset{\NR}{\otimes} D(\BunN)^{\check{T}_S,\chi_S}\into D(\BunN)^{\check{T}_S,\chi_S}.
\]

\noindent It suffices to show this for the restriction of the localization functor to the categories
\[
\hat{\check{\fg}}_{\on{crit},I}\on{-mod}^{\check{I}_S, \chi_S},
\]

\noindent for each $I\in\on{fSet}_S$. By §\ref{s:gengoodsh}, the functors 
\[
\hat{\check{\fg}}_{\on{crit},I,\lambda_S}\on{-mod}^{\check{I}_S, \chi_S}\to \hat{\check{\fg}}_{\on{crit},I}\on{-mod}^{\check{I}_S,\chi_S}
\]

\noindent generate the target under colimits. Thus, we may further restrict our attention to the category $\hat{\check{\fg}}_{\on{crit},I,\lambda_S}\on{-mod}^{\check{I}_S, \chi_S}$.

The assertion that the localization functor
\[
\hat{\check{\fg}}_{\on{crit},I,\lambda_S}\on{-mod}^{\check{I}_S, \chi_S}\to D(\BunN)^{\check{T}_S,\chi_S}
\]

\noindent lands in the full subcategory
\[
\on{QCoh}(\LSN)\underset{\NR}{\otimes} D(\BunN)^{\check{T}_S,\chi_S}
\]

\noindent is now an immediate consequence of Theorem \ref{t:factormoving}.

\vspace{2mm}

\qed

\subsection{Construction of Hecke eigensheaves}\label{s:Heckeeigensheaf}
We now prove the existence of a Hecke eigensheaf associated to any irreducible local system on an open subvariety $U$ of $X$ with regular singularities. Assume first that our group $G$ is adjoint. The case of an arbitrary connected reductive group reduces to the adjoint case as argued in §\ref{s:adjoint}.

\subsubsection{} Fix $x_1,...,x_n\in X(k)$ and eigenvalues $\chi_{x_1},...,\chi_{x_n}\in\ft//\widetilde{W}^{\on{aff}}$. Let $\sigma\in\on{LocSys}_G(U,\chi_S)^{\on{irr}}(k)$ be an irreducible local system. By \cite[Cor. 1.1]{arinkin2016irreducible}, there exists $T=\lbrace x_1',...,x_m'\rbrace\subset U(k)$ such that $\sigma$ admits a lift to $\on{Op}_G(U')$ with $U'=U-T$.\footnote{The assumption that $G$ be adjoint is crucial here.} Since $\sigma$ has regular singularities, \cite[Prop. 2.4.1]{frenkel2006local} says that the oper necessarily defines an element of $\on{Op}_G(\overset{\circ}{D}_x)^{\on{RS}}$ when restricted to each $x\in S$.

Thus, we may find coweights $\lambda_x\in \ft$ lying over $\chi_x$ such that $\sigma$ admits a lift to $\on{Op}_{G,\lambda_S}^{\on{RS}}(U')^{\on{mon-free}}$.

\subsubsection{} Denote by $\delta_{\sigma}\in \on{QCoh}(\on{Op}_{G,\lambda_S}^{\on{RS}}(U')^{\on{mon-free}})$ the delta sheaf corresponding to some fixed lift of $\sigma$. Let
\[
\cF_{\sigma}:=\on{Loc}_{G,\on{spec}}^{\on{glob}}(\delta_{\sigma})\in D(\BunN)^{\check{T}_S,\chi_S},
\]

\noindent where $\on{Loc}_{G,\on{spec}}^{\on{glob}}$ is the functor given in §\ref{s:diagram}. Recall that $\on{Loc}_{G,\on{spec}}^{\on{glob}}$ is $\on{QCoh}(\LSN)$-linear. In particular, $\cF_{\sigma}$ is a Hecke eigensheaf with eigenvalue $\sigma$. Moreover, it follows from Proposition \ref{p:globalsect=coeff} that $\cF_{\sigma}$ is Whittaker-normalized in the sense that
\[
\on{coeff}_{\chi_S}(\cF_{\sigma})[-\on{dim}\;\mathrm{Bun}_{\check{N}^{-}}^{\omega_X(S)}(X,S)]=k\in \on{Vect}^{\heartsuit}.
\]
\begin{rem}
Suppose $\chi_S=0$, and let $\on{Bun}_{\check{G},\check{B}}(X,S)$ be the moduli of $\check{G}$-bundles with a reduction to $\check{B}$ at $S$. It is not difficult to show that $\on{coeff}_{\chi_S}$ kills anything pulled back via the forgetful map
\[
\on{Bun}_{\check{G},\check{B}}(X,S)\to \on{Bun}_{\check{G}}(X).
\]

\noindent Suppose our local system $\sigma$ above is unramified; i.e., monodromy-free. Then we have just constructed a Whittaker-normalized Hecke eigensheaf with $\sigma$ as eigenvalue. However, we also know from the unramified story that there is a Hecke eigensheaf corresponding to $\sigma$ on $\on{Bun}_{\check{G}}(X)$. Since the latter is killed by $\on{coeff}_{\chi_S}$, these two Hecke eigensheaves have to be different (note, however, that we could have chosen a Whitaker functional that does not kill unramified Hecke eigensheaves, see Remark \ref{r:glc}). As such, in the tamely ramimified setting, one can have different Hecke eigensheaves associated to the same (irreducible) local system.\footnote{Another way to see this is the following. For convenience, suppose $S=\lbrace x\rbrace$ is a single point. Ignoring temperedness issues, one is supposed to have an equivalence of categories $D(\on{Bun}_{\check{G},\check{B}}(X,S))\simeq \on{QCoh}(\on{LocSys}_G(U)\underset{\on{LocSys}_G(\overset{\circ}{D}_x)}{\times} \fn/B)$. Thus, the existence of multiple Hecke eigensheaves simply reflects the fact that the map $\fn/B\to \cN/G$ is not an embedding.}
\end{rem}

\noindent This concludes the proof of Theorem \ref{t:eigen} in the adjoint case.

\subsubsection{Reduction to the adjoint case.}\label{s:adjoint} We will now deduce the existence of Hecke eigensheaves associated to irreducible regular singular local systems by reduction to the adjoint case above. 

Let $G$ be an arbitrary reductive algebraic group, and let $H\to G$ be an isogeny. This induces an isogeny of dual groups $\check{G}\to \check{H}$. To distinguish, we write 
\begin{equation}\label{eq:charsheaves}
(\check{I}_{G,S})_{\on{Ran}_{X,S}}, (\check{I}_{H,S})_{\on{Ran}_{X,S}}
\end{equation}

\noindent for the subgroups of $\check{G}(O)_{\on{Ran}_{X,S}}, \check{H}(O)_{\on{Ran}_{X,S}}$, respectively, defined in §\ref{s:group}. Let $\chi_x$ be a character sheaf on the torus of $\check{H}$ for each $x\in S$. These induce character sheaves on the groups (\ref{eq:charsheaves}).

\subsubsection{} Let $\sC$ be a unital factorization category over $\on{Ran}_{X,S}$ and suppose moreover $\sC$ is equipped with a strong $\check{G}(K)_{\on{Ran}_{X,S}}$-action over $\on{Ran}_{X,S}$. Then the category
\[
\sC^{(\check{I}_{G,S})_{\on{Ran}_{X,S}},\chi_S}
\]

\noindent is also unital, and the independent category
\[
(\sC^{(\check{I}_{G,S})_{\on{Ran}_{X,S}},\chi_S})_{\on{indep}}=(\sC^{(\check{I}_{G,S})_{\on{Ran}_{X,S}},\chi_S})\underset{D(\on{Ran}_{X,S})}{\otimes} \on{Vect}
\]

\noindent is equipped with a $\cH^{\on{geom}}_{\check{G},\chi_S,\on{Ran},S}$-action with its external convolution structure, cf. §\ref{s:wildram}. By Theorem \ref{p:upgrade}, we also obtain a $\NR$-action on $(\sC^{(\check{I}_{G,S})_{\on{Ran}_{X,S}},\chi_S})_{\on{indep}}$.

\subsubsection{} Consider the category 
\[
\widetilde{\sC}:=D(\check{H}(K)_{\on{Ran}_{X,S}})\underset{D(\check{G}(K)_{\on{Ran}_{X,S}})}{\otimes} \sC
\]

\noindent with its natural $D(\check{H}(K)_{\on{Ran}_{X,S}})$-module structure. We similarly obtain an action 
\[
\on{QCoh}((\on{LocSys}_H(\overset{\circ}{D})^{\on{RS}}_{\chi_S})_{\on{Ran},S})\curvearrowright (\widetilde{\sC}^{(\check{I}_{H,S})_{\on{Ran}_{X,S}},\chi_S})_{\on{indep}}.
\]

\noindent We have a natural map
\begin{equation}\label{eq:isog}
(\on{LocSys}_H(\overset{\circ}{D})^{\on{RS}}_{\chi_S})_{\on{Ran},S}\to (\on{LocSys}_G(\overset{\circ}{D})^{\on{RS}}_{\chi_S})_{\on{Ran},S}.
\end{equation}

\noindent Pulling back, we obtain an action
\[
\NR\curvearrowright (\widetilde{\sC}^{(\check{I}_{H,S})_{\on{Ran}_{X,S}},\chi_S})_{\on{indep}}.
\]

\noindent Note that we have a natural $D(\on{Ran}_{X,S})$-linear map
\[
\sC\to \widetilde{\sC}
\]

\noindent induced by pushing forward along $\check{G}(K)_{\on{Ran}_{X,S}}\to \check{H}(K)_{\on{Ran}_{X,S}}$. In particular, we obtain a map
\[
(\sC^{(\check{I}_{G,S})_{\on{Ran}_{X,S}},\chi_S})_{\on{indep}}\to (\widetilde{\sC}^{(\check{I}_{H,S})_{\on{Ran}_{X,S}},\chi_S})_{\on{indep}}.
\]

\begin{prop}\label{p:isogeny}
The functor
\[
(\sC^{(\check{I}_{G,S})_{\on{Ran}_{X,S}},\chi_S})_{\on{indep}}\to (\widetilde{\sC}^{(\check{I}_{H,S})_{\on{Ran}_{X,S}},\chi_S})_{\on{indep}}
\]

\noindent is $\NR$-linear.

\end{prop}

\noindent Before the proof, we need a preliminary lemma:
\begin{lem}\label{l:adj}
The maps
\[
D(\check{G}(K)/\check{G}(O))\to D(\check{H}(K)/\check{H}(O)),
\]
\[
D(\check{G}(K)/\check{I}_{G},\chi)\to D(\check{H}(K)/\check{I}_H,\chi)
\]

\noindent are linear for the natural actions of $\on{Rep}(G)$ and $\on{QCoh}(\on{LocSys}_G(\overset{\circ}{D}_x)^{\on{RS}}_{\chi})$, respectively.
\end{lem}

\begin{proof}
\emph{Step 1.}
Let us prove the first assertion. Factor the functor in question as
\[
D(\check{G}(K)/\check{G}(O))\to D(\check{H}(K)/\check{G}(O))\to D(\check{H}(K)/\check{H}(O)),
\]

\noindent where the functors are given by pushforward along the evident maps. It suffices to show that the second functor is $\on{Rep}(G)$-linear. Let $\check{Z}$ be the kernel of $\check{G}\to \check{H}$. Then we have an action
\[
D(\check{H}(K)/\check{G}(O))\curvearrowleft D(\bB \check{Z})
\]

\noindent that commutes with the action of $\on{Rep}(G)$. In particular, we get an action of $\on{Rep}(G)$ on 
\begin{equation}\label{eq:equivadj12}
D(\check{H}(K)/\check{G}(O))\underset{D(\bB \check{Z})}{\otimes} \on{Vect}\simeq D(\check{H}(K)/\check{H}(O)).
\end{equation}

\noindent Here, $D(\bB \check{Z})$ acts on $\on{Vect}$ by pushforward along the map $\bB \check{Z}\to \on{pt}$. Now observe that the action of $\on{Rep}(G)$ on $D(\check{H}(K)/\check{H}(O))$ via the map $\on{Rep}(G)\to \on{Rep}(H)$ coincides with the action induced by the equivalence (\ref{eq:equivadj12}) and the action
\[
D(\check{H}(K)/\check{G}(O))\curvearrowleft \on{Rep}(G).
\]

\emph{Step 2.} Next we prove the second assertion of the lemma. Suppose first that $\chi=0$. In exactly the same way as above, we see that the map
\begin{equation}\label{eq:tameisogeny}
D(\check{G}(K)/\check{I}_{G},\chi)\to D(\check{H}(K)/\check{I}_H,\chi)
\end{equation}

\noindent is $\on{Rep}(G)$-linear, where the latter acts on the left-hand side via the symmetric monoidal functor
\begin{equation}\label{eq:nilptendo}
\on{Rep}(G)\to \on{QCoh}(\cN/G),
\end{equation}

\noindent and on the right-hand side via
\[
\on{Rep}(G)\to \on{Rep}(H)\to \on{QCoh}(\cN/H).
\]

\noindent As such, for each $V\in\on{Rep}(G)$, we obtain two nilpotent endomorphisms of the functor 
\[
D(\check{G}(K)/\check{I}_{G},\chi)\to D(\check{H}(K)/\check{I}_H,\chi)\xrightarrow{-\star V}D(\check{H}(K)/\check{I}_H,\chi).
\]

\noindent We need to show that these two nilpotent endomorphisms coincide (naturally in $V$).

By construction of the functor \cite{arkhipov2009perverse}, this follows from the fact that pushforward along the map $\check{G}(K)/\check{G}(O)\to \check{H}(K)/\check{H}(O)$, being a closed embedding, commutes with the nearby cycles construction of Gaitsgory \cite{gaitsgory1999construction} for $\check{G}$ and $\check{H}$, respectively, compatibly with restriction of representations along $H\to G$.

For arbitrary $\chi$, we are reduced to the above case by construction of the functor in Theorem \ref{t:dyz}. Alternatively, we may argue as in the case $\chi=0$ using the generalization of Gaitsgory's central sheaves construction to arbitrary eigenvalues appearing in \cite[Thm. 1]{bezrukavnikov2009tensor} (see also \cite{salmon2023restricted}, especially Proposition 5.28).

\end{proof}

\textbf{Proof of Proposition \ref{p:isogeny}}

Note that $\sC^{(\check{I}_{G,S})_{\on{Ran}_{X,S}},\chi_S}$ and $\widetilde{\sC}^{(\check{I}_{H,S})_{\on{Ran}_{X,S}},\chi_S}$ carry actions of $\NR$ over $\on{Ran}_{X_{\on{dR}},S}$ equipped with its \emph{pointwise} monoidal structure. By definition, it suffices to show that the map
\[
\sC^{(\check{I}_{G,S})_{\on{Ran}_{X,S}},\chi_S}\to \widetilde{\sC}^{(\check{I}_{H,S})_{\on{Ran}_{X,S}},\chi_S}
\]

\noindent is $\NR$-linear with this monoidal structure. Moreover, it suffices to consider the universal case where $\sC=D(\check{G}(K)_{\on{Ran}_{X,S}})$. Note that in this case we have $\widetilde{\sC}=D(\check{H}(K)_{\on{Ran}_{X,S}})$.

\emph{Step 1.} For $I\in\on{fSet}_S$, let
\[
X^I_{S,\on{disj}}=X_S^I\underset{X^{I-S}}{\times} (X-S)^{I-S},\;\;\;\;\; X^I_{S,\on{intersect}}=X^I_S-X^I_{S,\on{disj}}.
\]

\noindent We start by showing that the map
\[
D(X^I_{S,\on{disj}})\underset{D(X^I_S)}{\otimes} D(\check{G}(K)_{X_S^I}/(\check{I}_{G,S})_{X_S^I},\chi_S)\to
D(X^I_{S,\on{disj}})\underset{D(X^I_S)}{\otimes} D(\check{H}(K)_{X_S^I}/(\check{I}_{H,S})_{X_S^I},\chi_S)
\]

\noindent is $D(X^I_{S,\on{disj}})\underset{D(\on{Ran}_{X,S})}{\otimes}\NR$-linear.

Note that over a point $\underline{x}\in X^I_{S,\on{disj}}$, the assertion is equivalent to the statement that the maps
\[
D(\check{G}(K)/\check{G}(O))\to D(\check{H}(K)/\check{H}(O)),
\]
\[
D(\check{G}(K)/\check{I}_{G},\chi)\to D(\check{H}(K)/\check{I}_H,\chi)
\]

\noindent are linear for the natural actions of $\on{Rep}(G)$ and $\on{QCoh}(\on{LocSys}_G(\overset{\circ}{D}_x)^{\on{RS}}_{\chi})$, respectively. This is the content of Lemma \ref{l:adj}.

Over the entire space $X^I_{S,\on{disj}}$, the assertion follows from the combinatorical description:
\[
D(X^I_{S,\on{disj}})\underset{D(\on{Ran}_{X,S})}{\otimes}\NR
\]
\[
\simeq \bigg(D(X^I_{S,\on{disj}})\underset{D(X^S_I)}{\otimes}\underset{(I\onto J\onto K)\in\on{TwArr}_{I/}}{\on{colim}} D(X^K)\otimes \on{Rep}(G)^{\otimes J}\bigg) \otimes\underset{x\in S}{\bigotimes} \on{QCoh}(\on{LocSys}_G(\overset{\circ}{D}_x)^{\on{RS}}_{\chi_x}),
\]

\noindent where $\on{TwArr}_{I/}$ is the twisted arrows category (relative to $I$), see \cite[§2.5.8]{gaitsgory2021semi}.

\vspace{2mm}

\emph{Step 2.} 
Next, let us show that the map
\[
D(X^I_{S,\on{intersect}})\underset{D(X^I_S)}{\otimes} D(\check{G}(K)_{X_S^I}/(\check{I}_{G,S})_{X_S^I},\chi_S)\to D(X^I_{S,\on{intersect}})\underset{D(X^I_S)}{\otimes} D(\check{H}(K)_{X_S^I}/(\check{I}_{H,S})_{X_S^I},\chi_S)
\]

\noindent is $D(X^I_{S,\on{intersect}})\underset{D(\on{Ran}_{X,S})}{\otimes}\NR$-linear. It suffices to consider the situation when $I=S$. In this case, the assertion becomes that the map
\[
D(\check{G}(K)/\check{I}_{G},\chi)\to D(\check{H}(K)/\check{I}_H,\chi)
\]

\noindent is $\on{QCoh}(\on{LocSys}_G(\overset{\circ}{D}_x)^{\on{RS}}_{\chi})$-linear for any character sheaf $\chi$ on $\check{I}_H$ and hence follows from Lemma \ref{l:adj}. 
\vspace{2mm}

\emph{Step 3.} Finally, we need to show that the map
\[
D(\check{G}(K)_{X_S^I}/(\check{I}_{G,S})_{X_S^I},\chi_S)\to D(\check{H}(K)_{X_S^I}/(\check{I}_{H,S})_{X_S^I},\chi_S)
\]

\noindent is $D(X^I_S)\underset{D(\on{Ran}_{X,S})}{\otimes}\NR$-linear. This follows from Step 1 and 2 above using the combinatorical description of $D(X^I_S)\underset{D(\on{Ran}_{X,S})}{\otimes}\NR$ appearing in \cite{Satakefunctor}, which glues the category from its restriction to $X^I_{S,\on{disj}}$ and $X^I_{S,\on{intersect}}$.

\qed

\subsubsection{} Denote by $\check{T}_G, \check{T}_H$ the torus of $\check{G},\check{H}$, respectively. Note that the unipotent radicals of the Borels of $\check{G}$ and $\check{H}$ coincide.

Consider the natural map
\[
f: \on{Bun}_{\check{G},\check{N}}(X,S)\to \on{Bun}_{\check{H},\check{N}}(X,S).
\]

\noindent We get an induced functor
\[
f_{*,\on{dR}}^{\chi_S}: D(\on{Bun}_{\check{G},\check{N}}(X,S))^{\check{T}_{G,S},\chi_S}\to D(\on{Bun}_{\check{H},\check{N}}(X,S))^{\check{T}_{H,S},\chi_S}.
\]

\noindent Note that $\on{QCoh}(\on{LocSys}_G(U,\chi_S))$ naturally acts on the right-hand side via the symmetric monoidal functor
\[
\on{QCoh}(\on{LocSys}_G(U,\chi_S))\to \on{QCoh}(\on{LocSys}_H(U,\chi_S)).
\]

\begin{cor}\label{c:lin}
The functor $f_{*,\on{dR}}^{\chi_S}$ is $\on{QCoh}(\on{LocSys}_G(U,\chi_S))$-linear.
\end{cor}

\begin{proof}
It suffices to show that $f_{*,\on{dR}}^{\chi_S}$ is $\NR$-linear, where $\NR$ acts on $D(\on{Bun}_{\check{H},\check{N}}(X,S))^{\check{T}_{H,S},\chi_S}$ via the map (\ref{eq:isog}).

Let $\sC=D(\on{Bun}_{\check{G},\on{Ran}_{X,S}}^{\infty})$, where $\on{Bun}_{\check{G},\on{Ran}_{X,S}}^{\infty}$ is the prestack appearing in §\ref{s:wildram}. Then 
\[
(\sC^{(\check{I}_{G,S})_{\on{Ran}_{X,S}},\chi_S})_{\on{indep}}=D(\on{Bun}_{\check{G},\check{N}}(X,S))^{\check{T}_{G,S},\chi_S}.
\]

\noindent Note that $f_{*,\on{dR}}^{\chi_S}$ naturally factors as
\[
(\sC^{(\check{I}_{G,S})_{\on{Ran}_{X,S}},\chi_S})_{\on{indep}}\to (\widetilde{\sC}^{(\check{I}_{H,S})_{\on{Ran}_{X,S}},\chi_S})_{\on{indep}}\to D(\on{Bun}_{\check{H},\check{N}}(X,S))^{\check{T}_{H,S},\chi_S}.
\]

\noindent By Proposition \ref{p:isogeny}, the first map is $\NR$-linear. The second map is clearly $\cH_{\check{H},\chi_S,\on{Ran},S}^{\on{geom}}$-linear and in particular $\NR$-linear.
\end{proof}

\subsubsection{} We are now in a position to prove the existence of Hecke eigensheaves for any reductive group.
\begin{thm}\label{t:Heckeexistsectfour}
Let $G$ be a connected reductive algebraic group, and let $\sigma\in \on{LocSys}_G(U,\chi_S)^{\on{irr}}$ be an irreducible regular singular $G$-local system on $U=X-S$. Then there exists a Whittaker-normalized Hecke eigensheaf with eigenalue $\sigma$.
\end{thm}

\begin{proof}

\step
First, we show that if $G\to H$ is an isogeny,\footnote{Note that we swapped the roles of $G$ and $H$ compared to the rest of §\ref{s:adjoint}} and the induced $H$-local system $\sigma_H$ admits a Whittaker-normalized Hecke eigensheaf, then so does $\sigma$.
Let $\cF_{\sigma_H}\in D(\on{Bun}_{\check{H},\check{N}}(X,S))^{\check{T}_{H,S},\chi_S}$ be the Whittaker-normalized Hecke eigensheaf corresponding to $\sigma_H$. By definition, $\cF_{\sigma_H}$ defines an object of the category
\[
\on{QCoh}(\on{LocSys}_G(U,\chi_S))\underset{\on{QCoh}(\on{LocSys}_H(U,\chi_S))}{\otimes}  D(\on{Bun}_{\check{H},\check{N}}(X,S))^{\check{T}_{H,S},\chi_S}.
\]

\noindent Denote by $\check{Z}$ the finite group given by the kernel of $\check{H}\to \check{G}$. Since the restriction of $\chi_S$ to $\check{Z}$ trivializes, we obtain an action
\begin{equation}\label{eq:Zact}
D(\on{Bun}_{\check{Z}}(X))\curvearrowright D(\on{Bun}_{\check{H},\check{N}}(X,S))^{\check{T}_{H,S},\chi_S}.
\end{equation}

\noindent We denote by
\[
\on{Vect}\underset{D(\on{Bun}_{\check{Z}}(X))}{\otimes}D(\on{Bun}_{\check{H},\check{N}}(X,S))^{\check{T}_{H,S},\chi_S}
\]

\noindent the coinvariant category with respect to this action. Note that the functor
\[
f_{*,\on{dR}}^{\chi_S}:D(\on{Bun}_{\check{H},\check{N}}(X,S))^{\check{T}_{H,S},\chi_S}\to D(\on{Bun}_{\check{G},\check{N}}(X,S))^{\check{T}_{G,S},\chi_S}
\]

\noindent factors through a fully faithful embedding
\begin{equation}\label{eq:finffemb}
\on{Vect}\underset{D(\on{Bun}_{\check{Z}}(X))}{\otimes}D(\on{Bun}_{\check{H},\check{N}}(X,S))^{\check{T}_{H,S},\chi_S}\overset{f_{*,\on{dR}}^{\chi_S}}{\into} D(\on{Bun}_{\check{G},\check{N}}(X,S))^{\check{T}_{G,S},\chi_S}.
\end{equation}

Next, we claim that we have a natural equivalence
\begin{equation}\label{eq:claim}
\on{QCoh}(\on{LocSys}_G(U,\chi_S))\underset{\on{QCoh}(\on{LocSys}_H(U,\chi_S))}{\otimes}  D(\on{Bun}_{\check{H},\check{N}}(X,S))^{\check{T}_{H,S},\chi_S}
\end{equation}
\[
\simeq \on{Vect}\underset{D(\on{Bun}_{\check{Z}}(X))}{\otimes}D(\on{Bun}_{\check{H},\check{N}}(X,S))^{\check{T}_{H,S},\chi_S}.
\]

\noindent Indeed, we have a monoidal equivalence 
\begin{equation}\label{eq:gerbe}
D(\on{Bun}_{\check{Z}}(X))\simeq \on{QCoh}(\on{LocSys}_{\bB Z}(X))
\end{equation}

\noindent exchanging convolution on the left-hand side with tensor product on the right-hand side. Here $Z$ denotes the Cartier dual of $\check{Z}$. The isogeny $\check{H}\to \check{G}$ induces a map $H\to \bB Z$. In particular, we get a map
\[
\on{LocSys}_H(U,\chi_S)\to \on{LocSys}_{\bB Z}(X).
\]

\noindent Pulling back, we obtain an action
\[
\on{QCoh}(\on{LocSys}_{\bB Z}(X))\curvearrowright D(\on{Bun}_{\check{H},\check{N}}(X,S))^{\check{T}_{H,S},\chi_S}.
\]

\noindent By Theorem \ref{t:gerbygerby} below, this action coincides with (\ref{eq:Zact}) under the equivalence (\ref{eq:gerbe}). This proves the equivalence (\ref{eq:claim}).

Denote by $\on{coeff}_{\chi_S}^{\check{G}}$ (resp. $\on{coeff}_{\chi_S}^{\check{H}}$) the Whittaker coefficient functor for $\check{G}$ (resp. $\check{H}$) defined in §\ref{s:fouriercoeff}. By construction, the composition
\[
\on{QCoh}(\on{LocSys}_G(U,\chi_S))\underset{\on{QCoh}(\on{LocSys}_H(U,\chi_S))}{\otimes}  D(\on{Bun}_{\check{H},\check{N}}(X,S))^{\check{T}_{H,S},\chi_S}
\]
\[
\simeq \on{Vect}\underset{D(\on{Bun}_{\check{Z}}(X))}{\otimes}D(\on{Bun}_{\check{H},\check{N}}(X,S))^{\check{T}_{H,S},\chi_S}\into D(\on{Bun}_{\check{G},\check{N}}(X,S))^{\check{T}_{G,S},\chi_S}\xrightarrow{\on{coeff}_{\chi}^{\check{G}}} \on{Vect}
\]

\noindent coincides with
\[
\on{QCoh}(\on{LocSys}_G(U,\chi_S))\underset{\on{QCoh}(\on{LocSys}_H(U,\chi_S))}{\otimes}  D(\on{Bun}_{\check{H},\check{N}}(X,S))^{\check{T}_{H,S},\chi_S}\to 
\]
\[
\to D(\on{Bun}_{\check{H},\check{N}}(X,S))^{\check{T}_{H,S},\chi_S}\xrightarrow{\on{coeff}_{\chi_S}^{\check{H}}}\on{Vect},
\]

\noindent where the first functor is given by pushforward along $\on{LocSys}_G(U,\chi_S)\to \on{LocSys}_H(U,\chi_S)$.

It now follows from Corollary \ref{c:lin} that the image of $\cF_{\sigma_H}$ under the embedding (\ref{eq:finffemb}) is a Whittaker-normalized Hecke eigensheaf with eigenvalue $\sigma$.

\step
First, let $G$ be a semisimple algebraic group. Combining Step 1 and the existence of Hecke eigensheaves for adjoint groups, we deduce the existence of a Hecke eigensheaf for $\sigma$ in this case. 

Finally, let $G$ be arbitrary and consider the isogeny $G\to G^{\on{ad}}\times G^{\on{ab}}$. Since we know the existence of Hecke eigensheaves for the group $G^{\on{ad}}\times G^{\on{ab}}$, we are done by Step 1.
\end{proof}  

\subsubsection{} In the above proof, we used the following result of Gaitsgory-Raskin:
\begin{thm}[\cite{multone}]\label{t:gerbygerby}
The two natural actions
\[
D(\on{Bun}_{\check{Z}})\curvearrowright D(\on{Bun}_{\check{H},\check{N}}(X,S))^{\check{T}_{H,S},\chi_S},
\]

\noindent the first coming from the geometry of $\on{Bun}_{\check{H},\check{N}}(X,S)$, the second from the Hecke action and (\ref{eq:gerbe}), coincide.

\end{thm}
\begin{rem}\label{r:Heckecoh}
Let $\cF_{\sigma}$ be the Hecke eigensheaf constructed in the proof of Theorem \ref{t:Heckeexistsectfour}. It follows from Corollary \ref{c:pc} that the pullback of $\cF_{\sigma}$ to any quasi-compact open substack $U\subset \BunN$ is compact.
\end{rem}

\subsection{Relationship between automorphic and spectral rigidity}
In this subsection, we establish a relationship between automorphic and spectral rigidity as described in §\ref{s:autspecrigidity}. The main goal is to prove Theorem \ref{t:autspecrigidity}. The proof is an easy consequence of the existence of Hecke eigensheaves associated to regular singular irreducible local systems together with the results of \cite{bezrukavnikov2023dimensions}.

\subsubsection{} For a parabolic subgroup $P\subset G$, we denote by $I_P\subset G(O)$ the parahoric subgroup defined as the preimage of $P$ under the evaluation at $t=0$:
\[
G(O)\to G.
\]

\noindent For a dual parabolic $\check{P}\subset \check{G}$, the parahoric $I_{\check{P}}\subset \check{G}(O)$ is defined similarly.

\subsubsection{} Let us denote by a \emph{tamely ramified geometric automorphic datum} a tuple:
\begin{itemize}
    \item $\sL$ a rank 1 character sheaf on $\on{Bun}_{Z(\check{G}),1}(X,S)$.\footnote{This stack was defined in §\ref{s:centralstack}.}

    \item A collection $\check{P}_S=\lbrace \check{P}_x\rbrace_{x\in S}$, where $\check{P}_x\subset \check{G}$ is a parabolic subgroup.

    \item A collection $\chi_S=\lbrace \chi_x\rbrace_{x\in S}$ of rank 1 character sheaves on $\check{P}_x$.
\end{itemize}

\noindent We also denote by $\chi_x$ the induced character sheaf on $I_{\check{P}_x}$. Note that $\chi_x$ is pulled back from a character sheaf on $\check{M}_x^{\on{ab}}=\check{M}_x/[\check{M}_x,\check{M}_x]$, where $\check{M}_x$ denotes the Levi subgroup of $\check{P}_x$.

\subsubsection{} Recall that we say a local system $\sigma\in\on{LocSys}_G(U,\chi_S)$ is afforded by the geometric automorphic datum $(\sL, \check{Q}_S,\chi_S)$ if there exists a non-zero Hecke eigensheaf in the category
\[
D(\on{Bun}_{\check{G}}^{\infty\cdot S})^{(I_{\check{P}_S},\chi_S),\sL}
\]

\noindent with eigenvalue $\sigma$.

\subsubsection{}\label{s:properdef} Let $\lbrace \cO_x\rbrace_{x\in S}$ be the $k$-points of $\on{LocSys}_G(\overset{\circ}{D}_x)_{\chi_x}^{\on{RS}}$ determined by $\sigma$, and let $\tau_{\sigma}$ be the abelianization of $\sigma$. Write
\[
M_{\sigma}:=\on{LocSys}_G(U,\cO_S,\tau_{\sigma}),
\]

\noindent where the right-hand side is defined in §\ref{s:rigid}. We write $M_{\sigma}^{\on{irr}}$ for the open substack of $M_{\sigma}$ whose underlying local systems on $U$ are irreducible. We will prove:
\begin{thm}\label{t:autspecrigidity2}
Let $\sigma$ be an irreducible local system afforded by a tamely ramified geometric automorphic datum $(\sL,\check{P}_S,\chi_S)$. If $\sigma$ is the only irreducible local system afforded by $(\sL,\check{P}_S,\chi_S)$, then $M_{\sigma}^{\on{irr}}$ has a single $k$-point. Similarly, if the datum $ (\sL,\check{P}_S,\chi_S)$ only affords finitely many irreducible local systems, then $M_{\sigma}^{\on{irr}}$ is a disjoint union of finitely many connected components, each of which has a single $k$-point.
\end{thm}

\subsubsection{} We write $M_x, P_x$ for the Langlands dual groups of $\check{M}_x,\check{P}_x$. Note that $\chi_x$ corresponds to a $k$-point of $Z(M_x)$, the center of the $M_x$, which we similarly denote by $\chi_x$.

\subsubsection{} Consider the moduli stack
\[
\on{LocSys}_{P_x}(\overset{\circ}{D}_x)_{\chi_x}:=\on{LocSys}_{P_x}(\overset{\circ}{D}_x)\underset{\on{LocSys}_{M_x}(\overset{\circ}{D}_x)}{\times} \lbrace \chi_x\rbrace/M_x
\]

\noindent of $P_x$-local systems on the punctured disk whose underlying $M_x$-local system coincides with $\chi_x$.\\

\textbf{Proof of Theorem \ref{t:autspecrigidity2}}.

\emph{Step 1.} Let $(\sL, \check{P}_S,\chi_S)$ be a tamely ramified geometric automorphic datum and suppose $\sigma$ is the only local system on $U$ afforded by the datum. We want to show that $\sigma$ defines the only $k$-point of $M_{\sigma}^{\on{irr}}$. Suppose $\rho$ is another $k$-point. We claim that $\sigma$ and $\rho$ admit lifts to the stack
\[
\on{LocSys}_G(U)\underset{\underset{x\in S}{\prod}\on{LocSys}_G(\overset{\circ}{D}_x)}{\times}\prod_{x\in S}\on{LocSys}_{P_x}(\overset{\circ}{D}_x)_{\chi_x}.
\]

\noindent When all $\chi_x=0$, this follows from the spectral decomposition, Theorem \ref{t:action}, and \cite[Lemma 6.7]{bezrukavnikov2023dimensions}, which says that, up to cohomological shift, under Bezrukavnikov's equivalence of monoidal categories
\[
D(\check{I}\backslash \check{G}(K)/\check{I})^{\on{ren}}\simeq \on{IndCoh}(\fn/B\underset{\fg/G}{\times}\fn/B),
\]

\noindent the constant sheaf $\underline{k}_{\check{P}/\check{B}}$ goes to the structure sheaf of $\fn_P/B\underset{\fn_P/P}{\times} \fn_P/B$. Here, $\fn_P$ denotes the nilpotent radical of $\fp=\on{Lie}(P)$. 

In general, we need the statement of \cite[Lemma 6.7]{bezrukavnikov2023dimensions} in the case when $\chi_x\neq 0$. That is, if $\chi$ is a central element of a Levi $M$, we need to show that under the equivalence of Dhillon-Li-Yun-Zhu:
\[
D(\check{I},\chi\backslash \check{G}(K)/\check{I}, \chi)^{\on{ren}}\simeq \on{IndCoh}(\on{LocSys}_B(\overset{\circ}{D})_{\chi}\underset{\on{LocSys}_G(\overset{\circ}{D})}{\times}\on{LocSys}_B(\overset{\circ}{D})_{\chi}),
\]

\noindent the character sheaf $\chi\in D(\check{P}/\check{B},\chi)\into D(\check{I},\chi\backslash \check{G}(K)/\check{I}, \chi)$ corresponds (up to shifts) to the structure sheaf of 
\[
(\on{LocSys}_{P_x}(\overset{\circ}{D}_x)_{\chi_x}\underset{\bB P}{\times} \bB B)\underset{\on{LocSys}_{P_x}(\overset{\circ}{D}_x)_{\chi_x}}{\times} (\on{LocSys}_{P_x}(\overset{\circ}{D}_x)_{\chi_x}\underset{\bB P}{\times}\bB B).
\]

\noindent If $\chi$ is central in $G$, the result is immediate from \cite{bezrukavnikov2023dimensions}. In general, the paper \cite{dhillonendo} exhibits the category
\begin{equation}\label{eq:endoscopyequiv}
D(\check{I}_H,\chi\backslash \check{H}(K)/\check{I}_H, \chi)
\end{equation}

\noindent as a direct summand of $D(\check{I},\chi\backslash \check{G}(K)/\check{I}, \chi)$. Here, $H\subset G$ is the group defined as the connected component of the centralizer in $G$ of the element $\on{exp}(\chi)$ in $T$ defining $\chi$, and $\check{I}_H$ denotes the Iwahori subgroup of $\check{H}(O)$. In particular, $\on{exp}(\chi)$ is a central element of $H$. It follows from the construction \cite{dhillonendo} that under the inclusion $D(\check{I}_H,\chi\backslash \check{H}(K)/\check{I}_H, \chi)\subset D(\check{I},\chi\backslash \check{G}(K)/\check{I}, \chi)$, the sheaf $\chi\in D(\check{P}\cap \check{H}/\check{B}\cap \check{H},\chi)$ goes to  $\chi\in D(\check{P}/\check{B},\chi)$. Since the corresponding assertion on the spectral side is immediate, we are reduced to $H=G$, which we considered above.\\

\emph{Step 2.} Let $\cF_{\rho}\in D(\on{Bun}_{\check{G}}^{\infty\cdot S})^{(\check{I}_S,\chi_S),\sL}=D(\on{Bun}_{\check{G},\check{N}}(X,S))^{(\check{T}_S,\chi_S),\sL}$ be the Hecke eigensheaf corresponding to $\rho$ guaranteed by Theorem \ref{t:Heckeexistsectfour}. From the spectral decomposition, it follows that $\cF_{\rho}$ lies in the cocompletion of the image of the map
\[
D(\on{Bun}_{\check{G}}^{\infty\cdot S})^{(\check{I}_S,\chi_S),\sL}\to D(\on{Bun}_{\check{G}}^{\infty\cdot S})^{(\check{I}_S,\chi_S),\sL},\;\; \cF\mapsto (\underset{x\in S}{\boxtimes}\cO_{\on{LocSys}_{P_x}(\overset{\circ}{D}_x)_{\chi_x}})\star \cF.
\]

\noindent By the generalization of \cite[Lemma 6.7]{bezrukavnikov2023dimensions} to general $\chi$ described in Step 1, it follows that $\cF_{\rho}$ lies in the cocompletion of the image of the action map
\[
\underset{x\in S}{\bigotimes} D(\check{I}_x,\chi_x\backslash \check{G}(K_x)/\check{I}_x, \chi_x)\otimes D(\on{Bun}_{\check{G}}^{\infty\cdot S})^{(\check{P}_S,\chi_S),\sL}\to
\]
\[
\xrightarrow{\on{id}\otimes \on{oblv}} \underset{x\in S}{\bigotimes} D(\check{I}_x,\chi_x\backslash \check{G}(K_x)/\check{I}_x, \chi_x)\otimes D(\on{Bun}_{\check{G}}^{\infty\cdot S})^{(\check{I}_S,\chi_S),\sL}\xrightarrow{-\star -} D(\on{Bun}_{\check{G}}^{\infty\cdot S})^{(\check{I}_S,\chi_S),\sL}.
\]

\noindent In particular, we may find $\underset{x\in S}{\boxtimes} \cH_x\in \underset{x\in S}{\bigotimes} D(\check{I}_x,\chi_x\backslash \check{G}(K_x)/\check{I}_x, \chi_x)$ such that
\begin{equation}\label{eq:altHeckeeigenshf}
\on{Av}_*^{(\check{I}_S,\chi_S),\sL)\to (\check{P}_S,\chi_S),\sL}(\underset{x\in S}{\boxtimes} \cH_x\star \cF_{\rho})\in D(\on{Bun}_{\check{G}}^{\infty\cdot S})^{(\check{P}_S,\chi_S),\sL}
\end{equation}

\noindent is non-zero. Then (\ref{eq:altHeckeeigenshf}) is a Hecke eigensheaf with eigenvalue $\rho$. By assumption, this forces $\rho=\sigma$. Thus, $M_{\sigma}^{\on{irr}}$ only has one $k$-point.\\

\emph{Step 3.} Finally, suppose the datum  $(\sL,\check{P}_S,\chi_S)$ affords finitely many local systems (with $\sigma$ being one of them). We want to show that $M_{\sigma}^{\on{irr}}$ is a disjoint union of finitely many stacky points. First, let us show that $M_{\sigma}^{\on{irr}}$ only has finitely many $k$-points. This is an immediate consequence of the analysis above. Namely, we showed that if $\rho\in M_{\sigma}^{\on{irr}}(k)$, then $\rho$ is afforded by the datum $(\sL,\check{P}_S,\chi_S)$. Finally, as noted in Remark \ref{r:cohorigid}, for any $\rho\in M_{\sigma}^{\on{irr}}(k)$, the corresponding embedding
\[
\on{pt}/\on{Stab}_G(\rho)\into M_{\sigma}^{\on{irr}}
\]

\noindent is closed. This implies that the connected components of $M_{\sigma}^{\on{irr}}$ are in bijection with $M_{\sigma}^{\on{irr}}(k)$.

\qed

\begin{rem}
While convenient, it is not necessary to use the spectral decomposition in the above argument. Namely, using \cite{bezrukavnikov2023dimensions} (and Conjecture \ref{c:linearity}), it follows from the construction of Hecke eigensheaves in §\ref{s:Heckeeigensheaf} that for two irreducible local systems with the same local monodromies and abelianization, one is afforded by a given geometric automorphic datum if and only if the other one is.
\end{rem}

\section{Motivic realization of rigid local systems} 

In this section, we prove Theorem \ref{t:main}. The main inputs are the spectral decomposition given by Theorem \ref{t:action} and the existence of Hecke eigensheaves associated to irreducible regular singular local systems.

\subsection{The category of motivic sheaves} 

\subsubsection{} Let $X$ be a smooth variety.\footnote{By variety, we mean a quasi-separated scheme of finite type.} We let $D_0(X)$ be the category of D-modules on $X$ whose singular support is contained in the zero section of $X$. In other words, $D_0(X)$ is the category of D-modules $\cF\in D(X)$ such that each $H^n(\cF)$ is a filtered colimit of vector bundles with a flat connection. The category $D_0(X)$ carries a natural t-structure making the inclusion $D_0(X)\into D(X)$ t-exact.

We will refer to coherent D-modules $\sL\in D_0(X)^{\heartsuit}$ as local systems on $X$.

\subsubsection{} Recall the usual notion of an irreducible local system on a smooth variety being motivic.
\begin{defin}\label{d:mot}
Let $X$ be a smooth variety, and let $\sL\in D_0(X)^{\heartsuit}$ be an irreducible local system. We say that $\sL$ is \emph{motivic} if there exists a smooth Zariski-open dense subscheme $j: U\into X$ and a smooth proper family $f:W\to U$ such that $\sL_{\vert U}$ appears as a subquotient of $H^n(f_{*,\on{dR}}(k_W))$ for some $n\in \bZ$.
\end{defin}

\noindent In other words, an irreducible local system is motivic if on a smooth Zariski-dense open subset $U\subset X$, the restriction of $\sL$ to $U$ appears in the Gauss-Manin connection corresponding to a smooth proper family.

\begin{rem}
    It follows from a combination of Chow's lemma and the decomposition theorem that we may take the family to be projective instead of proper.
\end{rem}

\begin{rem}
    If we require that the motivic local systems extend to $X$; that is, requiring $U=X$ in the above definition, one obtains an a priori stronger notion of being motivic.\footnote{These are referred to as local systems smoothly of geometric origin in \cite[§2.5]{langer2018rank}.} Our main theorem does not show that irreducible rigid local systems on curves with quasi-unipotent monodromies and torsion abelianization are motivic in this stronger sense.
\end{rem}

\subsubsection{} Next, we need to extend the definition of being to motivic to irreducible holonomic D-modules in general.
\begin{defin}\label{d:geom2}
Let $X$ be a (possibly singular) variety, and let $\cF\in D(X)^{\heartsuit}$ be an irreducible holonomic D-module. Write $\cF=i_{!*}(\sL)$ for some smooth locally closed subvariety $i:Z\into X$ and some irreducible local system $\sL$ on $Z$. We say that $\cF$ is motivic if $\sL$ is motivic in the sense of Definition \ref{d:mot}.
\end{defin}

\subsubsection{} For a variety $X$, we let $D(X)^{\on{mot}}$ be the category of D-modules $\cF\in D(X)$ such that any irreducible subquotient $\cG$ of $H^n(\cF)$ is motivic in the above sense. Note that $D(X)^{\on{mot}}$ is closed under finite colimits, shifts, and taking subobjects.

For $\cF\in D(X)$, we say that $\cG\in D(X)^{\heartsuit}$ is a subquotient of $\cF$ if $\cG$ is a subquotient of $H^n(\cF)$ for some $n\in \bZ$.

\subsubsection{}\label{s:motforst} Henceforth, we assume all algebraic stacks (in the sense of \cite[§1.1.3]{drinfeld2013some}) are locally of finite type. For such an algebraic stack $\sY$, we let $D(\sY)^{\on{mot}}$ be the category consisting of $\cF\in D(\sY)$ such that for any smooth surjective map from a variety $f:Y\to \sY$, $f^!(\cF)$ lies in $D(Y)^{\on{mot}}$. Since the pullback of a motivic D-module by any map between varieties is motivic (cf. Lemma \ref{l:pullback} below), this definition extends the notion for varieties.

We refer to $D(\sY)^{\on{mot}}$ as the category of motivic D-modules on $\sY$.
\begin{lem}\label{l:comptstruc}
The category $D(\sY)^{\on{mot}}$ is closed under filtered colimits.
\end{lem}
\begin{proof}
By definition, we are reduced to the case when $\sY=Y$ is a variety. Thus, let $\cF_i\in D(Y)^{\on{mot}}$ be a family of motivic D-modules, and let $\cF:=\underset{i}{\on{colim}}\; \cF_i$. Since the t-structure on $D(Y)$ is compatible with filtered colimits (see \cite[§4]{gaitsgory2011crystals}), we may assume that each $\cF_i$ is in the heart of the t-structure on $D(Y)$. Then the lemma follows from the fact that if we have a triangle
\[
\cG_1\to \cG_2\to \cG_3,
\]

\noindent where $\cG_1$ and $\cG_3$ are motivic, then so is $\cG_2$.

\end{proof}

\subsubsection{}\label{s:coherent} Recall (cf. \cite[§7.3]{drinfeld2013some}) that a D-module $\cF\in D(\sY)$ is called \emph{coherent} if for any smooth map from an affine scheme $f:Y\to \sY$, the sheaf $f^!(\cF)\in D(Y)$ is compact. In this case, the Verdier dual $\bD(\cF)$ is also coherent.
\begin{lem}\label{l:verdiermot}
Let $\cF\in D(\sY)^{\on{mot}}$ be coherent. Then $\bD(\cF)\in D(\sY)$ is motivic.
\end{lem}
\begin{proof}
Let $f:Y\to \sY$ be a smooth surjective map with $Y$ affine. By definition of $D(\sY)^{\on{mot}}$, we need to check that $f^!(\bD(\cF))$ lies in $D(Y)^{\on{mot}}$. By \cite[§7.3.4]{drinfeld2013some}, we have
\[
f^!(\bD(\cF))\simeq \bD(f^{*,\on{dR}}(\cF)).
\]

\noindent Since $f$ is smooth, the functors $f^!$ and $f^{*,\on{dR}}$ agree up to a cohomological shift (on each connected component of $Y$). Thus, we are reduced to the case when $Y=\sY$ is a variety.

The functor $\bD: D(Y)^c\to D(Y)^{c,\on{op}}$ is t-exact, and so we may assume that $\cF\in D(Y)^{\heartsuit}$ is irreducible. Write $\cF=j_{!*}(\sL)$ for some local  system $\sL$ on a smooth locally closed subscheme $U\subset Y$ that appears in the cohomology of a smooth proper family $\pi:W\to U$. Then
\[
\bD(\cF)=\bD(j_{!*}\sL)\simeq j_{!*}(\bD(\sL)).
\]

\noindent Since $\bD(\sL)$ also appears in the cohomology of the family $\pi:W\to U$, this finishes the lemma.
\end{proof}

\subsubsection{} The goal of this subsection is to prove that the categories of motivic D-modules are stable under sheaf-theoretic operations.
\begin{prop}\label{p:converse}
For a schematic and quasi-compact map $f$, the categories of motivic D-modules are stable under the six functors $\lbrace f_{*,\on{dR}},f^!,f^{*\on{dR}},f_!, -\boxtimes -, \bD\rbrace$.\footnote{For the functor of Verdier duality, it only really makes sense to say that it preserves coherent motivic D-modules.}
\end{prop}

\subsubsection{} Let us start by observing the following:

\begin{lem}\label{l:reducetosimple}
Let $f:\sX\to \sY$ be a schematic and quasi-compact map of algebraic stacks. To prove Proposition \ref{p:converse}, we may assume that $\sX=X$ and $\sY=Y$ are varieties.
\end{lem}

\begin{proof}
It is clear that the functor $-\boxtimes -$ preserves motivic objects in general. For the functor $\bD$, this is the content of Lemma \ref{l:verdiermot}.

Thus, we need to prove that the functors $f_{*,\on{dR}},f_!,f^!,f^{*,\on{dR}}$ preserve motivic objects assuming Proposition \ref{p:converse} holds for maps between varieties.

Let $p:Y\to \sY$ be a smooth surjective map with $Y$ a variety. Form the Cartesian diagram
\[\begin{tikzcd}
	{X=\sX\underset{\sY}{\times}Y} && Y \\
	\\
	\sX && \sY.
	\arrow["p", from=1-3, to=3-3]
	\arrow["f"', from=3-1, to=3-3]
	\arrow["{f'}", from=1-1, to=1-3]
	\arrow["q"', from=1-1, to=3-1]
\end{tikzcd}\]

\noindent Note that $X$ is a variety because $f$ is schematic and quasi-compact. 

\step For $\cF\in D(\sX)^{\on{mot}}$, let us start by proving that $f_{*,\on{dR}}(\cF)$ is motivic. We need to check that $p!f_{*,\on{dR}}$ lies in $D(Y)^{\on{mot}}$. By base-change, we have
\[
p^!f_{*,\on{dR}}(\cF)\simeq (f')_{*,\on{dR}}q^!(\cF).
\]

\noindent Since $q^!(\cF)$ is motivic by definition, we have $(f')_{*,\on{dR}}q^!(\cF)\in D(Y)^{\on{mot}}$, since $X$ and $Y$ are varieties.

In exactly the same way, one proves that $f_!$ (whenever defined) preserves motivic objects if this is the case for maps between varieties.

\step Next, let us prove that $f^!(\cF)\in D(Y)^{\on{mot}}$. By assumption, we know that
\[
(f\circ q)^!(\cF)\simeq (p\circ f')^!(\cF)\in D(X)^{\on{mot}}.
\]

\noindent By definition, we need to check that for any smooth surjective map $g:Z\to \sX$ with $Z$ a variety, we have $g^!\circ f^!(\cF)\in D(Z)^{\on{mot}}$. Form the Cartesian diagram
\[\begin{tikzcd}
	{W} && X \\
	\\
	Z && \sX.
	\arrow["q", from=1-3, to=3-3]
	\arrow["g"', from=3-1, to=3-3]
	\arrow["{}", from=1-1, to=1-3]
	\arrow["h"', from=1-1, to=3-1]
\end{tikzcd}\]

\noindent Since $h:W\to Z$ is smooth and surjective, it suffices to check that $(f\circ g\circ h)^!(\cF)$ is motivic, cf. \cite[Prop. 2.10]{langer2018rank}. However, this follows from the fact that $W\to X\to Y$ is a map of varieties.

The same argument shows that $f^{*,\on{dR}}$ preserves motivic objects.

\end{proof}

\subsubsection{} To prove Proposition \ref{p:converse}, we will consider pushforwards and pullbacks independently. Let us begin by proving that pushforwards preserve motivicity. 

\subsubsection{} We divide the proof of Proposition \ref{p:converse} into three parts.
\begin{lem}\label{l:dualizing}
Let $f:X\to Y$ be a map of varieties. Then $f_{*,\on{dR}}(\omega_X)$ is motivic. 
\end{lem}

\begin{proof}
We may assume that $f$ is dominant. Indeed, if not, we replace $Y$ by $\on{Im}(f)$ and use that pushforward along closed embeddings preserve motivic sheaves.

We induct on $n=\on{dim}X$, the case $n=0$ being trivial, and assume that $X$ is connected.

\step Assume first that $X$ is smooth and that $f$ is proper. Since $f$ is dominant, $Y$ is irreducible.

Thus, let $\cF$ be an irreducible subquotient of $f_{*,\on{dR}}(\omega_X)$. By the decomposition theorem, $\cF$ is a direct summand of $f_{*,\on{dR}}(\omega_X)$ (up to shifts). Let $j': V\into Y$ be a smooth connected locally closed subscheme such that
\[
\cF\simeq (j')_{!*}(\sL)
\]

\noindent for some irreducible local system $\sL$ on $V$. If $\overline{V}$ is a proper subscheme of $Y$, let $Z=f^{-1}(\overline{V})$. Then $\cF\in D(\overline{V})$ is a direct summand of $f_{*,\on{dR}}(\omega_Z)$. Since $\on{dim} Z<n$, this forces $\cF$ to be motivic by the induction hypothesis.

Thus, we may assume that $V\subset Y$ is a smooth Zariski-dense open subscheme. Further restricting to the smooth locus of $Y$, we may assume that $Y$ itself is smooth and that $V=Y$. In this case, $f$ is a dominant map between smooth varieties, and so there exists some non-empty smooth open $V'\subset Y$ (automatically dense by irreducibility of $Y$)  such that $f$ is smooth over $V'$. Hence $\sL_{\vert V'}$ occurs in the cohomology of a smooth and proper family.

\step Assume in this step that $X$ is smooth but $f$ is not necessarily proper. By Nagata's compactification theorem, we may factor $f$ as
    
    \[\begin{tikzcd}
	X && {\overline{X}} \\
	\\
	Y,
	\arrow["j", from=1-1, to=1-3]
	\arrow["f"', from=1-1, to=3-1]
	\arrow["{\overline{f}}", from=1-3, to=3-1]
\end{tikzcd}\]

\noindent where $\overline{f}$ is proper and $j$ is an open embedding. Denote by $i: Z=\overline{X}\setminus X\into \overline{X}$. We have a triangle
\begin{equation}\label{eq:loctriang1}
(\overline{f}\circ i)_{*,\on{dR}}i^!\on{IC}_{\overline{X}}\to \overline{f}_{*,\on{dR}}\on{IC}_{\overline{X}}\to f_{*,\on{dR}}\omega_X[-n].
\end{equation}

\noindent Here, we have used that $X$ is smooth so that $\on{IC}_X\simeq \omega_X[-n]$.

\noindent Choose a resolution of singularities $\pi: \widetilde{X}\to \overline{X}$. By the decomposition theorem, $\on{IC}_{\overline{X}}$ occurs as a direct summand of $\pi_{*,\on{dR}}(\omega_{\widetilde{X}})$ up to shifts. Denoting by $\widetilde{Z}=\pi^{-1}(Z)$, we similarly see that $i^!\on{IC}_{\overline{X}}$ occurs as a direct summand of $\pi_{*,\on{dR}}(\omega_{\widetilde{Z}})$. Since $\on{dim}\widetilde{Z}<n$, it follows that $(\overline{f}\circ i)_{*,\on{dR}}i^!\on{IC}_{\overline{X}}$ lies in $D(Y)^{\on{mot}}$ by the induction hypothesis. From (\ref{eq:loctriang1}), we conclude that $f_{*,\on{dR}}(\omega_X)$ is motivic if and only if $(\bar{f}\circ\pi)_{*,\on{dR}}(\omega_{\widetilde{X}})$ is motivic. However, $(\overline{f}\circ \pi)$ is proper and $\widetilde{X}$ is smooth, and so the latter is motivic by Step 1 above.

\step Finally, consider the arbitrary case. Let $j:X^{\on{sm}}\into X$ be the smooth locus and $Z=X\setminus X^{\on{sm}}$ the complement. We have a triangle
\[
(f\circ i)_{*,\on{dR}}\omega_Z\to f_{*,\on{dR}}\omega_X\to (f\circ j)_{*,\on{dR}}\omega_{X^{\on{sm}}}.
\]

\noindent The sheaf $(f\circ i)_{*,\on{dR}}\omega_Z$ is motivic, since $\on{dim}Z<n$. Moreover, since $X^{\on{sm}}$ is smooth, the sheaf $(f\circ j)_{*,\on{dR}}\omega_{X^{\on{sm}}}$ is motivic by Step 2. Thus, $f_{*,\on{dR}}\omega_X$ is motivic as well.

\end{proof}
\begin{lem}\label{l:pushfowards}
    Let $f: X\to Y$ be a map of varieties. Then $f_{*,\on{dR}}$ sends $D(X)^{\on{mot}}$ to $D(Y)^{\on{mot}}$.
\end{lem}

\begin{proof}
We induct on $n=\on{dim} X$ and assume that $X$ is connected.

\step Let $\cF\in D(X)^{\heartsuit}$ be irreducible and motivic. Let $j: U\into X$ be a smooth connected locally closed subscheme such that $\cF\simeq j_{!*}(\sL)$ for some motivic local system $\sL$ on $U$. Replacing $U$ by a smooth open subscheme, we may assume that $\sL$ occurs in the cohomology of a smooth proper family $\pi: W\to U$. Moreover, replacing $X$ by $\overline{U}$, we may assume that $U\subset X$ is Zariski-dense and open. Let $i: Z=X\setminus U\into X$ be the complement. We have a triangle
\[
i_{*,\on{dR}}i^!\cF\to \cF\to j_{*,\on{dR}}\sL.
\]

\noindent Note that $j_{*,\on{dR}}\sL$ is motivic. Indeed, by the decomposition theorem, $j_{*,\on{dR}}\sL$ is a direct summand of $(j\circ \pi)_{*,\on{dR}}(k_W)$ (up to shifts), and the latter is motivic by Lemma \ref{l:dualizing}. Thus, it follows from the above triangle that $i_{*,\on{dR}}i^!\cF$ is motivic. 

By the induction hypothesis, so is $(f\circ i)_{*,\on{dR}}i^!\cF$. Similarly, $(f\circ j)_{*,\on{dR}}\sL$ is motivic by Lemma \ref{l:dualizing} since it occurs as a direct summand of $(f\circ j\circ \pi)_{*,\on{dR}}k_W$. From the triangle
\[
(f\circ i)_{*,\on{dR}}i^!\cF\to f_{*,\on{dR}}\cF\to (f\circ j)_{*,\on{dR}}\sL,
\]

\noindent we conclude that $f_{*,\on{dR}}\cF$ is motivic as well.

\step 

If $\cF\in D(X)^{\on{mot}}$ is bounded from below, we may write $\cF$ as a filtered colimit of its irreducible subquotients. By Step 1 and Lemma \ref{l:comptstruc}, $f_{*,\on{dR}}(\cF)$ is motivic in this case.

Finally, let $\cF\in D(X)^{\on{mot}}$ be arbitrary. Let $\cG$ be an irreducible subquotient of $H^n(f_{*,\on{dR}}\cF)$ for some $n\in\bZ$. By \cite[Prop. 7.6.8]{drinfeld2013some}, there exists $m\ll 0$ such that $\cG$ is an irreducible subquotient of 
\[
H^n(f_{*,\on{dR}}\tau^{\geq m}(\cF)).
\]

\noindent Since $\tau^{\geq m}(\cF)$ is bounded from below, the lemma follows.

\end{proof}

\subsubsection{} Next, we prove that $!$-pullback preserves the property of being motivic.
\begin{lem}\label{l:pullback}
Let $f:X\to Y$ be a map of varieties. Then $f^!$ sends $D(Y)^{\on{mot}}$ to $D(X)^{\on{mot}}$.
\end{lem}

\begin{proof}
We induct on $n=\on{dim} Y$ and assume that $Y$ is connected. Note that for $n=0$, the statement becomes that $\omega_X$ is motivic. This is a particular case of Lemma \ref{l:dualizing}.

\step Assume first that $\cF\in D(Y)^{\on{mot}}$ is an irreducible D-module. Write $\cF=j_{!*}(\sL)$, where $j:V\into Y$ is a smooth locally closed subscheme, and $\sL$ is a local system on $V$ appearing in the cohomology of a smooth proper family $\pi: W\to V$. Replacing $Y$ by $\overline{V}$, we may assume that $V$ is a Zariski-dense and open in $Y$.

We need to show that $f^!\cF$ lies in $D(X)^{\on{mot}}$. Let $i: Z\into Y$ be the complement of $V$. We have a triangle
\[
i_{*,\on{dR}}i^!\cF\to \cF\to j_{*,\on{dR}}\sL.
\]

\noindent By Lemma \ref{l:pushfowards}, we see that $j_{*,\on{dR}}\sL$ is motivic. Hence, so are the sheaves $i_{*,\on{dR}}i^!\cF$ and $i^!\cF$. Note that $\on{dim}Z<n$. Applying base-change along the diagram
\[\begin{tikzcd}
	{X\underset{Y}{\times}Z} && Z \\
	\\
	X && Y,
	\arrow[from=1-3, to=3-3]
	\arrow[from=3-1, to=3-3]
	\arrow[from=1-1, to=1-3]
	\arrow[from=1-1, to=3-1]
\end{tikzcd}\]

\noindent it follows that $$f^!i_{*,\on{dR}}i^!\cF$$ is motivic by the induction hypothesis. From the Decomposition Theorem, the sheaf $f^!j_{*,\on{dR}}\sL$ is a direct summand of $f^!(j\circ \pi)_{*,\on{dR}}k_W$ and hence motivic by base-change and Lemma \ref{l:dualizing}. Moreover, from the triangle
\[
f^!i_{*,\on{dR}}i^!\cF\to f^!\cF\to f^!j_{*,\on{dR}}\sL,
\]

\noindent we conclude that $f^!\cF$ is motivic as well.

\step Next, if $\cF\in D(X)^{\on{mot}}$ is bounded from below, we may write $\cF$ as a filtered colimit of its irreducible subquotients. Thus, $f^!(\cF)$ is motivic by Step 1 and Lemma \ref{l:comptstruc}.

Finally, let $\cF\in D(X)^{\on{mot}}$ be arbitrary. Let $\cG$ be an irreducible subquotient of $H^n(f^!(\cF))$ for some $n\in \bZ$. As in the proof of \ref{l:pushfowards}, there exists some $m\ll 0$ such that $\cG$ is a subquotient of the sheaf $H^n(f^!(\tau^{\geq m}(\cF)))$. Since $\tau^{\geq m}(\cF)$ is bounded from below, the lemma follows.
\end{proof}

\textbf{Proof of Proposition \ref{p:converse}} 

We know that $*$-pushforward and $!$-pullback of maps between varieties preserve the property of being motivic. By Lemma \ref{l:verdiermot}, so does Verdier duality. Thus, we see that $!$-pushforward and $*$-pullback of maps between varieties similarly preserve motivic sheaves. Thus, we conclude by Lemma \ref{l:reducetosimple}.

\qed

\subsection{Putting it together}
In this subsection, we prove Theorem \ref{t:main}.

\subsubsection{} Let $X$ be a smooth projective curve and $S=\lbrace x_1,...,x_n\rbrace \subset X(k)$ a finite set of $k$-points of $X$. Let $\sigma$ be an irreducible regular singular rigid $G$-local system on $U=X-S$ with quasi-unipotent monodromy around each $x_i\in S$. Moreover, assume that the abelianization $\tau$ of $\sigma$ (i.e. the image of $\sigma$ under $G\to G^{\on{ab}}$) is torsion.

\subsubsection{} For each $x\in S$, let $\chi_x\in \ft//\widetilde{W}^{\on{aff}}$ be the eigenvalue of the residue of $\sigma$ around $x$. We similarly write $\chi_x\in D(\check{T})$ for the corresponding character sheaf on $\check{T}$. The assumption that the monodromies are quasi-unipotent forces $\chi_x\in D(\check{T})$ to have finite order; that is, $\chi_x^{\otimes n}$ is trivial for some $n\gg 0$. In particular, the local system $\chi_x$ is motivic.

\subsubsection{} We remind that we let $\chi_x^{\on{ab}}$ be the restriction of $\chi_x$ to $Z(\check{G})$. Write
\[
\sL_{\tau}\in D(\on{Bun}_{Z(\check{G})^{\circ},1}(X,S))^{Z(\check{G})^{\circ}_S,\chi^{\on{ab}}_S}
\]

\noindent for the character sheaf corresponding to the abelianization $\tau$ (see §\ref{s:charsheaf} for details). The assumption that $\tau$ is torsion similarly forces $\sL_{\tau}$ to have finite order, and hence to be motivic.

\subsubsection{} By definition of rigidity (see Section \ref{s:rigid}), there exist $\cO_{x}\in \on{LocSys}_G(\overset{\circ}{D}_{x})^{\on{RS}}_{\chi_{x}}$ such that $\sigma$ defines an isolated point in the stack
\[
\LSsigma.
\]

\noindent We remind that
\[
\LSsigma=\LS\underset{\underset{x\in S}{\prod}\on{LocSys}_G(\overset{\circ}{D}_x)^{\on{RS}}_{\chi_x}}{\times} \prod_{x\in S} \bB\on{Aut}(\cO_x).
\]

\subsubsection{} Recall the notation of Section \ref{s:twist}. By (\ref{eq:acttwist}), we have an action
\[
\on{QCoh}(\LS)\curvearrowright D(\BunN)^{(\check{T}_S,\chi_S),\sL_{\tau}}.
\]

\noindent This induces an action
\begin{equation}\label{eq:act3}
\on{QCoh}(\LSsigma)\curvearrowright \on{Vect} \underset{\on{QCoh}(\underset{x\in S}{\prod}\on{LocSys}_G(\overset{\circ}{D}_x)^{\on{RS}}_{\chi_x})}{\otimes}D(\BunN)^{(\check{T}_S,\chi_S),\sL_{\tau}},
\end{equation}

\noindent where the action $\on{QCoh}(\underset{x\in S}{\prod}\on{LocSys}_G(\overset{\circ}{D}_x)^{\on{RS}}_{\chi_x})\curvearrowright \on{Vect}$ is induced by the point
\[
\prod_{x\in S} \cO_x\in \prod_{x\in S}\on{LocSys}_G(\overset{\circ}{D}_x)^{\on{RS}}_{\chi_x}(k).
\]

\subsubsection{} Let $\cF_{\sigma}$ be the non-zero Hecke eigensheaf with eigenvalue $\sigma$ guaranteed by Theorem \ref{t:Heckeexistsectfour}. By definition, we have:
\begin{equation}\label{eq:heckeigenintp}
\cF_{\sigma}\in\on{Vect} \underset{\on{QCoh}(\underset{x\in S}{\prod}\on{LocSys}_G(\overset{\circ}{D}_x)^{\on{RS}}_{\chi_x})}{\otimes}D(\BunN)^{(\check{T}_S,\chi_S),\sL_{\tau}}.
\end{equation}

\subsubsection{Delta sheaves} Let us briefly introduce some notation needed in the proof of Theorem \ref{t:main'} below.

For a $k$-point $\sP: \on{Spec}(k)\to \BunN$, we denote by $\delta_{\sP}=\sP_!(k)\in D(\BunN)$ the corresponding !-delta sheaf. Write 
\[
\widetilde{\delta}_{\sP}=\on{Av}_!^{(\check{T}_S,\chi_S),\sL_{\tau}}(\delta_{\sP})\in D(\BunN)^{(\check{T}_S,\chi_S),\sL_{\tau}},
\]

\noindent where $\on{Av}_!^{(\check{T}_S,\chi_S),\cL^{\tau}}$ denotes the partially defined left adjoint to the forgetful functor
\[
\on{oblv}^{(\check{T}_S,\chi_S),\sL_{\tau}}: D(\BunN)^{(\check{T}_S,\chi_S),\sL_{\tau}}\to D(\BunN).
\]
\begin{rem}
We emphasize that $\widetilde{\delta}_{\sP}$ is very concretely described. For example if each $\chi_x=0$, then $\widetilde{\delta}_{\sP}$ is given by $!$-pushforward of $\sL_{\tau}$ along the map
\[
\on{Bun}_{Z(\check{G})}(X)\to \on{Bun}_{\check{G},\check{B}}(X,S)
\]

\noindent given by acting on $\sP$. In general, $\widetilde{\delta}_{\sP}$ will be a twisted version of such a pushforward.
\end{rem}

\subsubsection{} It follows from the above description that $\widetilde{\delta}_{\sP}$ is motivic (when considered an object of $D(\BunN)$), since $\chi_x$ and $\sL_{\tau}$ are motivic. We will show:

\begin{thm}\label{t:main'}
There exists a motivic Hecke eigensheaf $\cG_{\sigma}\in D(\BunN)^{(\check{T}_S,\chi_S),\sL_{\tau}}$ with eigenvalue $\sigma$.

\end{thm}
\begin{proof}

\step Denote by $\cH$ the endofunctor
\[
D(\BunN)^{(\check{T}_S,\chi_S),\sL_{\tau}}\to \on{Vect} \underset{\on{QCoh}(\underset{x\in S}{\prod}\on{LocSys}_G(\overset{\circ}{D}_x)^{\on{RS}}_{\chi_x})}{\otimes}D(\BunN)^{(\check{T}_S,\chi_S),\sL_{\tau}}\to
\]
\[
\to D(\BunN)^{(\check{T}_S,\chi_S),\sL_{\tau}}
\]

\noindent induced by pull-push along $\on{Spec}(k)\xrightarrow{\underset{x\in S}{\prod}\cO_x} \underset{x\in S}{\prod} \on{LocSys}_G(\overset{\circ}{D}_x)^{\on{RS}}_{\chi_x}$. By Remark \ref{r:Heckecoh}, we may find $\sP\in \BunN(k)$ such that the $!$-fiber of the Hecke eigensheaf $\cF_{\sigma}$ (considered as an object of $D(\BunN)$) at $\sP$ is non-zero. By adjunction, we have 
\begin{equation}\label{eq:hom}
\on{Hom}_{D(\BunN)^{(\check{T}_S,\chi_S),\sL_{\tau}}}(\widetilde{\delta}_{\sP},\cF_{\sigma})\neq 0.
\end{equation}

\step The key point (and the only time we use the rigidity assumption) is that $\sigma$ defines an isolated point in $\LSsigma$, and so it follows from (\ref{eq:act3}) that we have a decomposition
\[
\cH(\widetilde{\delta}_{\sP})\simeq \cG_{\sigma}\oplus \cG_{\neq \sigma},
\]

\noindent where $\cG_{\sigma}$ is a Hecke eigensheaf with eigenvalue $\sigma$. We claim that $\cG_{\sigma}$ is non-zero. Indeed,
\[
\on{Hom}_{D(\BunN)^{(\check{T}_S,\chi_S),\sL_{\tau}}}(\cG_{\sigma},\cF_{\sigma})=\on{Hom}_{D(\BunN)^{(\check{T}_S,\chi_S),\sL_{\tau}}}(\cH(\widetilde{\delta}_{\sP}),\cF_{\sigma}),
\]

\noindent and the latter is non-zero by (\ref{eq:hom}) and (\ref{eq:heckeigenintp}).

\step Finally, we claim that $\cG_{\sigma}$ is motivic. Since $\widetilde{\delta}_{\sP}$ is motivic, we just need to show that the endofunctor $\cH$ from Step 1 preserves motivic objects. The latter is simply given by convolution with an object of the category
\[
\underset{x\in S}{\bigotimes} D(\check{I},\chi_x\backslash \check{G}(K)/\check{I},\chi_x).
\]

\noindent It follows from Proposition \ref{p:converse} and Lemma \ref{l:comptstruc} that convolution with an object of the category $D(\check{I},\chi_x\backslash \check{G}(K)/\check{I},\chi_x)$ preserves motivic objects.

\end{proof}

\subsubsection{} We are now in a position to prove the main theorem of this paper.

\textbf{Proof of Theorem \ref{t:main}}. Let $\cG_{\sigma}$ be as in Theorem \ref{t:main'}, and let $V$ be a finite-dimensional $G$-representation. We need to check that $\sigma_V=\sigma\overset{G}{\times} V$ lies in $D(U)^{\on{mot}}$, where $U=X-S$. Let 
\[
H_V: D(\BunN)^{(\check{T}_S,\chi_S),\sL_{\tau}}\to D(\BunN)^{(\check{T}_S,\chi_S),\sL_{\tau}}\otimes D(U)
\]

\noindent be the global Hecke functor associated to $V$. By Proposition \ref{p:converse}, we know that $H_V(\cG_{\sigma})$ is motivic. Since
\[
H_V(\cG_{\sigma})\simeq \cG_{\sigma}\boxtimes \sigma_V,
\]

\noindent it follows that $\sigma_V\in D(U)^{\on{mot}}$. Indeed, let $\sP\in \BunN(k)$ be a $k$-point where the $!$-fiber of $\cG_{\sigma}$ does not vanish.\footnote{That such a $k$-point exists follows from the fact that, by construction, $\cG_{\sigma}$ is coherent in the sense of §\ref{s:coherent}.} Pulling back $H_V(\cG_{\sigma})$ along $\lbrace \sP\rbrace \times U\into \BunN\times U$, we see that $\sigma_V\in D(U)^{\on{mot}}$ by Proposition \ref{p:converse}.

\qed

\newpage

\appendix

\section{Local systems on the punctured disk as a mapping stack}\label{app:A}

In this section, we recall Gaitsgory's description of the moduli stack of local systems on the punctured disk as a mapping stack, cf. \cite[Appendix A-B]{gaitsgory2014day}. See also \cite[§B.3]{Cstterm}.

\subsection{Fake disks} 

\subsubsection{} Let $\overset{\circ}{D_x}$ be the punctured disk at $x\in X$, cf. §\ref{s:disk}. One would like to say that $\on{LocSys}_G(\overset{\circ}{D_x})$ can be realized as:
\[
\on{LocSys}_G(\overset{\circ}{D_x})\simeq \on{Maps}((\overset{\circ}{D_x})_{\on{dR}},\bB G).
\]

\noindent However, since the prestack $(\overset{\circ}{D_x})_{\on{dR}}$ produces something non-sensical, this is not the right thing to do. By non-sensical, we mean that the commutative diagram
\[\begin{tikzcd}
	{\overset{\circ}{D_x}} && X \\
	\\
	{(\overset{\circ}{D_x})_{\textrm{dR}}} && {X_{\textrm{dR}}}
	\arrow[from=1-1, to=1-3]
	\arrow[from=1-3, to=3-3]
	\arrow[from=1-1, to=3-1]
	\arrow[from=3-1, to=3-3]
\end{tikzcd}\]

\noindent is not Cartesian. The same problem occurs for $D_x$. Instead, Gaitsgory defines a prestack $(\overset{\circ}{D_x})^{\on{fake}}_{\textrm{dR}}$, which we recall below, that does fit into a Cartesian diagram

\[\begin{tikzcd}
	{\overset{\circ}{D_x}} && X \\
	\\
	{(\overset{\circ}{D_x})^{\textrm{fake}}_{\textrm{dR}}} && {X_{\textrm{dR}}.}
	\arrow[from=1-1, to=1-3]
	\arrow[from=1-3, to=3-3]
	\arrow[from=1-1, to=3-1]
	\arrow[from=3-1, to=3-3]
\end{tikzcd}\]

\subsubsection{} Let $T$ be a test scheme and consider a map $x_I: T_{\on{red}}\to X^I$. Let $X_{\on{inf}}^{\bullet}$ be the Čech nerve of the map $X\to X_{\on{dR}}$. That is, the simplicial prestack:
\[\begin{tikzcd}
	\cdots & {X\underset{X_{\textrm{dR}}}{\times}X\underset{X_{\textrm{dR}}}{\times}X} && {X\underset{X_{\textrm{dR}}}{\times}X} && {X.}
	\arrow[shift right=1, from=1-4, to=1-6]
	\arrow[shift left=1, from=1-4, to=1-6]
	\arrow[shift right=2, from=1-2, to=1-4]
	\arrow[shift left=2, from=1-2, to=1-4]
	\arrow[from=1-2, to=1-4]
\end{tikzcd}\]

\noindent Here, $X\underset{X_{\on{dR}}}{\times} X$ identifies with the formal completion of $X\times X$ along the diagonal. Similarly, $X\underset{X_{\textrm{dR}}}{\times}X\underset{X_{\textrm{dR}}}{\times}X$ identifies with the formal completion of $X\times X\times X$ along the diagonal etc.

We denote by $X_{\on{inf},k}^{(n)}$ the $n$'th order neighbourhood of $X^{k+1}$ along its diagonal, and by $X_{\inf}^{\bullet, (n)}$ the pullback of $X_{\on{inf}}^{\bullet}$ to the $n$'th order neighbourhoods.

\subsubsection{} Form the new simplicial prestack
\begin{equation}\label{eq:infsimp}
(\hat{D}_{x_I})_{\on{dR}}^{\bullet}:=\big((\hat{D}_{x_I})_{\on{dR}}\underset{T_{\on{dR}}}{\times}T\big)\underset{X_{\on{dR}}}{\times}X_{\on{inf}}^{\bullet}.
\end{equation}

\noindent We denote by $(\hat{D}_{x_I})_{\on{dR}}^{\bullet, (n)}$ the further pullback of (\ref{eq:infsimp}) to $X_{\inf}^{\bullet, (n)}$.

\noindent Recall the affinization procedure in Section \ref{s:disk}. Apply this to $(\hat{D}_{x_I})_{\on{dR}}^{\bullet, (n)}$ to obtain a simplicial affine scheme
\[
(D_{x_I})_{\on{dR}}^{\bullet, (n)}.
\]

\noindent Note that $(\hat{D}_{x_I})_{\on{dR}}\simeq (\Gamma_{x_I})_{\on{dR}}$. In particular,
\[
(\hat{D}_{x_I})_{\on{dR}}\underset{T_{\on{dR}}}{\times}T
\]

\noindent identifies with the formal completion of $T$ along $\Gamma_{x_I}$.

\subsubsection{} We claim that the canonical map of simplicial prestacks
\[
\big((\hat{D}_{x_I})_{\on{dR}}\underset{T_{\on{dR}}}{\times}T\big)\underset{X_{\on{dR}}}{\times}X_{\on{inf}}^{\bullet, (n)}\to X_{\on{inf}}^{\bullet, (n)}
\]

\noindent factors through a map
\[
(D_{x_I})_{\on{dR}}^{\bullet, (n)}\to  X_{\on{inf}}^{\bullet, (n)}.
\]

\noindent This is immediate if $X$ is affine, since then $X_{\on{inf},k}^{(n)}$ is an affine scheme. In general, this still holds on very general grounds by \cite{bhatt2014algebraization}.

\subsubsection{} We put
\[
(D_{x_I})_{\on{dR}}^{\on{fake}}:=\underset{n\geq 1}{\on{colim}}\; \vert (D_{x_I})_{\on{dR}}^{\bullet, (n)}\vert.
\]

\noindent It follows from the construction (see \cite[\S B.3.8]{Cstterm}) that we have
\[
(D_{x_I})_{\on{dR}}^{\on{fake}}\underset{X_{\on{dR}}}{\times} X\simeq D_{x_I}.
\]

\subsubsection{Allowing for punctures} Let $S\subset X(k)$ be a finite set, and let $U=X-S$. Let $x_I: T\to X^I_{\on{dR}}$ be a map. We will consider the corresponding adic disk with punctures along $S$:
\[
(D_{x_I})_{\on{dR}}^{\on{fake}}\underset{X_{\on{dR}}}{\times} U_{\on{dR}}.
\]

\subsubsection{} When $x\in S$, we simply write
\[
(\overset{\circ}{D}_x)_{\on{dR}}^{\on{fake}}:=(D_{x})_{\on{dR}}^{\on{fake}}\underset{X_{\on{dR}}}{\times} U_{\on{dR}}.
\]

\noindent We similarly have
\[
(\overset{\circ}{D}_x)_{\on{dR}}^{\on{fake}}\underset{X_{\on{dR}}}{\times} X\simeq \overset{\circ}{D}_x.
\]

\subsubsection{} For $x\in S$, we may consider the mapping prestack
\[
\on{Maps}((\overset{\circ}{D}_x)_{\on{dR}}^{\on{fake}}, \on{pt}/G).
\]

\noindent Here, $\on{pt}/G$ denotes the \emph{prestack quotient}; that is, we do not apply étale sheafification to the quotient. We have the following lemma in \cite[App. B.3.3]{gaitsgory2014day}:
\begin{lem}\label{l:asheaf}
There is a canonical map
\[
\on{Maps}((\overset{\circ}{D}_x)_{\on{dR}}^{\on{fake}},\on{pt}/G)\to\on{LocSys}_G(\overset{\circ}{D_x})
\]

\noindent that becomes an isomorphism after étale sheafification of the left-hand side.
\end{lem}

\subsubsection{}\label{s:fin} Define the prestack
\[
\on{Jets}_{\on{horiz}}^{S-\on{mer}}(\bB G)^{\on{naive}}_{\on{Ran}}
\]

\noindent over $\on{Ran}_{X_{\on{dR}}}$ as follows. For a test scheme $T$, $I\in \on{fSet}$ and a map $x_I: T_{\on{red}}\to X^I$, the datum of a lift to $\on{Jets}_{\on{horiz}}^{S-\on{mer}}(\bB G)^{\on{naive}}_{\on{Ran}}$ is given by an element of the groupoid
\[
\lbrace (D_{x_I})_{\on{dR}}^{\on{fake}}\underset{X_{\on{dR}}}{\times} U_{\on{dR}} \to \on{pt}/G \rbrace.
\]

\noindent Here, we consider $\on{pt}/G$ as a prestack quotient as above. Hence for $x\in S$, the fiber of $\on{Jets}_{\on{horiz}}^{S-\on{mer}}(\bB G)^{\on{naive}}_{\on{Ran}}$ at $x$ is 
\[
\on{Maps}((\overset{\circ}{D}_x)_{\on{dR}}^{\on{fake}},\on{pt}/G).
\]

\noindent We let 
\[
\on{Jets}_{\on{horiz}}^{S-\on{mer}}(\bB G)_{\on{Ran}}
\]

\noindent be the étale sheafification of $\on{Jets}_{\on{horiz}}^{S-\on{mer}}(\bB G)^{\on{naive}}_{\on{Ran}}$. Thus, the fiber of $\on{Jets}_{\on{horiz}}^{S-\on{mer}}(\bB G)_{\on{Ran}}$ at $x\in S$ is
\[
\on{LocSys}_G(\overset{\circ}{D}_x),
\]

\noindent whereas the fiber at $x\in U$ is 
\[
\on{LocSys}_G(D_x)\simeq \bB G.
\]

\noindent We let
\[
\on{Jets}_{\on{horiz}}^{S-\on{mer}}(\bB G)_{\on{Ran},S}:=\on{Jets}_{\on{horiz}}^{S-\on{mer}}(\bB G)_{\on{Ran}}\underset{\on{Ran}_{X_{\on{dR}}}}{\times}\on{Ran}_{X_{\on{dR}},S}.
\]

\subsection{Finishing the proof of Proposition \ref{p:ff}}\label{s:constr}

\subsubsection{} Write
\[
\on{Jets}_{\on{horiz}}^{S-\on{mer}}(\bB G)_{X}:=\on{Jets}_{\on{horiz}}^{S-\on{mer}}(\bB G)_{\on{Ran}}\underset{\on{Ran}_{X_{\on{dR}}}}{\times} X_{\on{dR}}.
\]

\noindent Let
\[
\on{Sect}_{\nabla}(X,\on{Jets}_{\on{horiz}}^{S-\on{mer}}(\bB G)_{X})=\on{Maps}_{/X_{\on{dR}}}(X_{\on{dR}},\on{Jets}_{\on{horiz}}^{S-\on{mer}}(\bB G)_{X})
\]

\noindent denote its stack of horizontal sections. Note that we have a natural map
\begin{equation}\label{eq:horizsecwild}
\on{Sect}_{\nabla}(X,\on{Jets}_{\on{horiz}}^{S-\on{mer}}(\bB G)_{X})\to\on{LocSys}_G(U)
\end{equation}

\noindent given by restricting the section to $U_{\on{dR}}$.
\begin{lem}\label{l:horizsecwilda}
The map (\ref{eq:horizsecwild}) is an isomorphism.
\end{lem}

\begin{proof}

Let us provide an inverse map to (\ref{eq:horizsecwild}). Let $T\to\on{LocSys}_G(U)$ be a map with $T$ affine. We need to provide a naural map
\[
X_{\on{dR}}\times T\to \on{Jets}_{\on{horiz}}^{S-\on{mer}}(\bB G)_{X}
\]

\noindent over $X_{\on{dR}}$. Thus, let
\[
T'\to X_{\on{dR}}\times T
\]

\noindent be a map with $T'$ affine. Let $x: T'_{\on{red}}\to X$ be the induced map. By construction, we have a map
\begin{equation}\label{eq:finalcountdown}
(D_{x})_{\on{dR}}^{\on{fake}}\underset{X_{\on{dR}}}{\times} U_{\on{dR}}\to T'\times U_{\on{dR}}\to T\times U_{\on{dR}}\to\bB G,
\end{equation}

\noindent where the last map comes from $T\to \on{LocSys}_G(U)$. For (\ref{eq:finalcountdown}) to define a map $T'\to \on{Jets}_{\on{horiz}}^{S-\on{mer}}(\bB G)_{X}$, we need to show that the induced $G$-bundle on $T'\times U$ is étale-locally on $T'$ trivial when restricted to $T'\widehat{\times}\overset{\circ}{D}_S$ (here, $\overset{\circ}{D}_S=\underset{x\in S}{\coprod}\overset{\circ}{D}_x$). However, this follows from Lemma \ref{l:drinfsimp}.

\end{proof}

\subsubsection{} Next, let 
\[
\sY_x\to \on{LocSys}_G(\overset{\circ}{D}_x)
\]

\noindent be a monomorphism of prestacks for each $x\in S$. Define
\begin{equation}\label{eq:modmono}
\on{Jets}_{\on{horiz}}^{S-\on{mer}}(\bB G)_{X,\sY_S}
\end{equation}

\noindent to be the prestack whose $T$-points parametrize maps $T\to \on{Jets}_{\on{horiz}}^{S-\on{mer}}(\bB G)_{X}$ such that the induced map
\[
T\underset{X_{\on{dR}}}{\times} S \to \underset{x\in S}{\prod} \on{LocSys}_G(\overset{\circ}{D}_x)
\]

\noindent factors through
\[
\underset{x\in S}{\prod} \sY_x\to \underset{x\in S}{\prod} \on{LocSys}_G(\overset{\circ}{D}_x).
\]

\subsubsection{} We let 
\[
\on{LocSys}_G(U,\sY_S):=\on{LocSys}_G(U)\underset{\underset{x\in S}{\prod}{\on{LocSys}_G(\overset{\circ}{D}_x)}}{\times} \underset{x\in S}{\prod}\sY_x.
\]

\noindent The following is a formal consequence of Lemma \ref{l:horizsecwilda}:
\begin{lem}
The natural map
\[
\on{Sect}_{\nabla}(X,\on{Jets}_{\on{horiz}}^{S-\on{mer}}(\bB G)_{X,\sY_S})\to \on{LocSys}_G(U,\sY_S)
\]

\noindent is an isomorphism of prestacks.
\end{lem}

\subsubsection{}\label{s:monobix} In particular, taking $\sY_x=\on{LocSys}_G(\overset{\circ}{D}_x)_{\chi_x}^{\on{RS}}$,\footnote{Some justification is required to show that $\on{LocSys}_G(\overset{\circ}{D}_x)_{\chi_x}^{\on{RS}}\to \on{LocSys}_G(\overset{\circ}{D}_x)$ is a monomorphism. In fact, the map $\on{LocSys}_G(\overset{\circ}{D}_x)^{\on{restr}}\to \on{LocSys}_G(\overset{\circ}{D}_x)$ is a monomorphism as shown in \cite{lsrestr}.} we write
\[
(\on{LocSys}_G(\overset{\circ}{D})^{\on{RS}}_{\chi_S})_{X_{\on{dR}}}
\]

\noindent for the resulting prestack (\ref{eq:modmono}). Thus, we get:
\begin{lem}\label{l:sectrschi}
The natural map
\[
\on{Sect}_{\nabla}(X,(\on{LocSys}_G(\overset{\circ}{D})^{\on{RS}}_{\chi_S})_{X_{\on{dR}}})=\on{LocSys}_G(U,\chi_S)
\]

\noindent is an isomorphism.
\end{lem}

\bibliographystyle{alpha}
\bibliography{Bip}

\end{document}